\documentclass[]{myarticle}   

\usepackage{graphicx}
\usepackage{latexsym,amsfonts,amsmath,theorem,amssymb}
\usepackage{array,tabularx,bbm}
\usepackage{booktabs,booktabs,threeparttable,boxedminipage}
\usepackage{orcidlink}
\usepackage{siunitx}
\usepackage{pstricks,epstopdf}
\usepackage{color}
\usepackage{caption}
\usepackage{framed}
\usepackage{blindtext}
\usepackage{float}
\usepackage{multirow}
\usepackage{lineno}
\usepackage[caption=false]{subfig}
\usepackage{algorithm,algpseudocode}
\usepackage{xcolor}
\usepackage[textsize=footnotesize]{todonotes}
\usepackage{mathrsfs} 
\usepackage{geometry}
\usepackage{arydshln}
\usepackage{titling}
\pretitle{\begin{center}\bfseries\LARGE}  
\posttitle{\end{center}}
\predate{}
\postdate{}
\date{}

\makeatletter
\let\old@footnotemark\@footnotemark
\renewcommand{\@footnotemark}{%
  \textcolor{black}{\old@makefnmark}%
}
\let\old@makefnmark\@makefnmark
\renewcommand{\@makefnmark}{\textcolor{black}{\old@makefnmark}}
\makeatother

\algrenewcommand\algorithmicrequire{\textbf{Input:}}
\algrenewcommand\algorithmicensure{\textbf{Output:}}

\newcommand\argmin{\mathop{\mathrm{argmin}}}

\newcommand\calG{\mathcal{G}}
\newcommand\calR{\mathcal{R}}

\definecolor{richPurple}{RGB}{160, 0, 200}
\definecolor{deepGreen}{RGB}{0, 160, 60}

\title{Data-informed posterior approximation for Bayesian linear inverse problems}
\author{
Haibo Li 
\thanks{School of Mathematics and Statistics \& Hubei Key Laboratory of Engineering Modeling and Scientific Computing, Huazhong University of Science and Technology, Wuhan 430074, China.
\href{mailto:haiboli@hust.edu.cn}{haiboli@hust.edu.cn}, \href{mailto:haibolee1729@gmail.com}{haibolee1729@gmail.com}}
}

\begin{document}

\maketitle

\begin{abstract}
  Computing posterior distributions in large-scale Bayesian linear inverse problems is challenging due to the high dimensionality of the parameter space. In this work, we develop a data-informed framework that shifts the computational focus from the parameter space to the data space. We rigorously characterize an intrinsically low-dimensional data space, establish its isometric embedding into the parameter space, and show that the prior-to-posterior update is confined to a data-informed subspace. This perspective allows posterior inference to be carried out in a reduced data-informed subspace. Based on this formulation, we propose a quotient-space Golub--Kahan bidiagonalization method to construct data-informed Krylov subspaces, and integrate empirical Bayesian inference into the iterative framework, enabling simultaneous hyperparameter estimation and posterior approximation in a matrix-free manner. Numerical experiments on representative problems support the theoretical framework and demonstrate the effectiveness of the resulting method.
\end{abstract}


\tableofcontents

\section{Introduction} 
Inverse problems arise across a diverse spectrum of scientific and engineering disciplines, including image reconstruction, computed tomography, geoscience, data assimilation and so on \cite{Hansen2006,Kaip2006,Buzug2008,Law2015,Richter2016}. The goal is to recover an unknown parameter or function from indirect measurements contaminated by noise. For linear inverse problems, upon discretization of the underlying continuous forward operator, these problems are typically formulated as a high-dimensional linear system:
\begin{equation}\label{inverse1}
  y = Gx + \eta,   \quad  \eta \sim \mathcal{N}(0,\,\Gamma),
\end{equation}
where $x \in \mathbb{R}^n$ represents the discretized unknown, $y \in \mathbb{R}^m$ is the observation data, and $G \in \mathbb{R}^{m \times n}$ is the forward matrix. The noise $\eta$ is assumed to follow a Gaussian distribution with a symmetric positive definite covariance $\Gamma \in \mathbb{R}^{m \times m}$. A fundamental challenge in solving such systems is the ill-posed nature of the inverse problem, which means that there may be multiple solutions that fit the observation data equally well, or the solution is very sensitive with respect to the observation noise. 

To address the ill-posedness, two predominant frameworks are widely employed: deterministic Tikhonov regularization approach and probabilistic Bayesian inference approach. The former is based on minimizing a loss function $\ell(Gx, y)$ added by a regularizer $R(x)$ that incorporates prior information about $x$ \cite{Tikhonov1977,Engl2000,hansen2010discrete}. In contrast, the Bayesian framework treats $x$ and $y$ as random variables, providing a complete characterization of the solution through the posterior distribution $x\,|\,y$ via Bayes' theorem \cite{Kaip2006,Stuart2010}. This approach has gained significant traction due to its ability to provide a rigorous quantification of the underlying uncertainty. The two methodologies are connected through the maximum a posteriori (MAP) estimator, where the mean of the posterior distribution corresponds to the solution of a specific Tikhonov regularization problem \cite{Kaip2006,dashti2013map,helin2015maximum}. 

To reconstruct the unknown, assume that $x$ has a Gaussian prior $\mathcal{N}(0,\lambda^{-1}\Sigma)$ independent of the noise $\eta$, where $\Sigma$ is a symmetric positive definite matrix, and $\lambda>0$ is a hyperparameter governing the prior variance level. In some applications, $\lambda$ may be known. Quite often, however, it is unknown and we assume that is the case here. According to Bayes' rule, the prior density and the likelihood function induced by the forward model  together yield the posterior probability density:
\begin{equation}
	\pi(x\,|\,y) \propto \pi(x)\,\pi(y\,|\,x)
	\propto \exp\left(-\frac{1}{2}\|Gx-y\|_{\Gamma^{-1}}^{2}-\frac{\lambda}{2}\|x\|_{\Sigma^{-1}}^2\right),
\end{equation}
where  $\|x\|_{B}:=(x^{\top}Bx)^{1/2}$ is the $B$-norm of $x$ for a positive definite matrix $B$. Since both the likelihood and the prior are Gaussian, the posterior distribution is also Gaussian, $x\,|\,y \sim \mathcal{N}(x_\lambda, C_\lambda)$, with its mean and covariance given by:
\begin{equation}\label{post_distr}
  x_{\lambda} = C_{\lambda}G^{\top}\Gamma^{-1}y, \qquad
  C_{\lambda} = (G^{\top}\Gamma^{-1}G+\lambda\Sigma^{-1})^{-1}.
\end{equation}
Specifically, the MAP estimator is the posterior mean, which can be obtained by minimizing the negative log-density of the Bayesian posterior:
\begin{equation}\label{Bayes1}
	x_{\lambda} = \argmin_{x \in \mathbb{R}^{n}}\{\|Gx-y\|_{\Gamma^{-1}}^{2} + \lambda\|x\|_{\Sigma^{-1}}^2\},
\end{equation}
It is evident from \cref{Bayes1} that the Tikhonov penalty term is explicitly determined by the prior covariance structure, bridging the probabilistic and deterministic perspectives of the inverse problem.

In practical large-scale applications, the parameter dimension $n$ is often extremely high, typically with $n \gg m$, as considered in this paper. Combined with the inherent ill-posedness of the inverse problem, this poses significant challenges for both stability and computational tractability. In particular, a major computational bottleneck is the inversion of the matrix $G^\top\Gamma^{-1}G+\lambda\Sigma^{-1}$ to obtain the posterior mean and covariance. A variety of approaches have been proposed for dimension reduction of Bayesian inverse problems. Early work uses the Karhunen--Lo\`eve expansion \cite{karhunen1947lineare,Loeve1978} to represent the prior distribution in terms of its dominant eigenmodes and truncates the expansion to achieve dimension reduction \cite{li2006efficient,marzouk2009dimensionality}. However, this strategy relies entirely on the spectral properties of the prior and does not account for the influence of the forward operator or the observation data, which restricts the choice of prior models and may lead to significant truncation errors. An alternative approach \cite{lieberman2010parameter} constructs reduced subspaces using model-constrained sampling and greedy strategy, but it may fail to capture posterior variability in directions that are not informed by the data. Another approach \cite{flath2011fast} approximates the posterior covariance via a low-rank approximation of the Hessian of the log-likelihood. All of the above approaches aim at dimension reduction in the entire parameter space, but do not fully exploit the interaction between the unknown parameter and observation data.

In the common setting of inverse problems, the prior distribution encodes smoothness of the unknown parameter, and the observation data are linked to the parameter through a smoothing forward operator. As a consequence, the prior-to-posterior update is often confined to a relatively low-dimensional subspace of the parameter space. Exploiting this structure can enable efficient low-rank approximations of the posterior mean and covariance, leading to substantial computational savings \cite{martin2012stochastic,bui2012extreme,bui2013computational}. In \cite{cui2014likelihood,spantini2015optimal}, the terminology \emph{likelihood-informed subspace} is proposed to identify such low-dimensional subspaces, which contain the directions along which the posterior differs significantly from the prior. All these methods approximate the posterior covariance through low-rank corrections of the prior covariance, which can provide a more balanced treatment of prior and data information. However, these approaches are still formulated in the parameter space and often rely on explicit factorizations of the prior covariance or its inverse, which can be computationally prohibitive in large-scale settings.

In this paper, we shift the focus from the parameter space to the data space. Endowed with the prior-induced geometry, the parameter space is given by the finite-dimensional Hilbert space $(\mathbb{R}^{n},\,\langle \cdot,\cdot \rangle_{\Sigma^{-1}})$, which can be extremely high-dimensional if $n$ is very large. In contrast, we introduce the \emph{data space} as the quotient Hilbert space $(\mathbb R^m/\mathcal N(M),\,\langle\cdot,\cdot\rangle_M)$, with its precise meaning and motivation detailed in \Cref{sec2}, where $M = G\Sigma G^\top$ is symmetric positive semidefinite and may be rank deficient with $\mathcal{N}(M)$ denoting its null space. The dimension of this space is $\mathrm{rank}(M)=\mathrm{rank}(G)$, and the smoothing nature of $G$ typically leads to rapidly decaying singular values and hence a much smaller effective dimension. We show that the data space admits an isometric embedding into the parameter space, and that the prior-to-posterior update occurs entirely within the image of this embedding, which has the same  dimension as the data space. This structure reveals that the informative directions of the posterior are entirely characterized by a low-dimensional, data-informed subspace that is independent of the ambient parameter dimension. While the principle that the prior-to-posterior change may concentrate on a low-dimensional subspace is well recognized in the literature, the present work provides a rigorous and explicit mathematical characterization of this phenomenon through the quotient-space geometry of the data space.

The data-space formulation offers a fundamentally different computational paradigm. In the parameter space, approximating the posterior mean and covariance requires dealing with $G^\top \Gamma^{-1} G$ as well as manipulating the prior covariance $\Sigma$, which is often large, dense, and ill-conditioned when defined through covariance kernels \cite{Williams2006gaussian,Roininen2014whittle}. Such operations typically demand matrix factorizations or inversions that are computationally prohibitive in high dimensions. In contrast, the data-space formulation allows the problem to be recast in a low-dimensional space that captures dominant data-informed variations, while eliminating directions that are not identifiable from the data. As a consequence, both the computational complexity and the representation of the posterior are governed by the intrinsic low dimension of the data. To obtain efficient approximations of the posterior, we propose a new Krylov subspace method \cite{liesen2013krylov} to iteratively capture dominant data-informed subspaces, as detailed in \Cref{sec3}. This leads naturally to matrix-free algorithms whose computational cost scales with the dimension of the data rather than the ambient parameter dimension. In addition, we show that the likelihood-informed subspace introduced in \cite{spantini2015optimal} arises naturally from the data-space formulation and admits a geometric interpretation within this framework, providing a structural understanding of data-informed subspaces and their connection to the proposed Krylov method.

An additional challenge in Bayesian inverse problems is the selection of the hyperparameter $\lambda$. In much of the existing literature, $\lambda$ is assumed to be known. In practice, however, it is typically determined either by repeated posterior evaluations over multiple candidates or by maximizing the marginal likelihood within an empirical Bayesian inference framework \cite{carlin1997bayes,petrone2014bayes,rousseau2017asymptotic}. Both approaches require repeated large-scale computations---either solving ill-posed linear systems or evaluating the marginal likelihood in the high-dimensional parameter space. From the perspective of Tikhonov regularization, the hyperparameter $\lambda$ can be iteratively estimated within the hybrid regularization framework \cite{Chung2017,saibaba2020efficient,li2024preconditioned,chung2024computational}, where the original problem is projected onto a sequence of low-dimensional Krylov subspaces and the parameter is determined from the reduced problems using parameter-choice rules such as the discrepancy principle or generalized cross-validation \cite{Kilmer2001,Chungnagy2008}. While these approaches significantly reduce computational cost, their convergence behavior remains not fully understood, and they may fail to produce reliable estimates of $\lambda$ in practice \cite{Chungnagy2008,Renaut2017,li2025projected}. These limitations motivate the development of scalable methods that enable stable and efficient hyperparameter estimation while simultaneously approximating the posterior, without requiring repeated computations in the parameter space.

Within the data-space formulation, the data-informed Krylov subspaces are constructed in the quotient space $(\mathbb R^m/\mathcal N(M),\langle\cdot,\cdot\rangle_M)$. To apply the Krylov subspace method and handle the degeneracy of $M$, one may work in the range space $(\mathcal{R}(M), \langle\cdot,\cdot \rangle_{M})$ that is isometrically isomorphic to the quotient space, where $\mathcal{R}(M)$ denotes the range of $M$. However, this approach requires computing the orthogonal projector $MM^{\dagger}$, which is computationally expensive for large $m$. To overcome this difficulty, we develop a Krylov subspace framework directly in the quotient space, without requiring the computation of projections onto the range of $M$. Specifically, we propose a quotient-space formulation of the Golub--Kahan bidiagonalization process, referred to as Q-GKB, which constructs a sequence of data-informed Krylov subspaces of increasing dimension in a fully matrix-free manner. Building on this construction, we integrate empirical Bayesian inference into the Q-GKB iteration, enabling simultaneous hyperparameter estimation as well as posterior approximation by restricting the prior-to-posterior update to the data-informed Krylov subspaces. 

The contributions of this paper can be summarized as follows:
\begin{itemize}
    \item We establish a data-space geometric framework for Bayesian linear inverse problems based on the quotient space induced by the semidefinite matrix $M = G \Sigma G^\top$. This yields an isometric correspondence between data-space and parameter-space representations and provides a rigorous characterization of the low-dimensional nature of posterior update.
    \item We develop a quotient-space Krylov method via a novel formulation of the Golub--Kahan bidiagonalization process, which constructs a sequence of data-informed Krylov subspaces iteratively in a fully matrix-free manner.
    \item We integrate empirical Bayesian inference into the proposed framework, enabling simultaneous hyperparameter estimation and posterior approximation within low-dimensional Krylov subspaces. We propose an efficient algorithm for computing low-rank posterior approximation with theoretical accuracy guarantees.
\end{itemize}
To assess the proposed theory and the performance of the method, we present several representative numerical examples of increasing scale and complexity. In particular, the large-scale examples highlight the effectiveness of the matrix-free method within the data-space formulation, where direct access to the posterior distribution is computationally prohibitive.

The remainder of this paper is organized as follows. In \Cref{sec2}, we develop the data-space geometric framework and establish a structural correspondence between the parameter and data spaces. In \Cref{sec3}, we present the Q-GKB method for constructing data-informed Krylov subspaces and investigate its properties. In \Cref{sec4}, we propose a Q-GKB-based empirical Bayesian method that enables simultaneous estimation of the hyperparameter $\lambda$ and efficient low-rank posterior approximation in a matrix-free manner, and analyze its approximation accuracy. Numerical experiments are presented in \Cref{sec5}, and the conclusion is provided in \Cref{sec6}.

\section{Data-informed structure of the parameter space} \label{sec2}
We investigate the properties of the Bayesian posterior and the geometry of the parameter space, which motivate the introduction of a quotient-space structure for the data space. In the remainder of the paper, we use $\mathrm{Im}(\cdot)$ to denote the image of a linear operator, and $I$ to denote the identity matrix or operator, with its order clear from the context. The symbol $0$ is used to denote either the real number zero or the zero vector of a linear space, with its intended meaning being clear in each instance.

The following result gives equivalent representations of the posterior mean and covariance. We remark that the same expressions appear in the literature, e.g., \cite{Kaip2006,Stuart2010}. For completeness, we 
provide the full proof.

\begin{proposition}\label{prop:post1}
  The posterior mean and covariance admit the equivalent representations
  \begin{align}
    x_{\lambda} &= \Sigma G^{\top}(G\Sigma G^{\top}+\lambda\Gamma)^{-1}y , \label{post_mean1} \\
    C_{\lambda} &= \lambda^{-1}\Sigma - \lambda^{-2}\Sigma G^{\top}(\Gamma+\lambda^{-1}G\Sigma G^{\top})^{-1}G\Sigma . \label{post_cov1}
  \end{align}
\end{proposition}
\begin{proof}
    The covariance formula can be derived by applying the Woodbury identity, which is
    \[(A+UCV)^{-1}=A^{-1}-A^{-1}U(C^{-1}+VA^{-1}U)^{-1}VA^{-1}, \]
    for any nonsingular $A\in\mathbb{R}^{n\times n}$, $C\in\mathbb{R}^{m\times m}$, and $U\in\mathbb{R}^{n\times m}$, $V\in\mathbb{R}^{m\times n}$. Therefore,
    \begin{align*}
      C_\lambda
      &= (\lambda\Sigma^{-1}+G^\top\Gamma^{-1}G)^{-1} 
      = \lambda^{-1}\Sigma-\lambda^{-1}\Sigma G^\top(\Gamma+\lambda^{-1}G\Sigma G^\top)^{-1}G\lambda^{-1}\Sigma \\
      &= \lambda^{-1}\Sigma - \lambda^{-2}\Sigma G^\top (\Gamma+\lambda^{-1}G\Sigma G^\top)^{-1}G\Sigma,
    \end{align*}
    which proves the covariance formula. Substituting the above representation into $x_\lambda=C_\lambda G^\top\Gamma^{-1}y$, we have
  \begin{align*}
    x_\lambda 
    &= \left[\lambda^{-1}\Sigma - \lambda^{-2}\Sigma G^\top(\Gamma+\lambda^{-1}G\Sigma G^\top)^{-1}G\Sigma \right] G^\top\Gamma^{-1}y \\
    &= \lambda^{-1}\Sigma G^\top\Gamma^{-1}y - \lambda^{-2}\Sigma G^\top (\Gamma+\lambda^{-1}G\Sigma G^\top)^{-1} G\Sigma G^\top\Gamma^{-1}y.
  \end{align*}
  Let $M=G\Sigma G^\top \in \mathbb R^{m\times m}$. It holds that
  \begin{align*}\label{eq:xlambda-proof-mid}
    x_\lambda 
    &= \Sigma G^\top \left[\lambda^{-1}\Gamma^{-1}-\lambda^{-2}(\Gamma+\lambda^{-1}M)^{-1}M\Gamma^{-1}\right]y \\
    &= \lambda^{-1}\Sigma G^\top\left[I-(\lambda\Gamma+M)^{-1}M\right]\Gamma^{-1}y \\
    &= \lambda^{-1}\Sigma G^\top\bigl[\lambda(\lambda\Gamma+M)^{-1}\Gamma\bigr]\Gamma^{-1}y 
    = \Sigma G^\top(\lambda\Gamma+M)^{-1}y,
  \end{align*}
  where we have used the identity $I-(\lambda\Gamma+M)^{-1}M = \lambda(\lambda\Gamma+M)^{-1}\Gamma.$ This proves the formula for the posterior mean.
\end{proof}

From \cref{post_mean1} and \cref{post_cov1}, it is natural to introduce the positive semidefinite matrix 
\begin{equation}\label{def_M}
  M = G\Sigma G^\top \in \mathbb R^{m\times m}.
\end{equation}
Define the data-space variable $z_\lambda =(M+\lambda\Gamma)^{-1}y\in\mathbb{R}^{m}$. Then the posterior mean can be written as $x_\lambda = \Sigma G^\top z_\lambda$. Consider the regularization problem
\begin{equation}\label{regu_data}
  \min_{z\in\mathbb{R}^{m}}\{\|Mz-y\|_{\Gamma^{-1}}^2 + \lambda z^{\top}Mz \}.
\end{equation}
If $M$ is not positive definite, this problem may admit multiple solutions. However, it can be shown that any solution coincides with $z_\lambda$ up to an element of $\mathcal{N}(M)$.

\begin{proposition}\label{prop:zlambda-variational}
Consider the regularization problem \cref{regu_data} with \(\lambda>0\). Then 
\begin{enumerate}
  \item[(a)] \(z_\lambda\) is a minimizer of \cref{regu_data}, and the set of all minimizers is $z_\lambda+\mathcal N(M)$;
  \item[(b)] if \(z\in z_\lambda+\mathcal N(M)\), it holds that $\Sigma G^\top z = x_\lambda$. Hence, all minimizers of \cref{regu_data} yield the same posterior mean.
\end{enumerate}
\end{proposition}
\begin{proof}
(a) Define $J(z)=\|Mz-y\|_{\Gamma^{-1}}^2+\lambda z^\top M z.$ Then \(J\) is a convex quadratic functional on \(\mathbb R^m\). Hence its minimizers are exactly the solutions of 
\[\nabla J(z) = 2M \Gamma^{-1}(Mz-y)+2\lambda Mz = 0.\]
Since $z_{\lambda}$ satisfies $(M+\lambda\Gamma)z_\lambda=y$, or equivalently $Mz_\lambda-y=-\lambda\Gamma z_\lambda$, it can be verified that $\nabla J(z_\lambda)=0$. Next let \(z\) be vector such that $\nabla J(z)=0$. Using \(y=(M+\lambda\Gamma)z_\lambda\), we have
\[Mz-y = Mz-(M+\lambda\Gamma)z_\lambda= M(z-z_\lambda)-\lambda\Gamma z_\lambda.\]
Substituting it into $\nabla J(z)$ yields
\begin{align*}
\frac{1}{2}\nabla J(z)=M\Gamma^{-1}\bigl(M(z-z_\lambda)-\lambda\Gamma z_\lambda\bigr)+\lambda Mz 
=M\Gamma^{-1}M(z-z_\lambda)+\lambda M(z-z_\lambda).
\end{align*}
Taking the inner product with \(z-z_\lambda\), we obtain
\[\nabla J(z)= 0 \quad \Rightarrow \quad (z-z_\lambda)^\top M\Gamma^{-1}M(z-z_\lambda)
+ \lambda (z-z_\lambda)^\top M(z-z_\lambda)=0.\]
Since $\Gamma^{-1}$ is positive definite and $M$ is positive semidefinite, both terms in the left-hand-side are nonnegative and must therefore vanish. In particular, $(z - z_\lambda)^\top M (z - z_\lambda) = 0$, which implies that $z - z_\lambda \in \mathcal{N}(M)$. Conversely, for any \(p\in\mathcal N(M)\), we have $\frac{1}{2}\nabla J(z_\lambda+p)=M\Gamma^{-1}Mp+\lambda Mp=0$. This proves (a).

(b) For any \(v\in\mathcal N(M)\) we have $0=v^\top Mv = (G^\top v)^\top \Sigma (G^\top v)$.
This implies $G^\top v=0$, and thereby $\mathcal N(M)\subseteq \mathcal N(G^\top)$. Conversely, it is obvious that $G^\top v=0$ implies $G\Sigma G^\top v=0$, and hence $v\in \mathcal N(M)$. Therefore, it holds $\mathcal N(M)=\mathcal N(G^\top)$. Now for any $p\in\mathcal N(M)$, we have 
\[\Sigma G^\top(z_\lambda+p)
=\Sigma G^\top z_\lambda+\Sigma G^\top p = \Sigma G^\top z_\lambda=x_{\lambda}.\]
This completes the proof.
\end{proof}

By \Cref{prop:zlambda-variational}, the posterior mean lies in the range space $\mathcal{R}(\Sigma G^\top)$. 
From \cref{post_cov1}, define 
\[
\Delta=\lambda^{-1}\Sigma-C_{\lambda}
=\lambda^{-2}\Sigma G^{\top}(\Gamma+\lambda^{-1}M)^{-1}G\Sigma
\]
as the prior-to-posterior covariance update. 
Then $\mathcal{R}(\Delta)\subset\mathcal{R}(\Sigma G^\top)$ and 
$\mathcal{R}(\Sigma G^\top)^{\perp}\subset\mathcal{N}(\Delta)$. 
Hence, the covariance update acts only on $\mathcal R(\Sigma G^\top)$ and vanishes on its orthogonal complement, reflecting the principle that the data update the prior only along a subspace of dimension $\mathrm{rank}(G)$. In the parameter space, this phenomenon is described by the terminology \emph{likelihood-informed subspace} (LIS) introduced in \cite{cui2014likelihood,spantini2015optimal}. Let $H=G^{\top}\Gamma^{-1}G$, and the generalized eigenpairs of $(H,\,\Sigma^{-1})$ are $\{(\mu_{i},w_i)\}_{i=1}^{l}$ satisfying $\mu_{1}\ge \mu_{2}\ge \cdots \ge \mu_{l}>0$, $\mu_{i}=0$ for $i>l$, and $\{w_i\}_{i=1}^{l}$ are mutually $\Sigma^{-1}$-orthonormal. It is shown in \cite{spantini2015optimal} that the directions along which the posterior changes from the prior are entirely in $\mathrm{span}\{w_{i}\}_{i=1}^{l}$. For any $1\le r \le l$, the subspace $\mathrm{span}\{w_{i}\}_{i=1}^{r}$ is referred to as the $r$-dimensional LIS, which can be used to construct low-rank approximation of the posterior.

Motivated by the above structure, it is natural to consider the action of the mapping $\Sigma G^\top$ on $\mathbb{R}^m$ and identify vectors that differ by elements of $\mathcal{N}(M)$. Accordingly, we define the relation
\[ u \sim v \quad \Longleftrightarrow \quad u-v \in \mathcal N(M), \]
which is an equivalence relation on $\mathbb R^m$. Denote by $[u]$ the equivalence class of $u$, which is an element of the quotient linear space $\mathbb R^m/\mathcal N(M)$ with dimension
\[\mathrm{dim}(\mathbb R^m/\mathcal N(M)) = m - \mathrm{dim}(\mathcal N(M)) 
= \mathrm{rank}(M) = \mathrm{rank}(G).\]
We refer to the finite-dimensional Hilbert space 
\begin{equation}\label{eq:quotient-inner}
  (\mathbb{R}^m / \mathcal{N}(M), \langle \cdot, \cdot \rangle_M), \quad
   \langle [u],[v]\rangle_M := u^\top M v
\end{equation}
as the \emph{data space} of the underlying Bayesian linear inverse problem. The bilinear form $\langle \cdot, \cdot\rangle_{M}$ is well defined on $\mathbb R^m/\mathcal N(M)$, since if $u-u'\in \mathcal N(M)$ and $v-v'\in \mathcal N(M)$, then $u^\top M v - {u'}^\top M v' =(u-u')^\top M v + {u'}^\top M (v-v') = 0$. Moreover, if $\langle [u],[u]\rangle_M=0$, then $u\in\mathcal{N}(M)$, implying that $[u]=0$. Therefore, $\langle \cdot, \cdot\rangle_{M}$ defines an inner product on $\mathbb R^m/\mathcal N(M)$. Notice that $\Gamma+\lambda^{-1}M\in\mathbb{R}^{m\times m}$ is the covariance of the marginal distribution of $y$. This suggests considering the generalized eigenvalue decomposition of $(M,\,\Gamma)$, which characterizes the associated spectral structure of the data space. Since $\Gamma$ is positive definite, there exists a set of $M$-orthonormal generalized eigenvectors corresponding to the nonzero generalized eigenvalues of $(M,\Gamma)$. We refer to such eigenpairs as the nonzero $M$-orthonormal generalized eigenpairs.

The following theorem establishes a precise geometric and spectral correspondence between the data space and parameter space. 

\begin{theorem}[Correspondence between data and parameter spaces]\label{thm:isom}
  Define the linear operator 
  \begin{align}\label{isom}
    \begin{split}
      \iota: (\mathbb R^m/\mathcal N(M),\langle\cdot,\cdot\rangle_M)
          &\longrightarrow (\mathbb R^n,\langle\cdot,\cdot\rangle_{\Sigma^{-1}}) \\
      [u] &\longmapsto \Sigma G^{\top}u.
    \end{split}
  \end{align} 
  Then the following statements hold.
\begin{enumerate}
\item[(a)] (Isometric embedding)
The map $\iota$ is an isometric embedding, i.e.,
\[ \langle \iota([u]),\iota([v])\rangle_{\Sigma^{-1}}
=\langle [u],[v]\rangle_M,
\]
and $\iota$ is injective.
\item[(b)] (Parameter space decomposition)
It holds that 
\begin{equation*}
  (\mathbb R^n,\langle\cdot,\cdot\rangle_{\Sigma^{-1}})
  = \mathrm{Im}(\iota) \oplus \mathcal{N}(G) ,
\end{equation*}
where the two subspaces are $\Sigma^{-1}$-orthogonal, and $x_{\lambda}\in \mathrm{Im}(\iota)=\mathcal{R}(\Sigma G^\top)$.
\item[(c)] (Spectral correspondence)
Let $\{(\hat\mu_i,\hat w_i)\}_{i=1}^{\hat{l}}$ be all nonzero $M$-orthonormal generalized eigenpairs of $(M,\,\Gamma)$. Then $l=\hat{l}=\mathrm{rank}(G)$, and $\{\hat\mu_i,\iota([\hat w_i])\}_{i=1}^{l}$ are all nonzero $\Sigma^{-1}$-orthonormal eigenpairs of $(H,\,\Sigma^{-1})$.
\end{enumerate}
\end{theorem}
\begin{proof}
  (a) If $[u] = [u']$, then $u - u' \in \mathcal{N}(M)$. From part~(a) of \Cref{prop:zlambda-variational}, $\mathcal{N}(M) = \mathcal{N}(G^\top)$, so $G^\top(u - u') = 0$ and thus $\Sigma G^\top u = \Sigma G^\top u'$. Hence $\iota$ is well defined. For the injectivity, suppose $\iota([u]) = \Sigma G^\top u = 0$. Then $u^\top M u = \|\Sigma G^\top u\|_{\Sigma^{-1}}^2 = 0$, so $u \in \mathcal{N}(M)$, i.e., $[u] = 0$.  For any $[u], [v] \in \mathbb{R}^m/\mathcal{N}(M)$, it holds
  \[\langle \iota([u]),\, \iota([v]) \rangle_{\Sigma^{-1}} = (\Sigma G^\top u)^\top \Sigma^{-1} (\Sigma G^\top v)
  = u^\top G \Sigma G^\top v = u^\top M v = \langle [u],[v]\rangle_M.\] 
  Thus $\iota$ is an isometric embedding.

  (b) For any $[u]\in\mathbb R^m/\mathcal N(M)$ and $v\in\mathcal N(G)$, it holds 
  \begin{equation*}
    \langle \iota([u]),v\rangle_{\Sigma^{-1}}= (\Sigma G^{\top}u)^{\top}\Sigma^{-1} v 
    = u^{\top}Gv = 0.
  \end{equation*}
  Hence $\mathrm{Im}(\iota)$ and $\mathcal{N}(G)$ are $\Sigma^{-1}$ orthogonal. Since $\iota$ is injective, we have
  \[\mathrm{dim}(\mathrm{Im}(\iota))=\mathrm{dim}(\mathbb R^m/\mathcal N(M))
  =\mathrm{rank}(M)=\mathrm{rank}(G)=\mathrm{dim}(\mathcal{R}(\Sigma G^{\top})). \] 
  Therefore, it holds that $\mathrm{Im}(\iota)=\mathcal{R}(\Sigma G^\top)$ and $\mathrm{dim}(\mathrm{Im}(\iota))+\mathrm{dim}(\mathcal{N}(G))=n$. This also proves $(\mathbb R^n,\langle\cdot,\cdot\rangle_{\Sigma^{-1}})
  = \mathrm{Im}(\iota) \oplus \mathcal{N}(G)$. It is obvious that $x_{\lambda}\in\mathrm{Im}(\iota)$ since $x_{\lambda}=\Sigma G^{\top}z_{\lambda}$.

  (c) Since $\Sigma$ and $\Gamma$ are symmetric positive definite, we have
  \[\mathrm{rank}(M)=\mathrm{rank}(G\Sigma G^\top)=\mathrm{rank}(G), \quad
  \mathrm{rank}(H)=\mathrm{rank}(G^\top\Gamma^{-1}G)=\mathrm{rank}(G).\]
  Hence the pencils $(M,\,\Gamma)$ and $(H,\,\Sigma^{-1})$ have the same number of nonzero generalized eigenvalues, equals to $\mathrm{rank}(G)$. Now let $(\hat\mu,\hat w)$ be a nonzero generalized eigenpair of $(M,\Gamma)$, i.e., $M\hat w=\hat\mu\,\Gamma \hat w$ with $\hat\mu>0$. Since $\hat\mu\neq 0$, we must have $\hat w\notin\mathcal N(M)$, and thus $[\hat w]\neq 0$ in $\mathbb R^m/\mathcal N(M)$. Let $w=\iota([\hat w])=\Sigma G^\top \hat w$. Because $\iota$ is injective, we have $w\neq 0$.
  Using $M=G\Sigma G^\top$, we have
  \[ Hw = G^\top\Gamma^{-1}G\,(\Sigma G^\top \hat w)
  = G^\top\Gamma^{-1}M\hat w = \hat\mu\, G^\top \hat w. \]
  Combining with $\Sigma^{-1}w=G^\top\hat w$, we obtain $Hw=\hat\mu\,\Sigma^{-1}w$. This shows that $(\hat\mu,w)$ is a generalized eigenpair of $(H,\,\Sigma^{-1})$. Conversely, let $(\mu,w)$ be a generalized eigenpair of $(H,\,\Sigma^{-1})$ with $\mu>0$, i.e., $Hw=\mu\,\Sigma^{-1}w$. Let $u=\mu^{-1}\Gamma^{-1}G w\in\mathbb R^m$. Then 
  \[\iota([u])=\Sigma G^\top u = \mu^{-1}\Sigma G^\top\Gamma^{-1}Gw
  =\mu^{-1}\Sigma H w =w,\]
  and 
  \[Mu = G\Sigma G^\top u = Gw = \mu\Gamma(\mu^{-1}\Gamma^{-1}Gw) = \mu\Gamma u,
  \]
  so $(\mu,u)$ is a nonzero generalized eigenpair of $(M,\,\Gamma)$ satisfying $w=\iota([u])$. We have therefore established a one-to-one correspondence between the nonzero generalized eigenpairs of $(M,\,\Gamma)$ and $(H,\,\Sigma^{-1})$. Moreover, by part (a), $\iota$ maps $M$-orthonormal vectors to $\Sigma^{-1}$-orthonormal vectors. This completes the proof.
\end{proof}

\begin{remark}
Theorem~\ref{thm:isom} shows that the data-space formulation provides a complete geometric representation of the data-informed structure of the posterior. In particular, the image $\mathrm{Im}(\iota)$ identifies the data-informed subspace in the parameter space, along which the posterior differs from the prior, while its $\Sigma^{-1}$-orthogonal complement $\mathcal N(G)$ corresponds to directions that are not informed by the data. Furthermore, the spectral correspondence shows that the eigenstructure in the data space induces the corresponding LIS in the parameter space, thereby providing a geometric characterization of the LIS in terms of the quotient data space.
\end{remark}

In the parameter space, a common approach to approximate the posterior covariance $C_{\lambda}$ is based on LIS. We use the Loewner order ``$\preceq$'' for symmetric matrices, defined by $A \preceq B$ if $B-A$ is symmetric positive definite (SPD) \cite{horn2012matrix}. It is shown in \cite{spantini2015optimal} that the optimal approximation to $C_{\lambda}$ within the set
\[\mathcal{M}_r = \{\lambda^{-1}\Sigma-KK^{\top} \succ 0 : \mathrm{rank}(K)\leq r\}\]
with respect to the F{\"o}rstner distance is achieved by choosing $KK^{\top}=\sum_{i=1}^{r}\frac{\mu_i}{\lambda(\lambda+\mu_i)}\, w_{i}w_{i}^{\top}$, where $\{(\mu_i,w_i)\}_{i=1}^{r}$ are the leading $r$ generalized eigenpairs of $(H,\,\Sigma^{-1})$. The resulting approximation can be written as
\begin{equation}\label{cov_LIS}
  C_{\lambda}^{(r)} = (H_{r}+\lambda\Sigma^{-1})^{-1},
\end{equation}
where
\[H_{r}=\Sigma^{-1}\left(\sum_{i=1}^{r}\mu_{i}w_{i}w_{i}^{\top}\right)\Sigma^{-1}\]
is a rank-$r$ approximation of $H$ in the $r$-dimensional LIS.

By \Cref{thm:isom}, the data-space geometry induced by $M$ is often intrinsically low-dimensional due to the smoothing nature of the forward operator $G$. This provides a geometric justification for effective approximation of the posterior structure by identifying dominant subspaces of the data space. In particular, since the nonzero spectral structure of $(H,\,\Sigma^{-1})$ is completely determined by that of $(M,\,\Gamma)$, it is natural to approximate $H$ via the dominant generalized eigenpairs of $(M,\,\Gamma)$. Accordingly, we equivalently can write
\begin{equation}\label{Hk_1}
  H_{r}
  = \Sigma^{-1}\!\left(\Sigma G^{\top}
  \Big(\sum_{i=1}^{r}\mu_{i}\hat{w}_{i}\hat{w}_{i}^{\top}\Big)
  G\Sigma\right)\!\Sigma^{-1}
  \;=\;
  G^{\top}\!\left(\sum_{i=1}^{r}\mu_{i}\hat{w}_{i}\hat{w}_{i}^{\top}\right)\!G,
\end{equation}
where $\{(\mu_i,\hat{w}_i)\}_{i=1}^{r}$ are the leading $r$ generalized eigenpairs of $(M,\,\Gamma)$. This construction provides a data-space geometric interpretation of LIS-based posterior approximation.

\section{Quotient-space Krylov method in the data space} \label{sec3}
For large-scale problems, a direct eigenvalue decomposition of $(M,\,\Gamma)$ is computationally expensive, and forming $M$ explicitly is typically infeasible. Nevertheless, matrix-vector products can be computed efficiently as $Mv = G(\Sigma(G^\top v))$. Moreover, the hyperparameter $\lambda$ is often unknown in advance and its estimation in the full parameter space is costly. In view of the connection between Bayesian inference and Tikhonov regularization, it is natural to work with the regularization problem \cref{regu_data} in the data space, which enables the estimation of $\lambda$ in an iterative framework. These considerations make the Krylov subspace method particularly attractive, as they rely only on matrix-vector products and naturally capture dominant low-dimensional data-informed subspaces.

\subsection{Quotient-space Golub--Kahan bidiagonalization}
In the data space, the regularization problem \cref{regu_data} can be equivalently formulated as
\begin{equation}\label{regu_quot}
  \min_{[z]\in\mathbb{R}^{m}/\mathcal{N}(M)}\{\|Mz-y\|_{\Gamma^{-1}}^2 + \lambda\|[z]\|_{M}^2 \}.
\end{equation}
To design practical Krylov subspace method for this problem, one possible approach is to represent equivalence classes by elements in the range space $\mathcal R(M)$ and restrict the iteration accordingly. However, this requires explicit projections onto $\mathcal R(M)$, typically involving the projector $MM^\dagger$, which is computationally expensive. Instead, we work directly on the quotient space $\mathbb R^m/\mathcal N(M)$, which removes the degeneracy without requiring explicit projections. 

Taking into account the geometry of the data-space and the observation model, we consider the linear operator
\begin{align}\label{eq:TM-def}
  \begin{split}
  \mathcal T : (\mathbb R^m/\mathcal N(M),\langle\cdot,\cdot\rangle_M) 
  &\longrightarrow (\mathbb R^m,\langle\cdot,\cdot\rangle_{\Gamma^{-1}}) \\
  [u] &\longmapsto Mu,
  \end{split}
\end{align}
which is well defined, because $[u]=[v] \Leftrightarrow  u-v\in\mathcal N(M)$ implies $Mu=Mv$.
Its adjoint
\[\mathcal T^* : (\mathbb R^m,\langle\cdot,\cdot\rangle_{\Gamma^{-1}}) \longrightarrow  (\mathbb R^m/\mathcal N(M),\langle\cdot,\cdot\rangle_M) \]
is characterized by the relation
\[\langle \mathcal T[u], p\rangle_{\Gamma^{-1}}=
\langle [u], \mathcal T^\ast p\rangle_M, \quad \forall \, [u]\in \mathbb R^m/\mathcal N(M),\ p\in \mathbb R^m. \]
The following result provides an explicit expression for the action of $\mathcal T^\ast$ on a vector.

\begin{lemma}\label{lem:adjoint-TM}
For every $p\in\mathbb R^m$, the adjoint of $\mathcal T$ is given by
\begin{equation}\label{eq:adjoint-TM}
\mathcal T^\ast p = [\Gamma^{-1}p].
\end{equation}
\end{lemma}
\begin{proof}
By the definition of $\mathcal T^\ast$, for any $[u]\in \mathbb R^m/\mathcal N(M)$ and $p\in\mathbb R^m$,
\[\langle [u], \mathcal T^\ast p\rangle_M =
\langle \mathcal T[u], p\rangle_{\Gamma^{-1}}
= (Mu)^\top \Gamma^{-1}p = u^\top M \Gamma^{-1}p.\]
By definition of the inner product on $\mathbb R^m/\mathcal N(M)$, we have 
\[ u^\top M \Gamma^{-1}p = \langle [u],[\Gamma^{-1}p]\rangle_M
= \langle [u],\mathcal T^\ast p\rangle_M,
\]
which shows that $\mathcal T^\ast p = [\Gamma^{-1}p]$. In other words, $\Gamma^{-1}p$ is a representative element of the equivalence class $\mathcal T^\ast p$.
\end{proof}

We now apply the Golub--Kahan bidiagonalization (GKB) to the operator-vector pair $\{\mathcal T, y\}$; see \cite{caruso2019convergence} for the GKB process for general linear compact operators. This formulation allows the iteration to be initialized directly from the vector $y$ and carried out entirely within the data-space framework without computing orthogonal projections. The abstract GKB recursions read
\begin{equation}\label{eq:abstract-gGKB}
\left\{
\begin{aligned} 
& \beta_1 u_1 = y,\\
& \alpha_i\bar{w}_i = \mathcal T^\ast u_i - \beta_i \bar{w}_{i-1},\\
& \beta_{i+1} u_{i+1} = \mathcal T \bar{w}_i - \alpha_i u_i, \quad i=1, 2, \dots,
\end{aligned}
\right.
\end{equation}
where $\bar{w}_0:=0$, and $\alpha_i$ and $\beta_i$ are positive scalars such that $\{u_i\}_i$ are $\Gamma^{-1}$-orthonormal and $\{\bar{w}_i\}_i$ are orthonormal in the data space $(\mathbb R^m/\mathcal N(M),\langle\cdot,\cdot\rangle_M)$.

To obtain a practical matrix-vector implementation, we represent each quotient vector $\bar{w}_i$ by a representative element $v_i\in\mathbb R^m$. Then $\|\bar{w}_i\|_M=\sqrt{v_i^\top M v_i}$. By \Cref{lem:adjoint-TM}, we have $\mathcal T^\ast u_i=[\Gamma^{-1}u_i]$. Using these relations, the abstract iterations reduce to the practical recursions:
\begin{equation}\label{eq:practical-gGKB}
\left\{
\begin{aligned}
& \beta_1 u_1 = y,\\
& \alpha_i v_i = \Gamma^{-1}u_i-\beta_i v_{i-1},\\
& \beta_{i+1} u_{i+1} = M v_i-\alpha_i u_i, \quad i=1, 2, \dots,
\end{aligned}
\right.
\end{equation}
where $v_0:=0$, and the vectors $\{v_i\}_i$ are $M$-orthonormal, that is, $v_{i}^{\top}Mv_j=\delta_{ij}$ with $\delta_{ij}$ being the Kronecker symbol. The scalars $\alpha_i$ and $\beta_{i+1}$ can be computed by evaluating the $M$- and $\Gamma^{-1}$-norms of $\Gamma^{-1}u_i-\beta_i v_{i-1}$ and $M v_i-\alpha_i u_i$, respectively. This process can be carried out iteratively. We refer to this process as the \emph{quotient-space Golub--Kahan bidiagonalization} (Q-GKB). The pseudocode of Q-GKB is given in \Cref{alg:qgkb}. Importantly, the Q-GKB iteration requires only matrix-vector products with $M$ and $\Gamma^{-1}$ and does not require computing orthogonal projections onto $\mathcal R(M)$. While computations with $\Gamma^{-1}$ are unavoidable, in many practical settings where $\eta$ is an uncorrelated Gaussian noise, $\Gamma$ is diagonal, so $\Gamma^{-1}$ can be applied efficiently.

\begin{algorithm}[htbp]
\caption{Quotient-space Golub--Kahan bidiagonalization (Q-GKB)}
\label{alg:qgkb}
\algorithmicrequire \ Semidefinite $M\in\mathbb{R}^{m\times m}$, positive definite $\Gamma\in\mathbb{R}^{m\times m}$, nonzero $y\in\mathbb{R}^{m}$
\begin{algorithmic}[1]
\State Initialize: $\beta_1 = (y\Gamma^{-1}y)^{1/2}$, \ $u_1 = y/\beta_1$
\State Compute $s=\Gamma^{-1}u_1$, \ $\alpha_1=(s^{\top}Ms)^{1/2}$, 
\If{$\alpha_1=0$} terminate 
\Else 
\State $v_1=s/\alpha_1$ \EndIf
\For{$i=1,2,\dots,k$}
\State $r = Mv_i-\alpha_iu_i$, \ $\beta_{i+1}=(r\Gamma^{-1}r)^{1/2}$
\If{$\beta_{i+1}=0$} terminate
\Else 
\State $u_{i+1}=r/\beta_{i+1}$
\EndIf
\State $s=\Gamma^{-1}u_{i+1}-\beta_{i+1}v_i$, \ $\alpha_{i+1}=(s^{\top}Ms)^{1/2}$
\If{$\alpha_{i+1}=0$} terminate
\Else 
\State $v_{i+1}=s/\alpha_{i+1}$
\EndIf
\EndFor
\end{algorithmic}
\algorithmicensure \ $\{\alpha_i, \beta_i\}_{i=1}^{k+1}$, \ $\{u_i, v_i\}_{i=1}^{k+1}$
\end{algorithm}

The Q-GKB iteration terminates in at most $\mathrm{rank}(G)$ steps, since $\{v_i\}_i$ are $M$-orthonormal and therefore cannot have more than $\mathrm{rank}(G)$ linearly independent vectors. Define the breakdown step as
\begin{equation}\label{break}
  k_b = \min_{k\ge 1}\{k: \alpha_{k+1}\beta_{k+1}=0\}.
\end{equation}
Before the iteration breakdown, i.e., $1\leq k\leq k_b$, let
\begin{equation}\label{orth_mat}
  U_{k+1} = (u_1,\dots,u_{k+1}) \in \mathbb R^{m\times (k+1)}, \quad
  V_k = (v_1,\dots,v_k) \in \mathbb R^{m\times k},
\end{equation}
and let
\begin{equation}\label{eq:bidiag-matrix}
B_k =
\begin{pmatrix}
\alpha_1 \\
\beta_2 & \alpha_2 \\
& \beta_3 & \ddots \\
&& \ddots & \alpha_k \\
&&& \beta_{k+1}
\end{pmatrix}
\in \mathbb R^{(k+1)\times k}.
\end{equation}
If $\beta_{k_b+1}=0$ at step $k=k_b$, then we let $u_{k_b+1}=0$ in $U_{k_b+1}$. Now each computed vector $v_i$ is a representative element of $\bar{w}_i$ that is mutually orthonormal in $(\mathbb R^m/\mathcal N(M),\langle\cdot,\cdot\rangle_M)$. By the property of GKB of linear compact operator, it holds that $V_k^\top M V_k = I$. Moreover, if $k=k_b$ and $\beta_{k_b+1}=0$, then $U_{k_b}^\top \Gamma^{-1}U_{k_b}=I$, otherwise it holds that $U_{k+1}^\top \Gamma^{-1}U_{k+1}=I$ for any $1\le k \le k_b$.

The following result gives the basic properties of Q-GKB in the data space. We use $e_1$ and $e_{k+1}$ to denote the first and $(k+1)$th column of the identity matrix of order $k+1$.

\begin{proposition}\label{thm:qsggkb-properties}
Assume that Q-GKB runs for $k$ steps with $k\leq k_b$. Then the following properties hold.
\begin{enumerate}
\item[(a)] The matrix-form relations hold:
\begin{equation}\label{eq:bidiag-relations}
\left\{
\begin{aligned}
& \beta_1 U_{k+1} e_1 = y, \\
& M V_{k} = U_{k+1} B_k, \\
& \Gamma^{-1} U_{k+1} = V_{k} B_k^\top + \alpha_{k+1} v_{k+1} e_{k+1}^\top .
\end{aligned}
\right.
\end{equation}
\item[(b)] The vectors $\{u_i\}_{i=1}^k$ form an $\Gamma^{-1}$-orthonormal basis of the Krylov subspace
\[\mathcal K_k(M\Gamma^{-1},y) = \operatorname{span}\{(M\Gamma^{-1})^i y\}_{i=0}^{k-1},\]
and vectors $\{v_i\}_{i=1}^k$ form an $M$-orthonormal basis of the Krylov subspace
\[\mathcal K_k(\Gamma^{-1}M,\Gamma^{-1}y) = \operatorname{span}\{(\Gamma^{-1}M)^i\Gamma^{-1}y\}_{i=0}^{k-1}.\]
\end{enumerate}
\end{proposition}
\begin{proof}
Part (a) is directly obtained by stacking the recurrences in \cref{eq:practical-gGKB}. 

For part (b), from the theory of GKB, it produces two Krylov subspaces,
\begin{equation*}\label{Krylov_v}
  \mathcal K_{k}(\mathcal{T}^*\mathcal{T}, \mathcal{T}^*y) 
  = \mathrm{span}\{(\mathcal{T}^*\mathcal{T})^{i}\mathcal{T}^*y\}_{i=0}^{k-1}
  = \mathrm{span}\{[v_i]\}_{i=1}^{k} 
\end{equation*}
and 
\begin{equation*}
  \mathcal K_{k}(\mathcal{T}\mathcal{T}^*, y) 
  = \mathrm{span}\{(\mathcal{T}\mathcal{T}^*)^{i}y\}_{i=0}^{k-1}
  =  \mathrm{span}\{u_i\}_{i=1}^{k},
\end{equation*}
such that $\{u_i\}_{i=1}^{k}$ and $\{[v_i]\}_{i=1}^{k}$ are $\Gamma$- and $M$-orthonormal, respectively. From the properties of $\mathcal{T}$ and $\mathcal{T}^*$, we have 
\begin{equation*}
  \mathcal{T}^*\mathcal{T}[u] = [\Gamma^{-1}Mu], \quad
  \mathcal{T}\mathcal{T}^{*}v = M\Gamma^{-1}v
\end{equation*}
for any $[u]\in \mathbb{R}^m/\mathcal{N}(M)$ and $v\in\mathbb{R}^{m}$. Therefore, we have 
\begin{equation*}
  \mathcal K_{k}(\mathcal{T}^*\mathcal{T}, \mathcal{T}^*y) = \mathrm{span}\{[(\Gamma^{-1}M)^i\Gamma^{-1}y]\}_{i=0}^{k-1}, \quad
  \mathcal K_{k}(\mathcal{T}\mathcal{T}^*, y) = \mathrm{span}\{(M\Gamma^{-1})^iy\}_{i=0}^{k-1} .
\end{equation*}
This proves the desired property for $\{u_i\}_{i=1}^{k}$. Moreover, from the recursions \cref{eq:practical-gGKB}, we have  
\begin{equation*}
  \alpha_1\beta_1 v_1 = \Gamma^{-1}y, \quad
  \alpha_{i+1}\beta_{i+1}v_{i+1} = \Gamma^{-1}Mv_{i} - (\alpha_{i}^2+\beta_{i+1}^2)v_i - \alpha_{i}\beta_{i}v_{i-1}.
\end{equation*}
Combining these relations with mathematical induction, it can be verified that $v_i\in \mathcal K_k(\Gamma^{-1}M,\Gamma^{-1}y)$ for $1\leq i\leq k$. Since $\{v_i\}_{i=1}^{k}$ are $M$-orthonormal, it holds that 
\[\mathrm{dim}\left(\mathrm{span}\{v_i\}_{i=1}^{k}\right)=k=\mathrm{dim}\left(K_k(\Gamma^{-1}M,\Gamma^{-1}y)\right).\]
This implies that $\{v_i\}_{i=1}^k$ forms an $M$-orthonormal basis of $K_k(\Gamma^{-1}M,\Gamma^{-1}y)$.
\end{proof}

\Cref{thm:qsggkb-properties} shows that Q-GKB constructs mutually orthonormal bases of paired Krylov subspaces in the data space, analogous to the classical GKB process. On the other hand, it is not merely a formal modification of standard GKB, but is intrinsically adapted to the quotient-space geometry, enabling the iteration to proceed without explicitly resolving the null space of $M$.

\subsection{Data-informed Krylov subspace}
At the $k$th step of Q-GKB, define $\mathcal{S}_k=\mathrm{span}\{[v_i]\}_{i=1}^{k}$, which is a $k$-dimensional subspace of the data space. We approximate the solution of \cref{regu_quot} by restricting it to $\mathcal{S}_k$, i.e.,
\begin{equation}\label{proj_regu}
  \min_{[z]\in\mathcal{S}_k}\{\|Mz-y\|_{\Gamma^{-1}}^2 + \lambda\|[z]\|_{M}^2 \},
\end{equation}
The following result shows that the Q-GKB iteration captures the exact posterior mean after finitely many steps.

\begin{theorem}\label{thm:tikhonov-krylov}
Let \(k_b\) be the breakdown step of Q-GKB defined in \cref{break}. Then for every \(\lambda>0\), the unique minimizer \([z_\lambda]\) of \cref{regu_quot} belongs to \(\mathcal S_{k_b}=\operatorname{span}\{[v_i]\}_{i=1}^{k_b}\).
\end{theorem}
\begin{proof}
  Since the objective function of \cref{regu_quot} can be equivalently written as $\|\mathcal{T}[z]-y\|_{\Gamma^{-1}}^2 + \lambda\|[z]\|_{M}^2$, its unique minimizer is the equivalence class
  \[[z_\lambda]=(\mathcal T^\ast\mathcal T+\lambda I)^{-1}\mathcal T^\ast y .\]
  Since \(\mathcal S_{k_b}\) is the maximal Krylov subspace generated by \(\mathcal T^\ast y\) under \(\mathcal T^\ast\mathcal T\), which means that $\mathcal S_{k_b}=\mathcal{K}_{\infty}(\mathcal T^\ast\mathcal T, \mathcal T^\ast y)$, it is invariant under \(\mathcal T^\ast\mathcal T\) and contains all vectors of the form \(p(\mathcal T^\ast\mathcal T)\mathcal T^\ast y\), where \(p\) is a polynomial. Note that $\mathcal T_{\lambda}:=\mathcal T^\ast\mathcal T+\lambda I$ is a positive self-adjoint operator on $\mathbb{R}^{m}/\mathcal{N}(M)$. Let $q=\mathrm{dim}(\mathbb{R}^{m}/\mathcal{N}(M))$ and let $A\in\mathbb{R}^{q\times q}$ be a matrix representation of $\mathcal T_{\lambda}$ with respect to a basis of $\mathbb{R}^{m}/\mathcal{N}(M)$. By the Cayley--Hamilton theorem, it holds that $p_{A}(\mathcal T_{\lambda})=0$, where $p_{A}(\mu) = \mathrm{det}(\mu I-A)=\sum_{i=0}^{q}c_i \mu^i$ is the characteristic polynomial of $A$, satisfying that $c_q=1$ and $c_0=(-1)^{q}\mathrm{det}(A)\neq 0$ since $A$ is SPD. It follows that 
  \[\mathcal{T}_{\lambda}\sum_{i=1}^{l}c_{i}(\mathcal T^\ast\mathcal T + \lambda I)^{i-1}=-c_0 I,\] 
  which implies that $\mathcal{T}_{\lambda}^{-1}$ can be represented by a polynomial of \(\mathcal T^\ast\mathcal T\). Hence there exists a polynomial \(p_\lambda\) such that
  \[(\mathcal T^\ast\mathcal T+\lambda I)^{-1}\mathcal T^\ast y = 
  p_\lambda(\mathcal T^\ast\mathcal T)\mathcal T^\ast y\in \mathcal S_{k_b}.\]
  This proves the claim.
\end{proof}

\Cref{thm:tikhonov-krylov} implies that the posterior mean $x_{\lambda}$ can be approximated within the subspace $\iota(\mathcal S_k)$ as $k$ increases. We refer to $\iota(\mathcal{S}_k)$ as the $k$-dimensional \emph{data-informed Krylov subspace}. Suppose there are $t$ distinct nonzero generalized eigenvalues of $(M,\,\Gamma)$ with corresponding $M$-orthogonal eigenspaces $\mathcal{G}_1,\dots,\mathcal{G}_{t}$. For any closed subspace $\mathcal{G}$ of $\mathbb{R}^{m}$, denote by $P_{\mathcal{G}}$ the $M$-orthogonal projector onto $\mathcal{G}$, i.e., $P_{\mathcal{G}}$ satisfies $P_{\mathcal{G}}^2=P_{\mathcal{G}}$ and $MP_{\mathcal{G}}^{\top}=P_{\mathcal{G}}M$. The following result shows that, as the iteration proceeds, the subspace $\iota(\mathcal{S}_k)$ progressively captures the informative directions of the dominant data-informed subspace.

\begin{theorem}\label{thm:qgkb-lanczos}
  Assume that Q-GKB runs for $k$ steps with $k\leq k_b$, and the nonzero 2-orthonormal eigenpairs of $B_{k}^{\top}B_{k}$ are $\left\{(\theta_{i}^{(k)},\,s_{i}^{(k)})\right\}_{i=1}^{k}$. Then the following statements hold.
\begin{enumerate}
  \item[(a)] The pair $\left(\theta_{i}^{(k)}, p_{i}^{(k)}\right):=\left(\theta_{i}^{(k)},\iota\left([V_{k}s_{i}^{(k)}]\right)\right)$ gradually approximates a nonzero generalized eigenpair of $(H,\,\Sigma^{-1})$ as $k$ increases, and the approximation becomes exact at $k=k_b$.
  \item[(b)] At the breakdown step, $\{p_{i}^{(k_b)}\}_{i=1}^{k_b}$ are $\Sigma^{-1}$-orthonormal generalized eigenvectors of $(H,\,\Sigma^{-1})$. Moreover, each $p_i^{(k_b)}$ is the image under $\iota$ of a generalized eigenvector in one of those $\mathcal G_j$ for which $P_{\mathcal G_j}\Gamma^{-1}y\neq 0$. 
\end{enumerate}
\end{theorem}
\begin{proof}
  (a) Using the matrix form relations of Q-GKB, we have
  \begin{align*}
    \Gamma^{-1}MV_{k} = \Gamma^{-1}U_{k+1}B_{k}
    = (V_{k}B_{k}^{\top}+\alpha_{k+1}v_{k+1}e_{k+1}^{\top})B_{k}.
  \end{align*}
  A simple calculation leads to 
  \begin{align*}
    MV_{k}s_{i}^{(k)} - \theta_{i}^{(k)}\Gamma V_{k}s_{i}^{(k)}
    &= \Gamma(V_{k}B_{k}^{\top}+\alpha_{k+1}v_{k+1}e_{k+1}^{\top})B_{k}s_{i}^{(k)} - \theta_{i}^{(k)}\Gamma V_{k}s_{i}^{(k)} \\
    &= \Gamma V_{k}(B_{k}^{\top}B_{k}s_{i}^{(k)}-\theta_{i}^{(k)}s_{i}^{(k)}) + \alpha_{k+1}\beta_{k+1}\Gamma v_{k+1}e_{k}^{\top}s_{i}^{(k)} \\
    &= \alpha_{k+1}\beta_{k+1}\Gamma v_{k+1}e_{k}^{\top}s_{i}^{(k)}.
  \end{align*}
  Therefore, the pair $\left(\theta_{i}^{(k)}, V_{k}s_{i}^{(k)}\right)$ gradually approximates a nonzero generalized eigenpair of $(M,\Gamma)$ as $k$ increases, and at the breakdown step, $\left(\theta_{i}^{(k_b)}, V_{k}s_{i}^{(k_b)}\right)$ is an exact eigenpair since $\alpha_{k_b+1}\beta_{k_b+1}=0$. By \Cref{thm:isom}, $\left(\theta_{i}^{(k)},\iota\left([V_{k}s_{i}^{(k)}]\right)\right)$ gradually approximates a nonzero eigenpair of $(H, \Sigma^{-1})$, which becomes an exact one at the breakdown step.

  (b) Without loss of generality, we suppose that only the first $q$ elements in $\{P_{\mathcal{G}_1}\Gamma^{-1}y,\dots,P_{\mathcal{G}_t}\Gamma^{-1}y\}$ are nonzero. Noticing that $\{V_{k_b}s_{i}^{(k_b)}\}_{i=1}^{k_b}$ are mutually $M$-orthonormal, by \Cref{thm:isom} (a) and (c), we only need to prove that the $k_b$ vectors $\hat{p}_{i}^{(k_b)}:=V_{k_b}s_{i}^{(k)}$ belong separately to the $M$-orthogonal subspaces $\mathcal{G}_{1},\dots,\mathcal{G}_q$. Since $\theta_{i}^{(k_b)}>0$ have different values and $\calG_i$ are mutually $M$-orthogonal, these $\hat{p}_{i}^{(k_b)}$ must belong to different subspaces among $\{\mathcal{G}_i\}_{i=1}^{t}$. Therefore, we only need to prove $P_{\calG}\hat{p}_{i}^{(k_b)}=\hat{p}_{i}^{(k_b)}$ for each $1\leq i\leq q$, where $\calG=\calG_1\oplus\cdots\oplus\calG_q$ and $P_{\calG}$ is the $M$-orthogonal projector onto $\calG$. Suppose $W_i$ is an $M$-orthonormal matrix whose columns span $\mathcal{G}_i$. Then $MW_i=\mu_{i}\Gamma W_i$ and $P_{\mathcal{G}_i}=W_{i}W_{i}^{\top}M$, which leads to 
  \[\Gamma^{-1}MP_{\mathcal{G}_i}=\Gamma^{-1}MW_{i}W_{i}^{\top}M=\mu_i W_{i}W_{i}^{\top}M=\mu_iP_{\mathcal{G}_i}.\] 
  Therefore, it holds that 
  \[\Gamma^{-1}Mx 
    = \Gamma^{-1}M\sum_{i=1}^{t}P_{\mathcal{G}_i}x
    = \sum_{i=1}^{t}\mu_{i}P_{\mathcal{G}_i}x, \quad  \forall \, x\in\mathbb{R}^{m} .\]
  For any $j\geq 0$, it follows that 
  \begin{equation*}
    \widetilde{w}_i:=(\Gamma^{-1}M)^{i}\Gamma^{-1}y = \left(\sum_{i=1}^{t}\mu_{i}P_{\mathcal{G}_i} \right)^i \Gamma^{-1}y
    = \sum_{i=1}^{t}\mu_{i}^{j}P_{\mathcal{G}_i}\Gamma^{-1}y
    = \sum_{i=1}^{q}\mu_{i}^{j}P_{\mathcal{G}_i}\Gamma^{-1}y .
  \end{equation*}
  By \Cref{thm:qsggkb-properties}, we have $\hat{p}_{i}^{(k_b)}\in\mathcal{K}_{k_b}(\Gamma^{-1}M,\Gamma^{-1}y)=\mathrm{span}\{\widetilde{w}_i\}_{i=0}^{k_b-1}$, and
	\begin{equation*}
		(\widetilde{w}_{0},\dots,\widetilde{w}_{k_b-1}) = (P_{\mathcal{G}_1}\Gamma^{-1}y, \dots, P_{\mathcal{G}_q}\Gamma^{-1}y)
		\begin{pmatrix}
			1 & \mu_{1} & \cdots & \mu_{1}^{k_b-1} \\
			1 & \mu_{2} & \cdots & \mu_{2}^{k_b-1} \\
			\vdots & \vdots & \ddots & \vdots \\
			1 & \mu_{q} & \cdots & \mu_{q}^{k_b-1}
			\end{pmatrix} =: \widetilde{G}T_{k_b},
	\end{equation*} 
	where $\widetilde{G}=(P_{\mathcal{G}_1}\Gamma^{-1}y, \dots, P_{\mathcal{G}_q}\Gamma^{-1}y)$. Since the $q$-by-$q$ leading part of $T_{k_b}$ is a Vandermonde matrix with $\mu_i\neq\mu_j$, the rank of $T_{k_b}$ is at most $q$. Thus the rank of $\widetilde{W}:=(w_{0},\dots,w_{k_b-1})$ is at most $q$, which is achieved if $k_b=q$. This implies that the expansion of $\mathcal K_{k}(\Gamma^{-1}M,\Gamma^{-1}y)$ ends at $k=q$, hence $k_b =q$. Since $T_{k_b}=T_{q}$ is nonsingular, it holds that $\calR(\widetilde{W})=\calR(\widetilde{G})$, and we can write $\hat{p}_{i}^{(k_b)}$ as $\hat{p}_{i}^{(k_b)}=\widetilde{G}c\in\calG$ with a nonzero $c\in\mathbb{R}^{q}$. Now we immediately obtain $P_{\calG}\hat{p}_{i}^{(k_b)}=\hat{p}_{i}^{(k_b)}$, which is the desired result.
\end{proof}

By \Cref{thm:qgkb-lanczos}, the subspace $\iota(\mathcal S_k)$ progressively captures one component from each generalized eigenspace $\mathcal G_j$ for which $P_{\mathcal G_j}\Gamma^{-1}y \neq 0$. Moreover, those components associated with larger generalized eigenvalues are approximated more rapidly, due to the well-known extremal eigenvalue convergence properties of Lanczos-type methods \cite{saad2011numerical}. Since the observation data takes the form $y = Gx + \eta$ with $\eta\sim\mathcal N(0,\Gamma)$, we have $\Gamma^{-1}y = \Gamma^{-1}Gx + \Gamma^{-1}\eta$ where $\Gamma^{-1}\eta \sim \mathcal N(0,\Gamma^{-1})$. Thus, we have $P_{\mathcal G_i}\Gamma^{-1}y \neq 0$ almost surely for all $i=1,\dots,t$, as the Gaussian distribution has full support. For ill-posed problems, the matrix $M=G\Sigma G^\top$ is typically severely ill-conditioned, and the generalized eigenvalues of $(M,\,\Gamma)$ tend to decay gradually toward zero without noticeable gaps. In such cases, eigenvalue multiplicities are typically small (often equal to one in practice). This indicates that, in typical settings, the Q-GKB iteration progressively captures all directions of the dominant data-informed subspaces through $\iota(\mathcal S_k)$.

Since $V_k s_i^{(k)}$ progressively approximates the dominant generalized eigenvectors of $(M,\Gamma)$, this motivates the following low-rank approximation of $H$, analogous to \cref{Hk_1}. After $k$ steps of Q-GKB, define the rank-$k$ symmetric positive semidefinite matrix
\begin{equation}
  \widehat{H}_k = G^{\top} \left(\sum_{i=1}^{k}\theta_{i}^{(k)}(V_{k}s_{i}^{(k)})(V_{k}s_{i}^{(k)})^{\top}\right)G = G^{\top}V_{k}B_{k}^{\top}B_{k}V_{k}^{\top}G .
\end{equation}
By \Cref{thm:qgkb-lanczos} and its proof, if all nonzero generalized eigenvalues of $(M,\Gamma)$ are simple and $P_{\mathcal G_i}\Gamma^{-1}y \neq 0$ for $i=1,\dots,l$, while $P_{\mathcal G_i}\Gamma^{-1}y = 0$ for $i>k$, then the breakdown occurs at $k_b=l$ and $\widehat{H}_l = H_l$. Replacing $H=G^{\top}\Gamma^{-1}G$ in \cref{post_distr} by $\widehat{H}_k$, we define the \emph{approximate Gaussian posterior} with mean and covariance given by
\begin{equation}\label{approx_post_distr}
  \widehat{x}_{\lambda}^{(k)} = \widehat{C}_{\lambda}^{(k)}G^{\top}\Gamma^{-1}y, 
  \qquad
  \widehat{C}_{\lambda}^{(k)} = (\widehat{H}_k+\lambda\Sigma^{-1})^{-1}.
\end{equation}
In the next section, we show how to simultaneously estimate the hyperparameter $\lambda$ and compute the approximate posterior mean and covariance efficiently within the Q-GKB iteration.

\section{Iterative approximation of the Bayesian posterior}\label{sec4}
The goal of this section is to construct efficient approximations of the Bayesian posterior based on Q-GKB interation, and to analyze the accuracy of these approximations. We integrate empirical Bayesian inference into the Q-GKB iteration to obtain an efficient estimate of $\lambda$ within the data-informed Krylov subspaces. Given $\lambda=\lambda_k$, we then compute the approximations $\widehat{x}_{\lambda}^{(k)}$ and $\widehat{C}_{\lambda}^{(k)}$ using byproducts of the Q-GKB iteration.

\subsection{Approximate the Bayesian posterior by Q-GKB}
In the empirical Bayesian (EB) inference, the hyperparameter $\lambda$ is treated as a random variable and correspondingly, the prior of $x$ is conditioned on $\lambda$. Suppose the prior of $\lambda$ is $\pi(\lambda)$. Since $x\,|\,\lambda\sim\mathcal{N}(0,\lambda^{-1}\Sigma)$ and $y\,|\,(x,\lambda)\sim\mathcal{N}(Gx,\Gamma)$, using Bayes' formula, we have 
\begin{align*}
  \pi(x,\lambda\,|\,y) 
  &= \frac{\pi(y\,|\,x,\lambda)\pi(x\,|\,\lambda)\pi(\lambda)}{\pi(y)} 
  \propto \frac{\pi(\lambda)\lambda^{n/2}\exp\left(-\frac{1}{2}\|y-Gx\|_{\Gamma^{-1}}^2-\frac{\lambda}{2}\|x\|_{\Sigma^{-1}}^2 \right)}{\det(\Gamma)^{1/2}\det(\Sigma)^{1/2}} .
\end{align*}
Integrating out $x$, the marginal posterior of $\lambda$ is
\begin{align*}
  \pi(\lambda\,|\,y) 
  &= \int_{\mathbb{R}^{n}} \pi(x,\lambda\,|\,y)\mathrm{d}x \\
  &\propto \pi(\lambda)\lambda^{n/2} \int_{\mathbb{R}^{n}} \exp\left(-\frac{1}{2}\|y-Gx\|_{\Gamma^{-1}}^2-\frac{\lambda}{2}\|x\|_{\Sigma^{-1}}^2 \right)\mathrm{d}x \\
  &\propto \pi(\lambda)\det(\Gamma+\lambda^{-1}G\Sigma G^{\top})^{-1/2}\exp\left(-\frac{1}{2}y^{\top}(\Gamma+\lambda^{-1}G\Sigma G^{\top})^{-1}y \right) ,
\end{align*}
see e.g., \cite{franklin1970well,lehtinen1989linear}.
Notice that $\pi(\lambda\,|\,y) = \frac{\pi(y\,|\,\lambda)\pi(\lambda)}{\pi(y)} \propto \pi(\lambda)\pi(y\,|\,\lambda)$. In the EB method, the hyperparameter $\lambda$ is typically estimated by maximizing the marginal likelihood $\pi(y\,|\,\lambda)$, rather than the posterior $\pi(\lambda\,|\,y)$. Ignoring constant terms, minimizing the negative log-marginal likelihood $-\log \pi(y\,|\,\lambda)$ is equivalent to solving the following optimization problem:
\begin{equation*}
  \min_{\lambda>0}\mathcal{L}(\lambda) 
  := \log \det(\Gamma+\lambda^{-1}M) + y^{\top}(\Gamma+\lambda^{-1}M)^{-1}y.
\end{equation*}
Although this problem is formulated in the data space, evaluating $\mathcal{L}(\lambda)$ remains computationally prohibitive for large $m$ due to the cost of determinant and matrix inversion operations. To address this, we integrate the EB approach into the Q-GKB iteration, enabling an efficient, adaptive estimation of $\lambda$ within low-dimensional data-informed Krylov subspaces.

After $k$ steps of Q-GKB, define the symmetric positive semidefinite matrix 
\begin{equation}
  M_{k} = MV_{k}V_{k}^{\top}M ,
\end{equation}
which has rank at most $k$. For the $M$-orthogonal projector
\begin{align*}
  \Pi_k: \mathbb{R}^{m}/\mathcal{N}(M) &\longrightarrow \mathcal{S}_k \\
  [z] &\longmapsto \sum_{i=1}^{k}\langle [z], [v_i] \rangle_{M} [v_i] ,
\end{align*}
it holds that 
\begin{equation*}
  \mathcal{T}(\Pi_{k}[z]) = \mathcal{T}([V_{k}V_{k}^{\top}Mz])
  = MV_{k}V_{k}^{\top}Mz = M_kz.
\end{equation*}
This implies that $M_{k}$ is the compression of $M$ onto the Krylov subspace $\mathcal{S}_k$. Accordingly, $M_k$ provides a natural low-rank approximation of $M$. Using the relation $M V_k = U_{k+1} B_k$, we obtain
\begin{equation}
  M_k = MV_{k}V_{k}^{\top}M=U_{k+1}B_{k}B_{k}^{\top}U_{k+1}^{\top}.
\end{equation}
Thus $M_{k}$ can be formed using only the small scale matrices $B_{k}$ and $U_{k+1}$ generated by Q-GKB. Replacing $M$ by $M_k$ in $\mathcal{L}(\lambda)$ yields the $k$th approximate negative log-marginal likelihood:
\begin{equation}
  \mathcal{L}^{(k)}(\lambda) = \log \det(\Gamma+\lambda^{-1}M_k) + y^{\top}(\Gamma+\lambda^{-1}M_k)^{-1}y.
\end{equation}
The following result shows that the approximate marginal likelihood can be computed efficiently using only the singular value decomposition (SVD) of $B_{k}$.

\begin{proposition}\label{prop:Lk}
  Suppose that the compact form SVD of $B_{k}$ is $B_{k}=PSQ^{\top}$, where $P\in\mathbb{R}^{(k+1)\times k}$ has 2-orthonormal columns and $S=\mathrm{diag}(s_1,\dots,s_k)$ with $s_i>0$. Then 
  \begin{equation}\label{L_k}
    \mathcal{L}^{(k)}(\lambda) = \log\det(\Gamma) + \sum_{i=1}^{k}\log\left(1+\frac{s_i^2}{\lambda}\right)
    + \beta_1^2\left(1-\sum_{i=1}^{k}\frac{s_i^2}{s_i^2+\lambda}p_{1i}^2 \right),
  \end{equation}
  where $(p_{11},\dots,p_{1k})$ is the first row of $P$.
\end{proposition}
\begin{proof}
  For the determinant term in $\mathcal{L}^{(k)}(\lambda)$, we write 
  \begin{align*}
    \Gamma+\lambda^{-1}M_k
    &= \Gamma^{1/2}\Bigl(I+\lambda^{-1}\Gamma^{-1/2}U_{k+1}B_kB_k^\top U_{k+1}^\top\Gamma^{-1/2}\Bigr)\Gamma^{1/2} \\
    &= \Gamma^{1/2}\Bigl(I+\lambda^{-1}WB_kB_k^\top W^{\top}\Bigr)\Gamma^{1/2},
  \end{align*}
  where $W:=\Gamma^{-1/2}U_{k+1}$ is column 2-orthonormal. Therefore, it holds that 
  \begin{align*}
    \det(\Gamma+\lambda^{-1}M_k)
    = \det(\Gamma)\det\left(I+\lambda^{-1}W B_k B_k^\top W^\top\right) 
    = \det(\Gamma)\det\left(I+\lambda^{-1}B_k^\top B_k \right),
  \end{align*}
  where we have used the determinant identity $\det(I+AB)=\det(I+BA)$. Since the eigenvalues of $B_{k}^{\top}B_{k}$ are $s_1^2,\dots,s_k^2$, we obtain 
  \begin{equation*}
    \det(\Gamma+\lambda^{-1}M_k) = \det(\Gamma)\prod_{i=1}^k\left(1+\frac{s_i^2}{\lambda}\right).
  \end{equation*}
  For the quadratic term in $\mathcal{L}^{(k)}(\lambda)$, using $\Gamma^{-1/2}y=\Gamma^{-1/2}\beta_1 U_{k+1} e_1=\beta_1 W e_1$, we have 
  \begin{align*}
    y^\top(\Gamma+\lambda^{-1}M_k)^{-1}y
    &= y^\top \Gamma^{-1/2}\left(I+\lambda^{-1}W B_k B_k^\top W^\top\right)^{-1}\Gamma^{-1/2}y \\
    &= \beta_1^2e_1^\top W^\top\left(I+\lambda^{-1}W B_k B_k^\top W^\top\right)^{-1}We_1 \\
    &= \beta_1^2e_1^\top \left[ I - B_{k}(B_{k}^{\top}B_{k}+\lambda I)^{-1}B_{k}^{\top} \right]e_{1} \\
    &= \beta_1^2\left[1-e_1^\top P\operatorname{diag}\left(\frac{s_1^2}{s_1^2+\lambda},\dots,\frac{s_k^2}{s_k^2+\lambda}\right)P^\top e_1\right]
  \end{align*}
  where we have used the identity $(I+\lambda^{-1}BB^\top)^{-1}=I-B(B^\top B+\lambda I)^{-1}B^\top$ and the SVD of $B_{k}$. Combining the above two terms leads to the desired result.
\end{proof}

At the $k$th iteration, we define $\lambda_k = \argmin_{\lambda>0}\mathcal{L}^{(k)}(\lambda)$. By \Cref{prop:Lk}, the term $\log\det(\Gamma)$ does not affect the minimization of $\mathcal{L}^{(k)}(\lambda)$ and can therefore be omitted. Hence, the dominant computational cost for determining $\lambda_k$ is the SVD of $B_k$, which requires only $O(k^3)$ flops.

By \Cref{thm:tikhonov-krylov}, we approximate $z_{\lambda}$ in $\mathcal{S}_k$ by solving \cref{proj_regu} with $\lambda=\lambda_k$. Since any $[z]\in\mathcal{S}_k$ can be written as $[z]=[V_k\xi]$ for some $\xi\in\mathbb{R}^k$ and $Mz=MV_{k}\xi$, using \cref{eq:bidiag-relations} and the $\Gamma^{-1}$-orthonormality of $U_{k+1}$, we obtain
\[ \|Mz-y\|_{\Gamma^{-1}} = \|U_{k+1}(B_k\xi - \beta_1 e_1)\|_{\Gamma^{-1}}
= \|B_k\xi - \beta_1 e_1\|_2 . \] 
Moreover, using the $M$-orthonormality of $V_k$, we have
\[ \|[z]\|_{M}^2 = \|[V_{k}\xi]\|_{M}^2 = \xi^{\top}V_{k}^{\top}MV_{k}\xi = \|\xi\|_{2}^2 . \]
Therefore, \cref{proj_regu} is equivalent to
\begin{equation}\label{proj_ls}
  \min_{\xi\in\mathbb{R}^{k}}\{\|B_k\xi-\beta_1e_1\|_2^2+\lambda_k\|\xi\|_2^2\},
\end{equation}
which has a unique minimizer $\xi_{\lambda_k}^{(k)}\in\mathbb{R}^{k}$. Consequently, the $k$th approximation to $z_{\lambda}$ is $[z_{\lambda_k}^{(k)}] = [V_{k}\xi_{\lambda_k}^{(k)}]$, and corresponding iteratively regularized solution is 
\begin{equation}\label{hyb_sol}
  x_{\lambda_k}^{(k)} = \Sigma G^{\top}V_{k}\xi_{\lambda_k}^{(k)}, \quad
  \xi_{\lambda_k}^{(k)} = (B_{k}^{\top}B_{k}+\lambda_k I)^{-1}B_{k}^{\top}\beta_1e_1.
\end{equation}
The following result provides an efficient computation of $\widehat{x}_{\lambda}^{(k)}$ and $\widehat{C}_{\lambda}^{(k)}$ defined in \cref{approx_post_distr}, and shows that the approximate posterior mean in the data-informed Krylov subspace coincides with the iteratively regularized solution. In what follows, we write $\lambda$ in place of $\lambda_k$ at the $k$th step to emphasize that the results hold for any value of $\lambda$.

\begin{proposition}\label{prop:approx_post}
At the $k$th step, the approximate posterior mean coincides with the iteratively regularized solution, namely,
\begin{equation}\label{eq:xhatk-form}
  \widehat{x}_{\lambda}^{(k)} = x_{\lambda}^{(k)}
  = \Sigma G^\top V_k(B_k^\top B_k+\lambda I)^{-1}B_k^\top \beta_1 e_1,
\end{equation}
and the approximate covariance admits the representation
\begin{equation}\label{eq:Chatk-form}
  \widehat{C}_{\lambda}^{(k)}
  = \lambda^{-1}\Sigma-\lambda^{-1}\Sigma G^\top V_k(\lambda I+B_k^\top B_k)^{-1}B_k^\top B_k
  V_k^\top G\Sigma .
\end{equation}
\end{proposition}
\begin{proof}
  Applying the Woodbury identity to the expression of $\widehat{C}_{\lambda}^{(k)}$, we have 
  \[\widehat{C}_{\lambda}^{(k)} = \lambda^{-1}\Sigma-\lambda^{-2}\Sigma G^\top V_k\Bigl((B_k^\top B_k)^{-1}
  + \lambda^{-1}V_k^\top G\Sigma G^\top V_k\Bigr)^{-1}V_k^\top G\Sigma .\]
  Using $V_k^\top G\Sigma G^\top V_k = V_k^\top M V_k =I$ and $\Bigl(T_{k}^{-1} +\lambda^{-1}I\Bigr)^{-1} = \lambda(\lambda I+T_{k})^{-1}T_{k}$ with $T_{k}:=B_k^\top B_k$, we immediately obtain \cref{eq:Chatk-form}.

  Now we show $\widehat{x}_{\lambda}^{(k)}=x_{\lambda}^{(k)}$. Using \cref{eq:Chatk-form}, we have
  \begin{align*}
    \widehat{x}_{\lambda}^{(k)}
    = \lambda^{-1}\Sigma G^\top\Gamma^{-1}y-\lambda^{-1}\Sigma G^\top V_k(\lambda I+T_k)^{-1}T_kV_k^\top M\Gamma^{-1}y.
  \end{align*}
  Since $MV_k=U_{k+1}B_k$, we have $V_k^\top M = B_k^\top U_{k+1}^\top$. From $\beta_1U_{k+1}e_1=y$, we have $U_{k+1}^\top\Gamma^{-1}y=\beta_1e_1$. Therefore, $V_k^\top M\Gamma^{-1}y = B_k^\top U_{k+1}^\top\Gamma^{-1}y = B_k^\top\beta_1e_1$. On the other hand, from \cref{eq:bidiag-relations}, we have $\Gamma^{-1}y =\beta_1\Gamma^{-1}U_{k+1}e_1=V_kB_k^\top\beta_1e_1$. It follows that $\lambda^{-1}\Sigma G^\top\Gamma^{-1}y=\lambda^{-1}\Sigma G^\top V_kB_k^\top\beta_1e_1$. Substituting these into the expression for $\widehat{x}_{\lambda}^{(k)}$ gives
  \begin{align*}
    \widehat{x}_{\lambda}^{(k)}
    &= \lambda^{-1}\Sigma G^\top V_kB_k^\top\beta_1e_1-\lambda^{-1}\Sigma G^\top V_k(\lambda I+T_k)^{-1}T_kB_k^\top\beta_1e_1 \\
    &= \Sigma G^\top V_k \left[\lambda^{-1}I-\lambda^{-1}(\lambda I+T_k)^{-1}T_{k}\right]B_k^\top\beta_1e_1 \\
    &= \Sigma G^\top V_k(\lambda I+T_k)^{-1}B_k^\top\beta_1e_1,
  \end{align*}
  where we used $I-(\lambda I + T_k)^{-1}T_{k}=\lambda(\lambda I + T_{k})^{-1}$. This completes the proof.
\end{proof}

By this result, the posterior mean can be approximated iteratively as the Q-GKB iteration proceeds, while the approximate posterior covariance can be computed efficiently using only the inversion of small-scale matrices. The resulting algorithm is summarized in \Cref{alg2}. As a stopping criterion, we can choose a tolerance $\mathtt{tol}>0$ and terminate the iteration when the relative change in $\lambda$ satisfies $|\lambda_k-\lambda_{k+1}|/\lambda_k\leq \mathtt{tol}$.

\begin{algorithm}[htbp]
\caption{Q-GKB based posterior approximation}
\label{alg2}
\algorithmicrequire \ Matrix $G\in\mathbb{R}^{m\times n}$, SPD matrices $\Sigma\in\mathbb{R}^{n\times n}$,  $\Gamma\in\mathbb{R}^{m\times m}$, vector $y\in\mathbb{R}^{m}$
\begin{algorithmic}[1]
\State Compute $\beta_1$, $\alpha_1$, $u_1$, $v_1$ by Q-GKB
\For{$k=1,2,\dots,K$}
\State Compute $\beta_{k+1}$, $\alpha_{k+1}$, $u_{k+1}$, $v_{k+1}$ by Q-GKB; form $B_{k}$ and $V_{k}$
\State (Terminate the iteration if $\beta_{k+1}$ or $\alpha_{k+1}$ is extremely small)
\State Compute the SVD of $B_k$; form $\mathcal{L}^{(k)}(\lambda)$ as \cref{L_k} \Comment{Discard $\log\det(\Gamma)$}
\State Compute $\lambda_k = \argmin_{\lambda>0}\mathcal{L}^{(k)}(\lambda)$
\State Compute $\widehat{x}_{\lambda_k}^{(k)}$ and $\widehat{C}_{\lambda_k}^{(k)}$ by \cref{eq:xhatk-form} and \cref{eq:Chatk-form}
\EndFor
\end{algorithmic}
\algorithmicensure \ $\widehat{x}_{\lambda_K}^{(K)}$, $\widehat{C}_{\lambda_K}^{(K)}$
\end{algorithm}

We remark that, for large-scale problems, the prior covariance matrix $\Sigma$ need not be formed explicitly in the algorithm. Instead, matrix-vector products of the form $\Sigma v$ can be computed efficiently. In particular, when the prior is given by a stationary or translation invariant kernels, the product $\Sigma v$ can be evaluated efficiently using the fast Fourier transform (FFT) \cite{nowak2003efficient,Williams2006gaussian}. We will illustrate such a case through a numerical example.

\subsection{Accuracy of the approximate posterior}\label{sec4.2}
We quantify the accuracy of the approximate posterior distribution constructed via the Q-GKB iteration. The exact posterior and its $k$th approximation are given by
\begin{equation}\label{eq:approx-post}
  \pi = \mathcal{N}(x_\lambda, C_\lambda), \qquad
  \widehat{\pi}_k = \mathcal{N}(\widehat{x}_\lambda^{(k)}, \widehat{C}_\lambda^{(k)}).
\end{equation}
To measure the difference between the exact covariance and its approximation, we employ the F{\"o}rstner distance. Denote by $\mathrm{Tr}(\cdot)$ the trace of a matrix. For two $d$-by-$d$ SPD matrices $A$ and $B$, the F{\"o}rstner distance is defined by
\begin{equation*}
  d_{F}^2(A,B) = \mathrm{Tr}\left[\log^2(A^{-1/2}BA^{-1/2})\right]
  = \sum_{i=1}^{d}\log^2(\sigma_i) ,
\end{equation*}
where $\{\sigma_i\}_{i=1}^{d}$ are the positive generalized eigenvalues of $(A,B)$~\cite{forstner2003metric}. It has been shown in~\cite{forstner2003metric,spantini2015optimal} that the F{\"o}rstner distance is particularly well suited for comparing covariance matrices. To measure the difference between two probability distributions, we use the Kullback--Leibler (KL) divergence. 
Let $\pi_1 = \mathcal N(m_1,\Sigma_1)$ and $\pi_2 = \mathcal N(m_2,\Sigma_2)$ be two Gaussian distributions on $\mathbb R^d$ with positive definite covariances. The KL divergence between them is defined as
\begin{align*}
  D_{\rm KL}(\pi_1\|\pi_2) 
  = \frac12 \Big[\mathrm{Tr}(\Sigma_2^{-1}\Sigma_1)-d-\log\det(\Sigma_2^{-1}\Sigma_1)
  +(m_2-m_1)^\top \Sigma_2^{-1}(m_2-m_1)\Big].
\end{align*}
Before establishing error bounds for the posterior approximation, we introduce two auxiliary lemmas that will be used in the subsequent analysis. We use $\|\cdot\|_{F}$ to denote the Frobenius norm of a matrix. The proofs are postponed to \Cref{apdx:A}.

\begin{lemma}\label{lem:lanczos_decay}
  Let $A=\Sigma^{1/2}H\Sigma^{1/2}$, $\widehat{A}_k=\Sigma^{1/2}\widehat{H}_{k}\Sigma^{1/2}$ and $\widehat{A}_{0}=0$, and define 
  \begin{equation}
    \zeta_{k} = \mathrm{Tr}(A-\widehat{A}_k), \quad
    \gamma_{k} = \|A-\widehat{A}_k\|_{F} .
  \end{equation}
  Then it hold that $\zeta_k, \gamma_k\geq 0$, and
  \begin{equation}\label{zeta_gamma}
    \begin{cases}
      \zeta_{k+1} = \zeta_{k} - (\alpha_{k+1}^2+\beta_{k+2}^2), \quad \zeta_0=\mathrm{Tr}(H\Sigma), \\
      \gamma_{k+1}^2 = \gamma_{k}^2-2\alpha_{k+1}^2\beta_{k+1}^2-(\alpha_{k+1}^2+\beta_{k+2}^2)^2 , \quad \gamma_{0}^2=\mathrm{Tr}(H\Sigma H \Sigma).
    \end{cases}
  \end{equation}
\end{lemma}

\begin{lemma}\label{lem:inv_bnd}
  Let $A_1$ and $A_2$ be two SPD matrice of the same order, then
  \begin{equation*}
    \|(I+A_1)^{-1}-(I+A_2)^{-1}\|_2 \leq \frac{\|A_1-A_2\|_2}{1+\|A_1-A_2\|_2}.
  \end{equation*}
\end{lemma}

We are now ready to quantify the discrepancy between the exact posterior and its approximation. For clarity, we again write $\lambda$ in place of $\lambda_k$ to emphasize that the result hold for any value of $\lambda$.

\begin{theorem}\label{thm:bnd_post}
  Assume that Q-GKB runs for $k$ steps with $k\leq k_b$, then we have
  \begin{equation}
    d_{F}(C_{\lambda},\widehat{C}_{\lambda}^{(k)}) \leq \frac{\gamma_k}{\lambda} ,
  \end{equation}
  and 
  \begin{equation}
    D_{\mathrm{KL}}(\widehat{\pi}_k\Vert\pi) \leq \frac{1}{2\lambda}\left(\zeta_{k}+\frac{\alpha_{1}^2\beta_{1}^2\gamma_{k}^2}{\lambda(\lambda+\gamma_{k})}\right).
  \end{equation}
\end{theorem}
\begin{proof}
  We first give the upper bound for $d_{F}(C_{\lambda},\widehat{C}_{\lambda}^{(k)})$.
  Rewriting $C_{\lambda}$ and $\widehat{C}_{\lambda}^{(k)}$ as
  \begin{equation*}
    C_{\lambda}  = \Sigma^{1/2}(A+\lambda I)^{-1}\Sigma^{1/2}, \quad
    \widehat{C}_{\lambda}^{(k)} = \Sigma^{1/2}(\widehat{A}_k+\lambda I)^{-1}\Sigma^{1/2},
  \end{equation*}
  where $A=\Sigma^{1/2}H\Sigma^{1/2}$ and $\widehat{A}_k=\Sigma^{1/2}\widehat{H}_{k}\Sigma^{1/2}$, and using the property
  \begin{equation*}
    d_{F}(A,B) = d_{F}(A^{-1},B^{-1})
    = d_{F}(NAN^{\top},NBN^{\top})
  \end{equation*}
  for any nonsingular matrix $N$~\cite{forstner2003metric}, we get $d_{F}(C_{\lambda},\widehat{C}_{\lambda}^{(k)})=d_{F}(\lambda I+A,\lambda I+\widehat{A}_k)$. By definition of the F\"orstner distance, we have 
  \begin{equation*}
    d_{F}^2(\lambda I+A, \lambda I+\widehat A_k) = \sum_{i=1}^{m} \log^2 \left(\sigma_i(D_k)\right),
  \end{equation*}
  where $D_k=(\lambda I+A)(\lambda I+\widehat A_k)^{-1}=(I+\lambda^{-1}A)(I+\lambda^{-1}\widehat{A}_{k})^{-1}$, and $\{\sigma_{i}(D_k)\}_i$ are the eigenvalues of $D_k$ that equal to the generalized eigenvalue of $(\lambda I+A, \,\lambda I+\widehat A_k)$. Using the inequality $(\log t)^2\le t+t^{-1}-2$ for $t>0$, we obtain 
  \[d_{F}^2(\lambda I+A,\lambda I+\widehat A_k) \le \mathrm{Tr}(D_k)+\mathrm{Tr}(D_k^{-1})-2m. \]
  Now let $E_k=A-\widehat A_k$. Then it holds that 
  \begin{equation*}
    D_k-I = E_k(\lambda I+\widehat A_k)^{-1}, \quad
    D_k^{-1}-I = -E_k(\lambda I+A)^{-1}.
  \end{equation*}
  It follows that
  \begin{align*}
    \mathrm{Tr}(D_k)+\mathrm{Tr}(D_k^{-1})-2n
    &= \mathrm{Tr}\!\big(E_k(\lambda I+\widehat A_k)^{-1}\big) - \mathrm{Tr}\!\big(E_k(\lambda I+A)^{-1}\big) \\
    &= \mathrm{Tr}\!\Big(E_k(\lambda I+\widehat A_k)^{-1}E_k(\lambda I+A)^{-1}\Big),
  \end{align*} 
  where we have used the identity 
  \[(\lambda I+\widehat A_k)^{-1}-(\lambda I+A)^{-1} 
  = (\lambda I+\widehat A_k)^{-1}E_k(\lambda I+A)^{-1}.\] 
  Since $\|(\lambda I+A)^{-1}\|_2\le \lambda^{-1}$ and $\|(\lambda I+\widehat A_k)^{-1}\|_2\le \lambda^{-1}$, using \Cref{lem:lanczos_decay} we obtain 
  \begin{align*}
    \mathrm{Tr}\!\Big(E_k(\lambda I+\widehat A_k)^{-1}E_k(\lambda I+A)^{-1}\Big)
    \le \|E_k(\lambda I+\widehat A_k)^{-1}\|_F\|E_k(\lambda I+A)^{-1}\|_F 
    \le \frac{\|E_k\|_F^2}{\lambda^2} = \frac{\gamma_k^2}{\lambda^2}.
  \end{align*}
  This is the desired bound for $d_{F}(C_{\lambda},\widehat{C}_{\lambda}^{(k)})$.

  The KL divergence between $\widehat{\pi}_{k}$ and $\pi$ is 
  \begin{align*}
    D_{\mathrm{KL}}(\widehat\pi_k\|\pi)
    = \frac{1}{2}\Big[\underbrace{\mathrm{Tr}\!\big(C_\lambda^{-1}\widehat C_\lambda^{(k)}\big)-n-\log\det\!\big(C_\lambda^{-1}\widehat C_\lambda^{(k)}\big)}_{\text{I}} + \underbrace{\|x_\lambda-\widehat x_\lambda^{(k)}\|_{C_\lambda^{-1}}^2}_{\text{II}}\Big].
  \end{align*}
  In what follows, we bound the terms I and II separately.
  Noticing that $C_\lambda^{-1}\widehat C_\lambda^{(k)}=\Sigma^{-1/2}(A+\lambda I)(\widehat A_k+\lambda I)^{-1}\Sigma^{1/2}$, we have 
  \begin{align*}
    \text{I} 
    &= \mathrm{Tr}((A+\lambda I)(\widehat{A}_k+\lambda I)^{-1}) - n - \log\det((A+\lambda I)(\widehat{A}_k+\lambda I)^{-1}) \\
    &= \mathrm{Tr}\big(E_k(\widehat A_k+\lambda I)^{-1}\big) -\Big(\log\det(A+\lambda I)-\log\det(\widehat A_k+\lambda I)\Big) ,
  \end{align*}
  where $E_{k}=A-\widehat A_k$. From the proof of \Cref{lem:lanczos_decay}, we have $\widehat A_k=Q_kT_kQ_k^\top$ and $T_k:=B_{k}^{\top}B_{k}=Q_k^\top A Q_k$, where $Q_k\in\mathbb{R}^{m\ k}$ have orthonormal columns. Combining these relations with the Woodbury identity, we get 
  \[(\widehat A_k+\lambda I)^{-1} = \lambda^{-1}I-\lambda^{-1}Q_kT_k(T_k+\lambda I)^{-1}Q_k^\top,\]
  which leads to 
  \begin{align*}
    \mathrm{Tr}\big(E_k(\widehat A_k+\lambda I)^{-1}\big)
    &= \lambda^{-1}\mathrm{Tr}(E_k)-\lambda^{-1}\mathrm{Tr}\Big(Q_k^\top E_k Q_k\,T_k(T_k+\lambda I)^{-1}\Big) = \lambda^{-1}\mathrm{Tr}(E_k),
  \end{align*}
  since $Q_k^\top E_k Q_k = Q_k^\top A Q_k-Q_k^\top \widehat A_k Q_k = T_k-T_k=0$. Therefore, we get 
  \begin{equation*}
    \mathrm{Tr}((A+\lambda I)(\widehat{A}_k+\lambda I)^{-1}) - n = \frac{\zeta_k}{\lambda}.
  \end{equation*}
  Next, let \(\theta_1\ge \cdots \ge \theta_n\ge 0\) be the eigenvalues of \(A\), and \(\hat{\theta}_1\ge \cdots \ge \hat{\theta}_k\ge 0\) be the eigenvalues of \(T_k\). Noticing that $T_k=Q_k^\top A Q_k$, the Cauchy interlacing theorem \cite{horn2012matrix} implies that $\theta_i\ge\hat{\theta}_i$ for $i=1,\dots,k$. Since the eigenvalues of \(\widehat A_k=Q_kT_kQ_k^\top\) are \(\{\hat{\theta}_1,\dots,\hat{\theta}_k,0,\dots,0\}\), we have
  \begin{equation*}
    \det(\widehat A_k+\lambda I) = \lambda^{n-k}\prod_{i=1}^k(\lambda+\hat{\theta}_i) 
    \le \prod_{i=1}^n(\lambda+\theta_i) = \det(A+\lambda I).
  \end{equation*}
  Hence $\log\det(A+\lambda I)-\log\det(\widehat A_k+\lambda I)\ge 0$. Therefore, we obtain $\text{I}\le\frac{\zeta_k}{\lambda}$. For the term II, by the formulas for the exact and approximate posterior means, we have
  \begin{align*}
    x_\lambda-\widehat x_\lambda^{(k)}
    = \Sigma^{1/2}\Big[(A+\lambda I)^{-1}-(\widehat A_k+\lambda I)^{-1}\Big]\Sigma^{1/2}G^\top\Gamma^{-1}y .
  \end{align*} 
  Let $F_k=(A+\lambda I)^{-1}-(\widehat A_k+\lambda I)^{-1}$ and $\hat{y}=\Sigma^{1/2}G^\top\Gamma^{-1}y$. Using $C_{\lambda}^{-1}=\Sigma^{-1/2}(A+\lambda I)\Sigma^{-1/2}$, we have 
  \begin{equation*}
    \|x_\lambda-\widehat x_\lambda^{(k)}\|_{C_\lambda^{-1}}^2 
    = \hat{y}^{\top}F_{k}(A+\lambda I)F_{k}\hat{y}
    \leq \|F_{k}\|_2 \, \|(A+\lambda I)F_{k}\|_2 \, \|\hat{y}\|_2^2 .
  \end{equation*}
  Using $\Gamma^{-1}y=V_k B_k^\top \beta_1 e_1=\alpha_1\beta_1v_1$ and $v_1^{\top}Mv_{1}=1$, we get $\|\hat{y}\|_{2}^2=(\Gamma^{-1}y)^{\top}M(\Gamma^{-1}y)=\alpha_1^2\beta_1^2$. Using \Cref{lem:inv_bnd}, we get 
  \begin{equation*}
    \|F_k\|_2 = \lambda^{-1}\|(I+\lambda^{-1}A)^{-1}-(I+\lambda^{-1}\widehat{A}_k)^{-1}\|_2
    \leq \frac{1}{\lambda^2}\frac{\|E_k\|_2}{1+\lambda^{-1}\|E_k\|_2} \leq \frac{\gamma_k}{\lambda(\lambda+\gamma_k)},
  \end{equation*}
  where the last ``$\le$" uses that $\|E\|_2\le\|E\|_{F}=\gamma_k$ and $\frac{t}{1+\lambda^{-1}t}$ is monotonically increasing for $t\geq 0$. Noticing that 
  \begin{equation*}
    (A+\lambda I)F_k=(A+\lambda I)(A+\lambda I)^{-1}\Big[(\widehat A_k+\lambda I)-(A+\lambda I)\Big](\widehat{A}_k+\lambda I)^{-1} = -E_{k}(\widehat{A}_k+\lambda I)^{-1},
  \end{equation*}
  we get $\|(A+\lambda I)F_{k}\|_2 \leq \|E_k\|_2 \, \|(\widehat{A}_k+\lambda I)^{-1}\|_2
    \leq \frac{\gamma_k}{\lambda}$ since $\widehat{A}_k$ is SPD. Therefore, we obtain
    \begin{equation}\label{mean_err_bnd}
      \text{II} = \|x_\lambda-\widehat x_\lambda^{(k)}\|_{C_\lambda^{-1}}^2 
      \leq \frac{\alpha_1^2\beta_1^2\gamma_k^2}{\lambda^2(\lambda+\gamma_k)} .
    \end{equation}
    Combining the upper bounds for I and II, we finally obtain the desired bound.
\end{proof}

This result shows that the accuracy of the approximate posterior is bounded by the computable quantities $\zeta_k$ and $\gamma_k$, which can be updated iteratively via the recurrences \cref{zeta_gamma} at negligible cost. In practice, these bounds can be more reliable than the directly computed true errors in some cases, since the latter may be contaminated by numerical ill-conditioning of the covariance matrices. Hence, the bounds can serve not only as theoretical guarantees, but also as practical and reliable indicators of posterior approximation accuracy.

\section{Numerical experiments} \label{sec5}
We present three representative numerical examples of increasing scale and complexity to illustrate the proposed theory and demonstrate the performance of the method. From a one-dimensional integral equation to large-scale CT reconstruction, we show how the method adapts to different problem settings. In particular, the large-scale example highlights the effectiveness of the matrix-free formulation, where direct access to the posterior distribution is computationally prohibitive.

All numerical experiments are implemented in MATLAB R2025b, and our methods can be executed on a standard personal laptop. The source code is publicly available at \url{https://github.com/HaiboLi99/DS-Bayes}.

\subsection{One-dimensional Fredholm integral equation}
In the first example, we consider the one-dimensional Fredholm integral equation of the first kind:
\begin{equation}
  g(s) = \int_{a}^{b}\phi(s,t)x(t)\mathrm{d}t,  \quad \phi(s,t) = \exp\left(-|s-t|/l \right),
\end{equation}
where we set $a=-\pi/2$, $b=\pi/2$, and set $l=10$ in the integeral kernel $\phi$. For the desretization, we use $n=5000$ uniform grids on $[-\pi/2, \pi/2]$ and use $m=3000$ uniform grids on $[-\pi/2, \pi/2]$ for the observation. This lead to the forward matrix $G\in\mathbb{R}^{3000\times 5000}$. The unknown $x$ is modeled as a Gaussian process on $[-\pi/2,\pi/2]$ with covariance induced by the Gaussian kernel 
\begin{equation}
  K(t_1,t_2) = \sigma^2\exp\left(-\frac{(t_1-t_2)^2}{2l^2}\right),
\end{equation}
where $l$ and $\sigma^2$ denote the correlation length and marginal variance, respectively. In the experiment, we set $l=0.4$ and $\sigma=0.2$, and draw a random sample from this Gaussian process on the uniform grid as the ground truth $x_{\text{true}}$. We generate noisy data $y$ by adding a white Gaussian noise to $Gx_{\mathrm{true}}$ with noise level $\|\eta\|_2 / \|Gx_{\mathrm{true}}\|_2=0.005$. The true solution and noisy observation are shown in Figure \ref{fig4}.  

\begin{figure}[!htbp]
	\centering
	\subfloat 
	{\label{fig:1a}\includegraphics[width=0.45\textwidth]{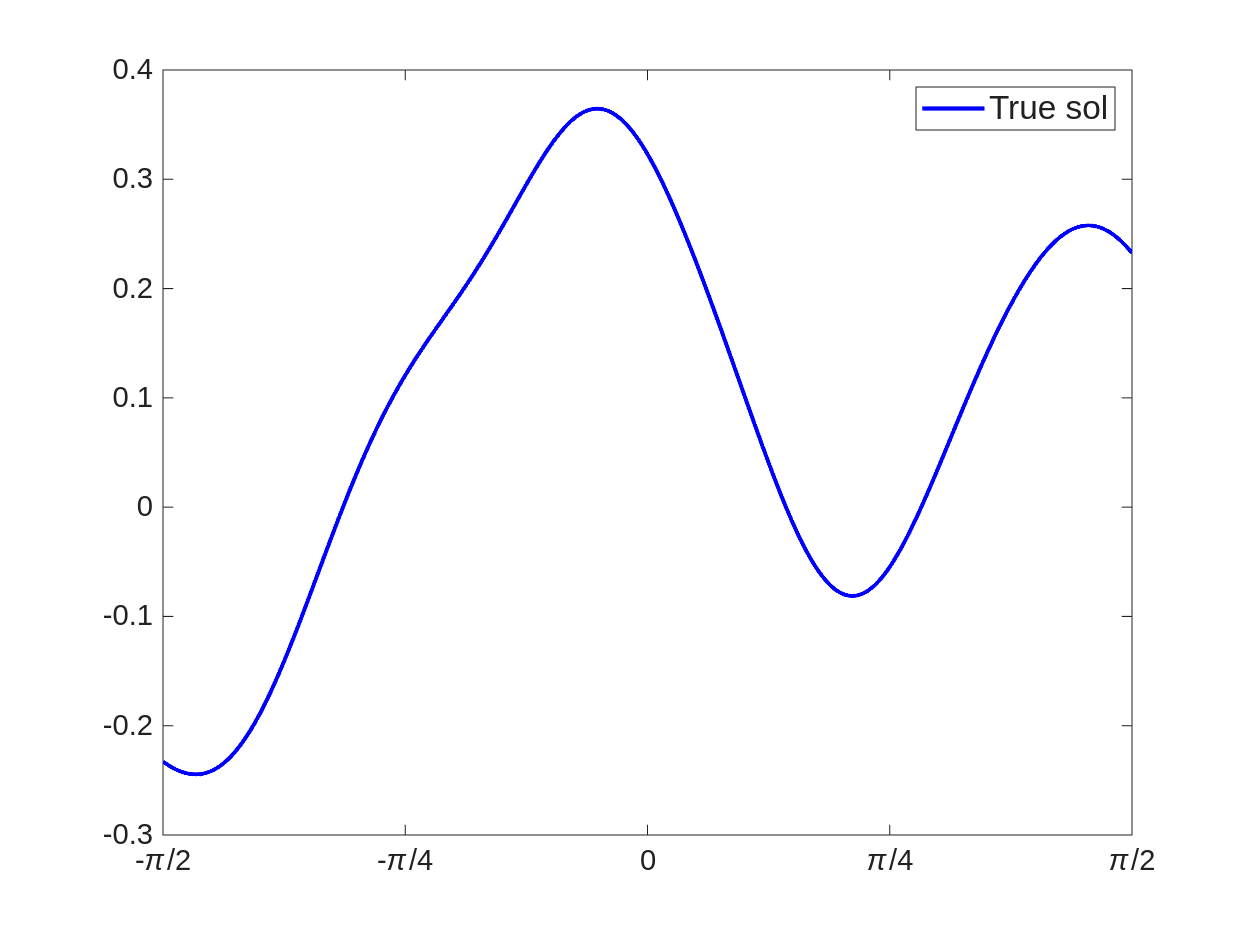}}\hspace{-5mm}
	\subfloat
	{\label{fig:1b}\includegraphics[width=0.45\textwidth]{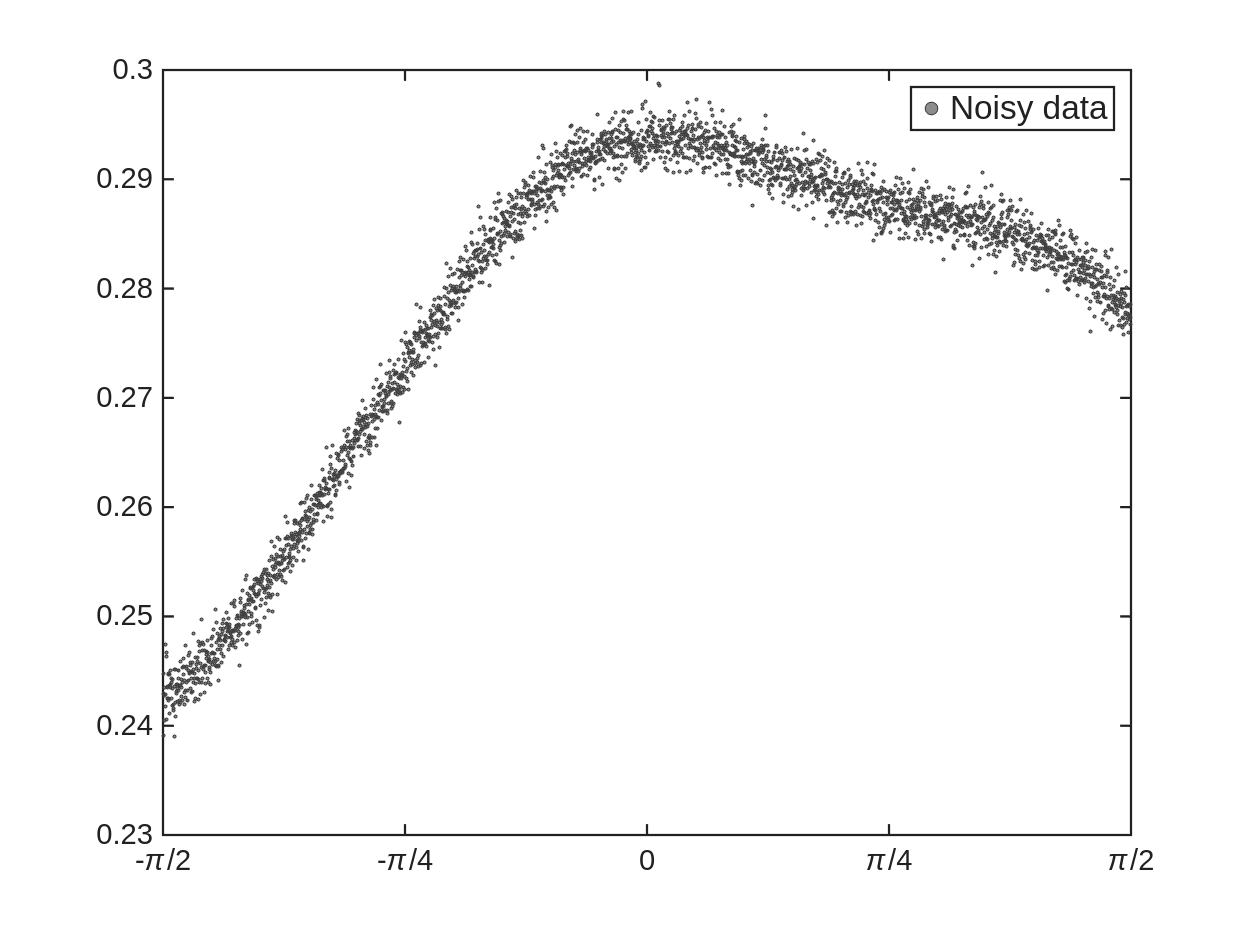}}\vspace{-2mm}
	\caption{Illustration of the true solution and noisy observation data.}
	\label{fig1}
\end{figure}

To reconstruct the unknown $x$, we construct the covariance matrix $\Sigma$ using the Gaussian kernel with $l=0.4$. The exact value of $\sigma$ is supposed to be unknown, hence we set $\sigma=1$ to form $\Sigma$. Given the relationship $\sigma^2 = 1/\lambda$, the Q-GKB based empirical Bayesian inference (QGKB-EB) method iteratively estimates the optimal hyperparameter $\lambda$, thereby obtaining an accurate estimate of $\sigma$, which is then be used to compute the posterior distribution.

\begin{figure}[!t]
	\centering
	\subfloat
	{\label{fig:2a}\includegraphics[width=0.45\textwidth]{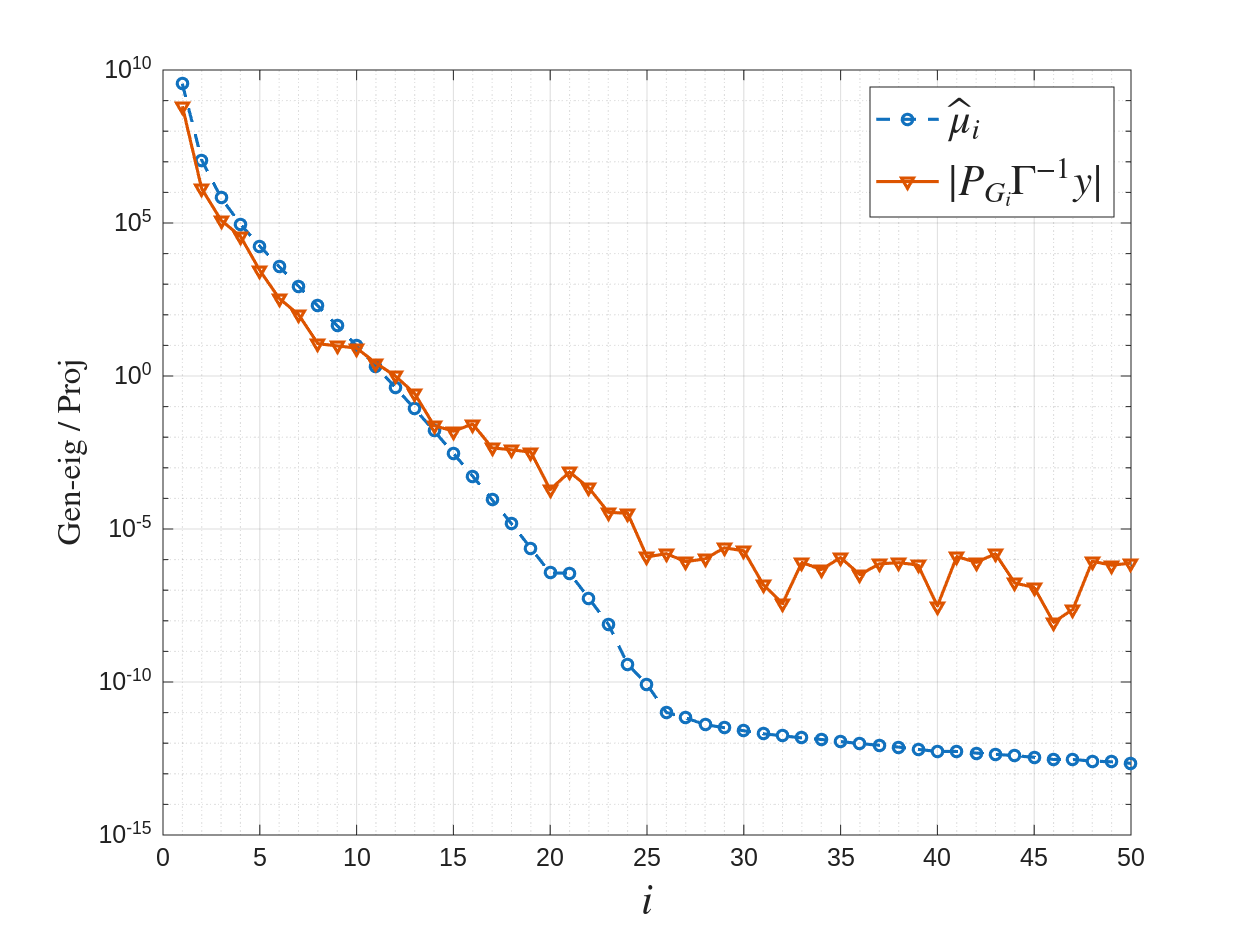}}\hspace{-5mm}
  \subfloat
	{\label{fig:2d}\includegraphics[width=0.45\textwidth]{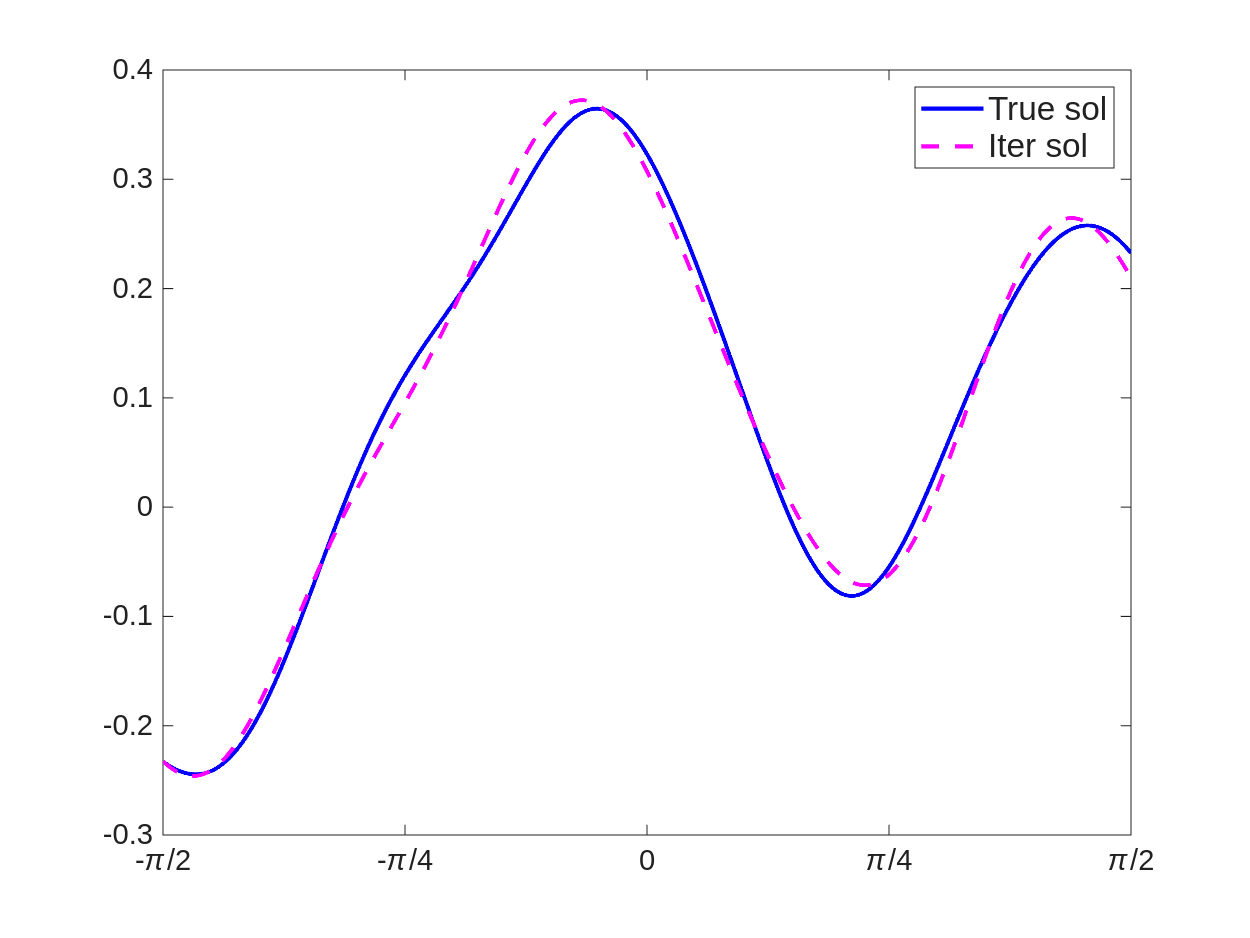}} \vspace{-4mm}
	\subfloat
	{\label{fig:2b}\includegraphics[width=0.45\textwidth]{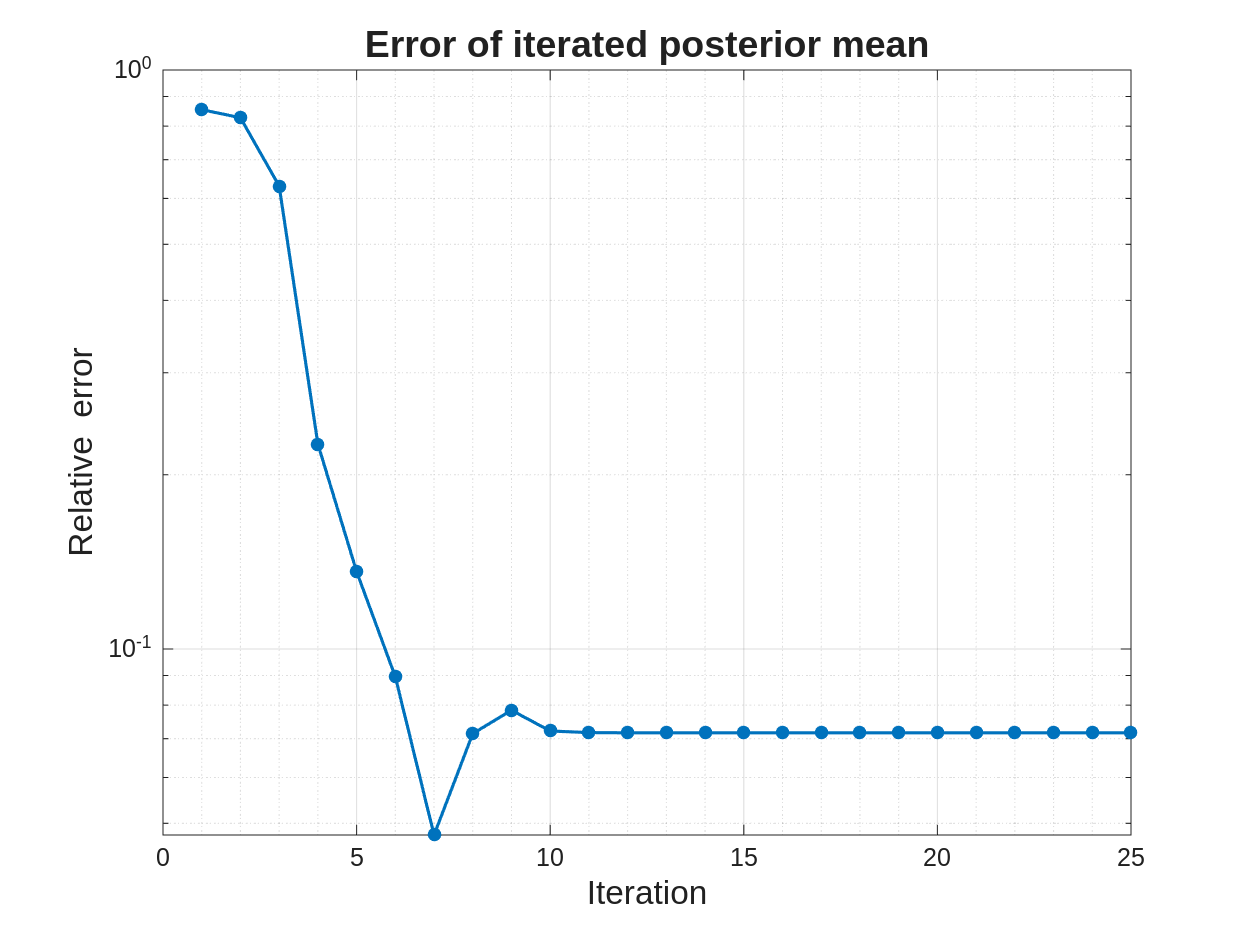}}\hspace{-5mm}
	\subfloat 
	{\label{fig:2c}\includegraphics[width=0.45\textwidth]{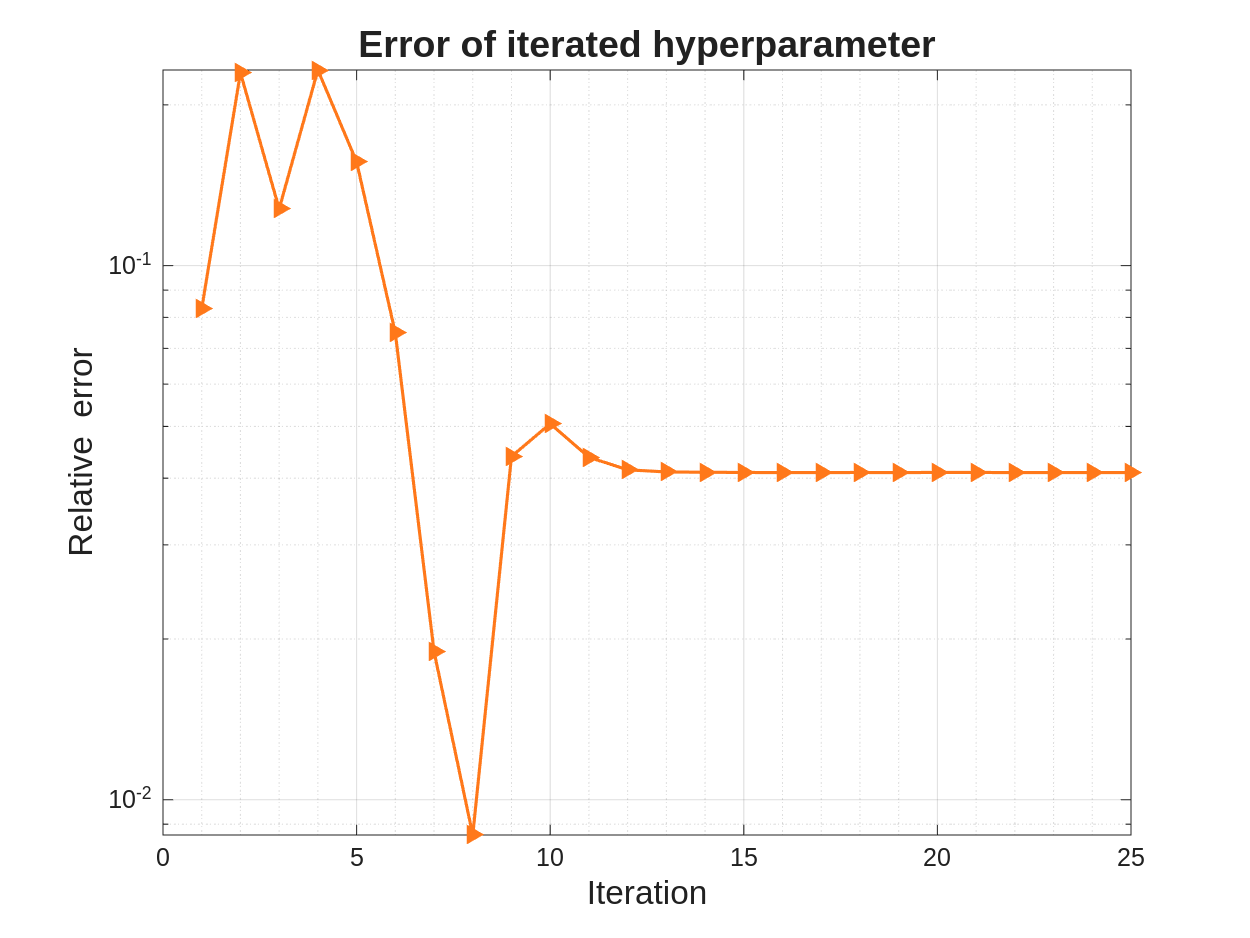}}
	\vspace{-2mm}
	\caption{Eigenvalue decay of $(M,\,\Gamma)$, reconstructed solution, and convergence behavior of the Q-GKB method for the first example.}
	\label{fig2}
\end{figure}

In \Cref{fig2}, we illustrate the eigenvalue decay of $(M,\,\Gamma)$ and the quantities $P_{\mathcal{G}_i}\Gamma^{-1}y$, which support \Cref{thm:qgkb-lanczos}. The relative error of the iterated hyperparameter is computed as $|\sigma_k-\sigma|/\sigma$, where $\sigma_k=1/\sqrt{\lambda_k}$. All generalized eigenvalues $\widehat{\mu}_i$ are simple and decay rapidly to zero without significant gaps, and all projected quantities $P_{\mathcal{G}_i}\Gamma^{-1}y$ are nonzero. Consequently, the Q-GKB method quickly capture all eigenspaces associated with nonzero generalized eigenvalues as $k$ increases. This behavior is due to the exponential decay of the eigenvalues of the forward operator induced by the integral kernel $\phi$. From the top-left subfigure, if we regard $\widehat{\mu}_i$ as negligible whenever it is smaller than $10^{-10}$, then by \Cref{thm:isom}, the effective dimension of the data space is approximately 25, which is much smaller than that of the parameter space. As a consequence, the bottom subfigures demonstrate rapid convergence of the Q-GKB iteration: the relative errors of both the reconstructed solution and the hyperparameter stabilize after approximately 15 iterations. The reconstructed solution at $k=25$ is shown to closely match the true solution.

\begin{figure}[!t]
	\centering
	\subfloat 
	{\label{fig:3a}\includegraphics[width=0.45\textwidth]{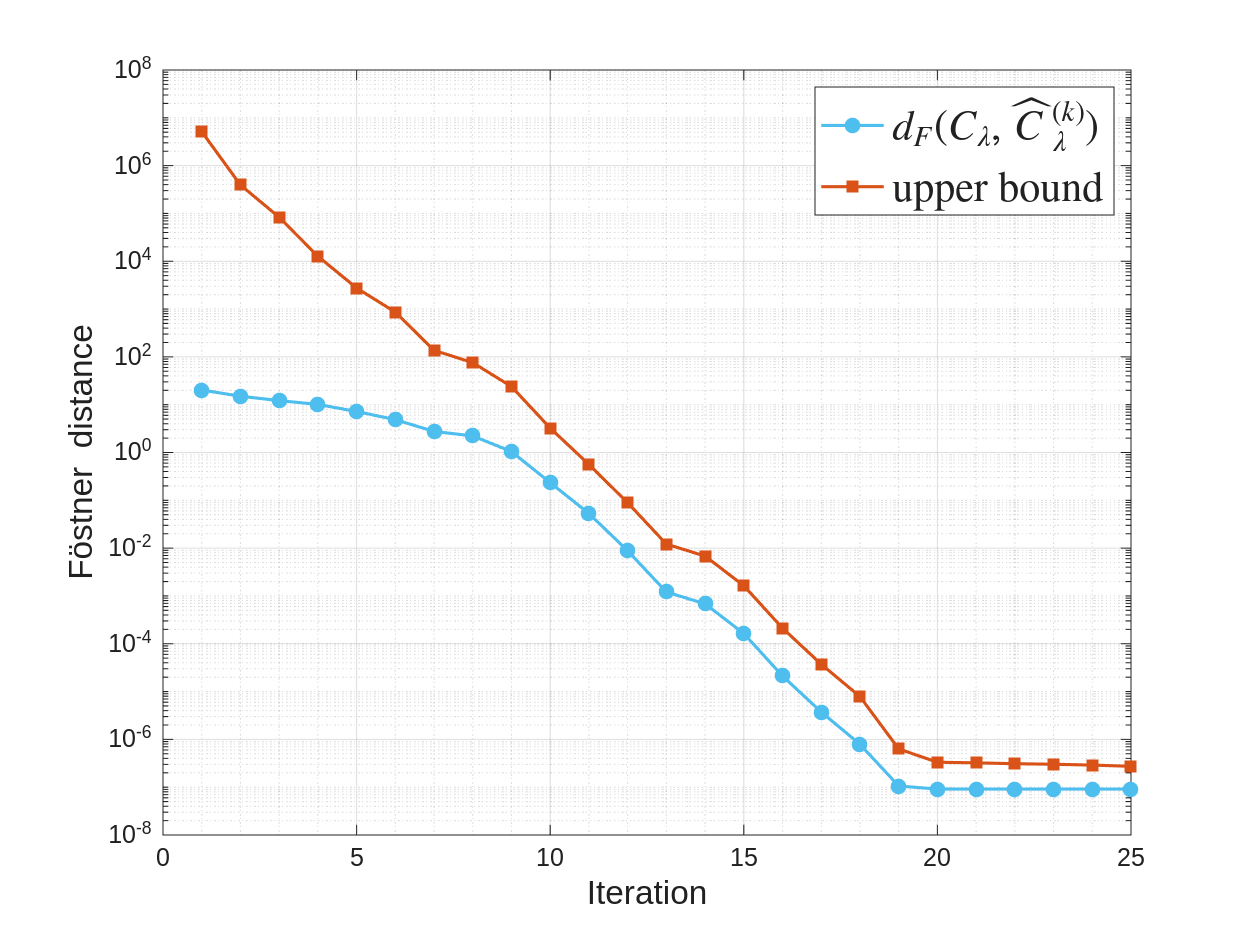}}\hspace{-5mm}
	\subfloat
	{\label{fig:3b}\includegraphics[width=0.45\textwidth]{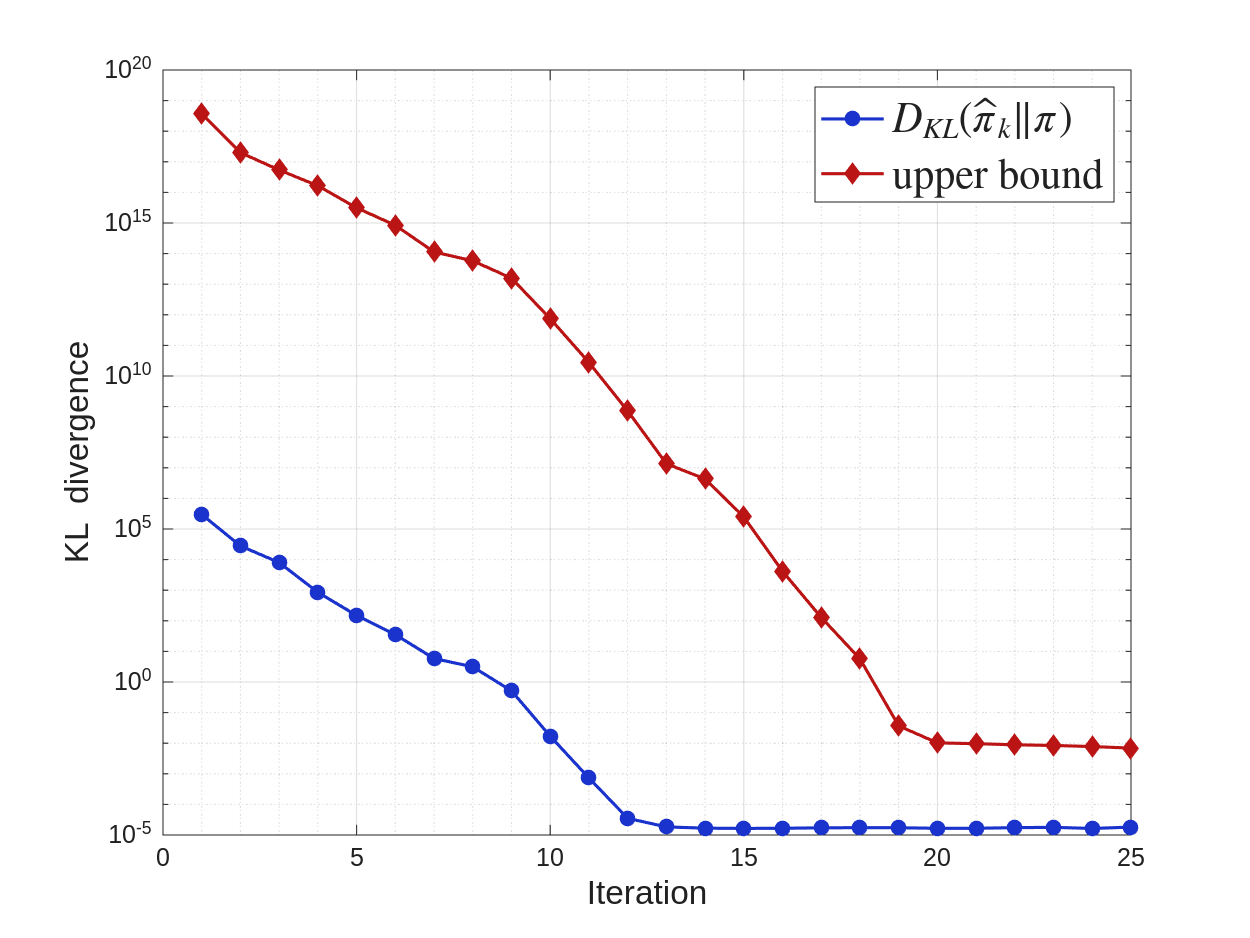}}\vspace{-2mm}
	\caption{The F{\"o}rstner distance and the KL divergence between the exact and approximate posterior, and their upper bounds.}
	\label{fig3}
\end{figure}

In \Cref{fig3}, we show the F{\"o}rstner distance $d_{F}(C_{\lambda},\widehat{C}_{\lambda}^{(k)})$ and the KL divergence $D_{\mathrm{KL}}(\widehat{\pi}_k,\pi)$ together with their upper bounds established in \Cref{thm:bnd_post}, where at each iteration we set $\lambda=\lambda_k$. As the Q-GKB iteration proceeds, both quantities and their upper bounds exhibit a clear decreasing trend. In the late stage of the iteration, the computed values of $d_{F}(C_{\lambda},\widehat{C}_{\lambda}^{(k)})$ and $D_{\mathrm{KL}}(\widehat{\pi}_k,\pi)$ appear to stagnate at a level around $10^{-5}$. This behavior is due to numerical instability in evaluating the exact posterior covariance, which becomes increasingly ill-conditioned. As a result, the computed F{\"o}rstner distance and KL divergence reach a saturation level determined by numerical precision. This saturation should be interpreted as a numerical precision barrier rather than a failure of the theoretical bounds. In contrast, the upper bounds derived from $\zeta_k$ and $\gamma_k$ continue to capture the true decay of the posterior approximation error, albeit at a very slow rate. In practice, these computable upper bounds can serve as reliable stopping criteria for iterative posterior approximation.

\begin{figure}[!htbp]
	\centering
	\subfloat 
	{\includegraphics[width=1\textwidth]{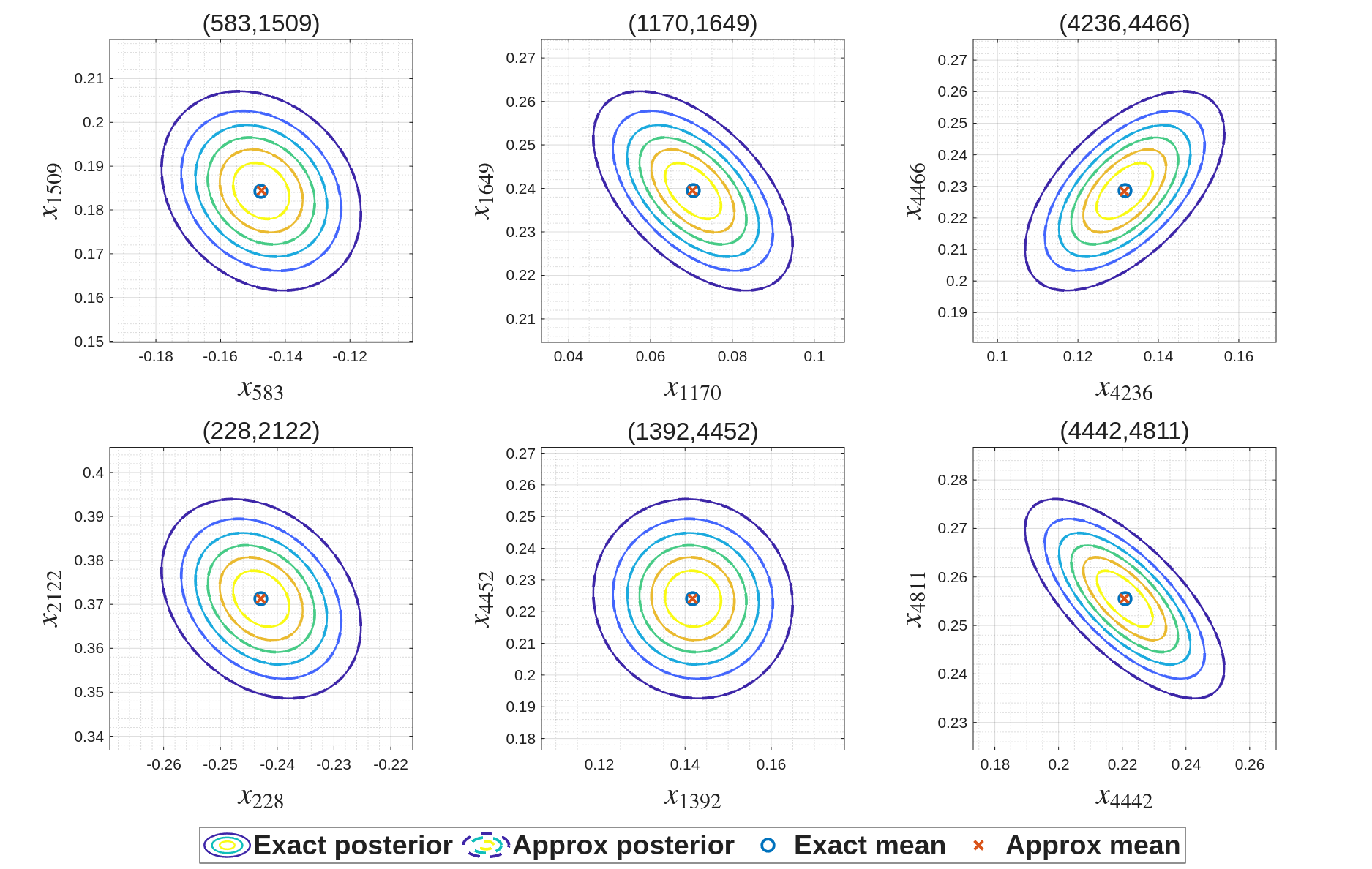}}
	\caption{Contour plots of the marginal distributions of six randomly selected $x_i$-$x_j$ pairs for both the exact and approximate posteriors.}
	\label{fig4}
\end{figure}

To further assess the quality of the posterior approximation at $k=25$, we compare the exact and approximate posteriors through their two-dimensional marginal distributions. In \Cref{fig4}, we depict contour plots of six randomly selected $x_i$-$x_j$ pairs. For the exact posterior, we set $\lambda=\lambda_{25}$ and compute its mean and covariance using \cref{post_mean1} and \cref{post_cov1}. The contours of the exact and approximate marginals are visually indistinguishable across all selected pairs, indicating an excellent agreement between the two distributions. This observation is consistent with the small KL divergence $D_{\mathrm{KL}}(\widehat{\pi}_k,\pi)$ at $k=25$. Indeed, an error on the order of $10^{-5}$ implies that the two distributions are nearly indistinguishable in practice. These results confirm that the Q-GKB approximation accurately captures the posterior structure in a low-dimensional data-informed Krylov subspace.

\subsection{Two-dimensional image deblurring}
In the second example, we consider a two-dimensional image deblurring problem on the domain $D = [0,1]^2$. Let $x : D \to \mathbb{R}$ denote the unknown image, and $b : D \to \mathbb{R}$ denote the observed blurred data. The forward model is given by the two-dimensional Fredholm integral equation of the first kind:
\begin{equation}\label{eq:forward_model}
    b(\boldsymbol{s}) = 
    \iint_D f_{\mathrm{psf}}(\boldsymbol{s},\boldsymbol{t})\, x(\boldsymbol{t})\,\mathrm{d}t, 
    \quad \boldsymbol{s} \in D,
\end{equation}
where $f_{\mathrm{psf}}$ is a Gaussian point spread function (PSF) defined as
\begin{equation}\label{eq:psf}
    f_{\mathrm{psf}}(s,t)
    = \exp\!\left(-\frac{\|\boldsymbol{s} - \boldsymbol{t}\|_2^2}{l_{\mathrm{blur}}}\right),
\end{equation}
with $l_{\mathrm{blur}} > 0$ controlling the size of the blur spreading. In the experiment we set $l_{\mathrm{blur}}=0.01$. We discretize the domain $D$ using $n_1$ uniform grids on each dimension such that the true image has $n_1\times n_1$ pixels; the blurred observations are collected on a coarser $m_1 \times m_1$ grid, where we take $n_1 = 256$ and $m_1 = 128$. This leads to a discrete forward matrix $G \in \mathbb{R}^{16384 \times 65536}$.

The unknown image $x$ is modeled as a two-dimensional Gaussian random field with a Matérn-type covariance structure. We adopt a separable Matérn kernel of the form
\begin{equation}\label{eq:separable_matern}
    \mathrm{Cov}(x(\boldsymbol{s}), x(\boldsymbol{t}))
    = K_1(s_1,t_1)\, K_1(s_2,t_2), \quad \boldsymbol{s}=(s_1,s_2),\; \boldsymbol{t}=(t_1,t_2).
\end{equation}
Here, $K_1$ denotes the one-dimensional Matérn kernel
\begin{equation}\label{eq:matern1d}
    K_1(s,t)
    = \sigma^2 \frac{2^{1-\nu}}{\Gamma(\nu)}\left( \frac{\sqrt{2\nu}\, |s-t|}{\rho} \right)^{\nu}
    B_{\nu}\!\left( \frac{\sqrt{2\nu}\, |s-t|}{\rho} \right)
\end{equation}
with smoothness parameter $\nu > 0$, correlation length $\rho > 0$, and marginal variance $\sigma^2$, where $\Gamma$ is the gamma function and $B_{\nu}$ is the modified Bessel function of the second kind \cite{Williams2006gaussian}. In the experiment we set $\nu = 3$, $\rho = 0.1$, and $\sigma = 1$, and a random sample drawn from this Gaussian random field on the $n_1 \times n_1$ grid is used as the ground truth. A white Gaussian noise with relative noise level $0.01$ is then added to the exact observation $G x_{\mathrm{true}}$ to generate the noisy data. The true image and the corresponding noisy observation are shown in~\Cref{fig5}.

\begin{figure}[htbp]
	\centering
	\subfloat 
	{\includegraphics[width=0.8\textwidth]{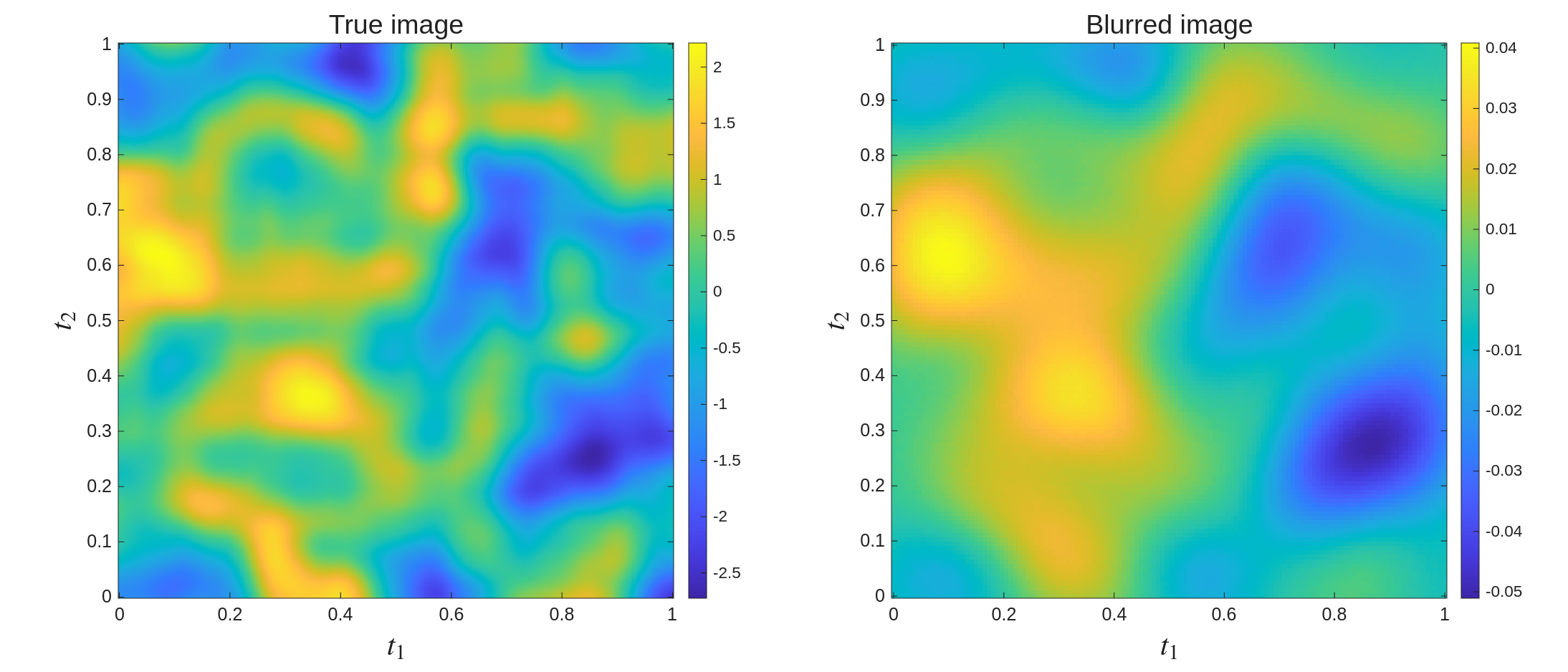}}\vspace{-2mm}
	\caption{Illustration of the true image and noisy blurred image.}
  \label{fig5}
\end{figure}

Similar to the first example, the prior covariance matrix $\Sigma$ is constructed using the separable Matérn kernel with $\nu = 3$, $\rho = 0.1$, and $\sigma = 1$. This leads to Kronecker structure $\Sigma = \Sigma_1 \otimes \Sigma_1$, where $\Sigma_1 \in \mathbb{R}^{256 \times 256}$ is the covariance matrix induced by $K_1$ on a uniform grid of 256 points in $[0, 1]$. While $\Sigma$ is initially formed with $\sigma = 1$, the marginal variance is treated as unknown and is iteratively estimated via the QGKB-EB method.

\begin{figure}[!t]
	\centering
  \subfloat 
	{\label{fig:6a0}\includegraphics[width=0.35\textwidth]{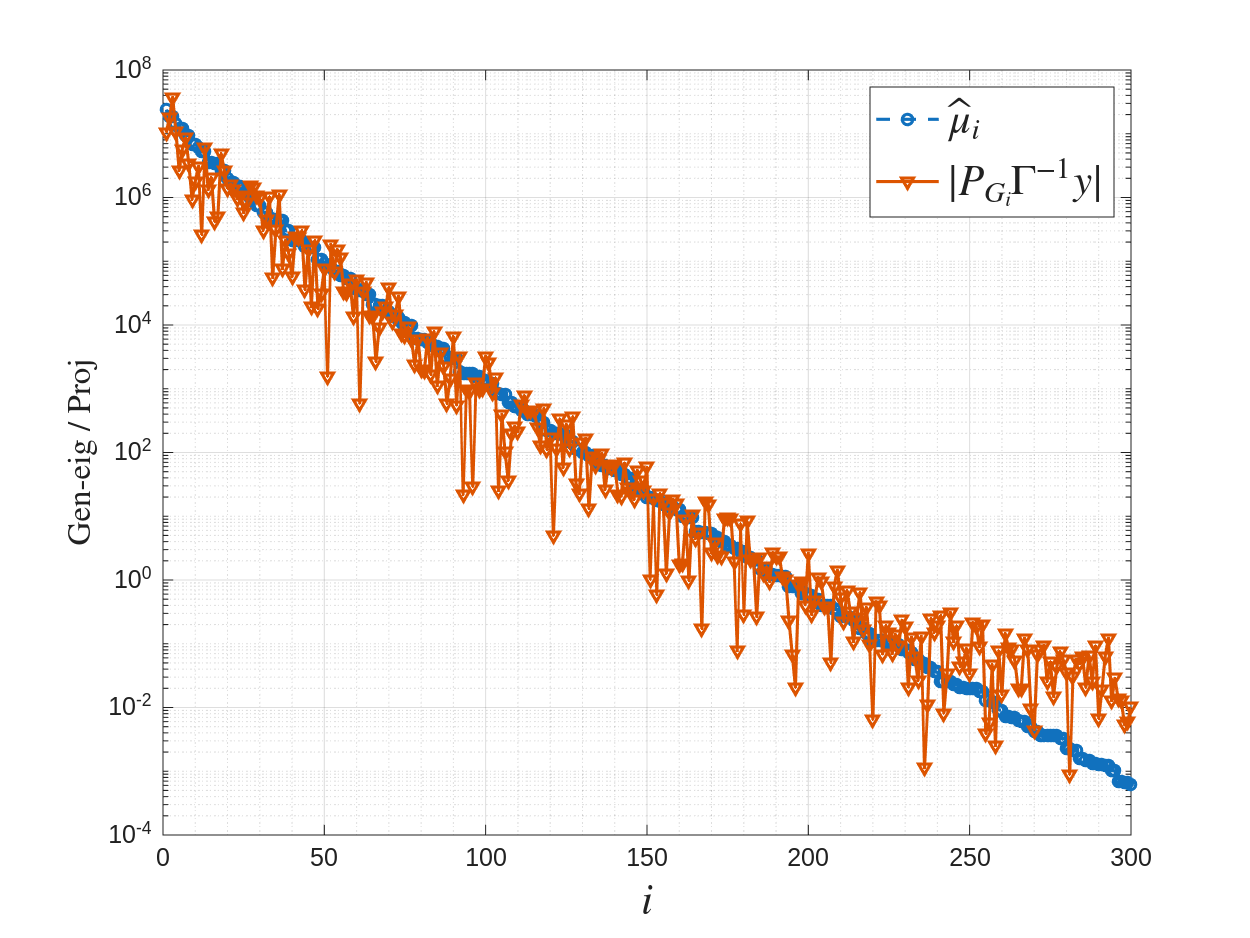}}\hspace{-5mm}
	\subfloat 
	{\label{fig:6a}\includegraphics[width=0.35\textwidth]{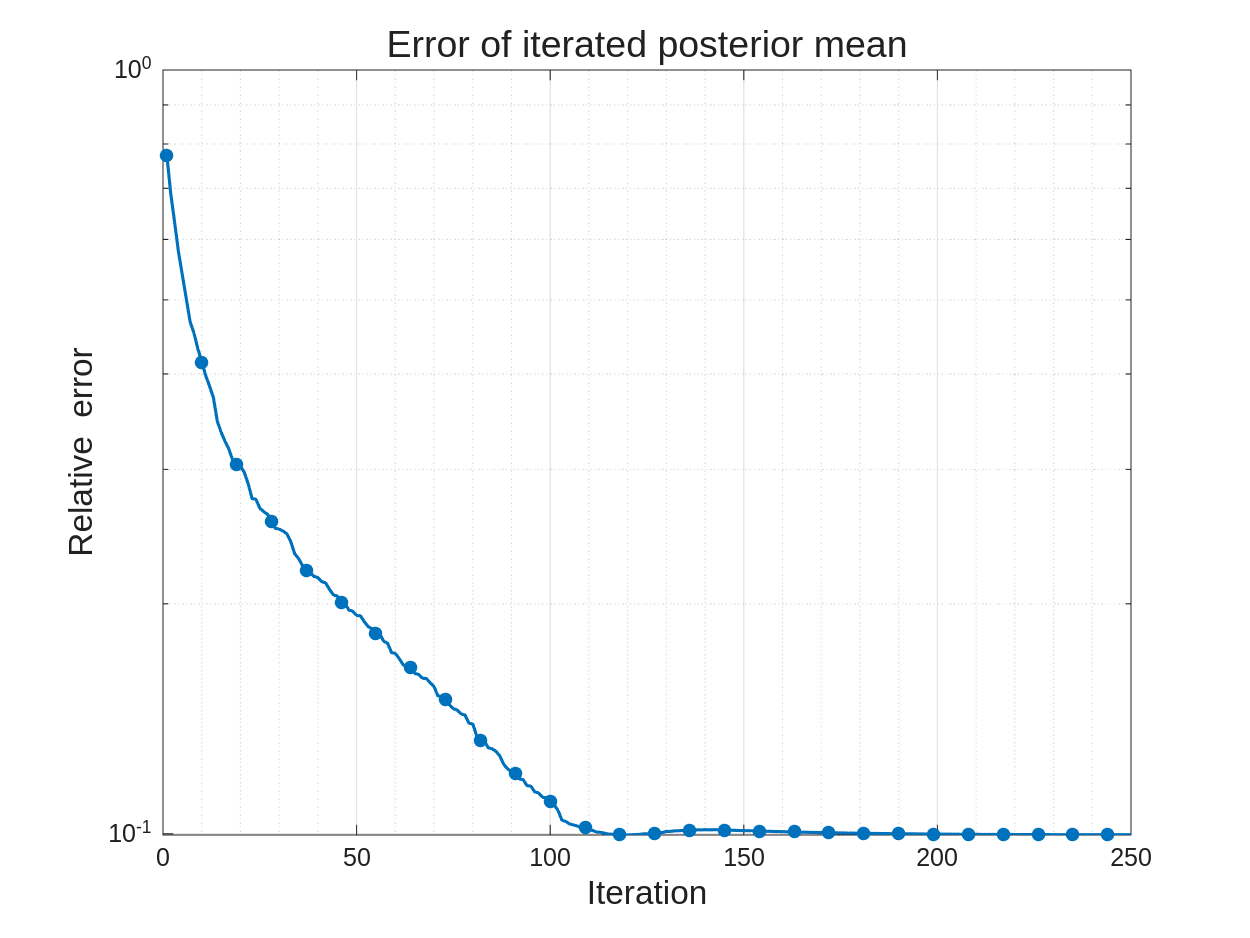}}\hspace{-5mm}
	\subfloat
	{\label{fig:6b}\includegraphics[width=0.35\textwidth]{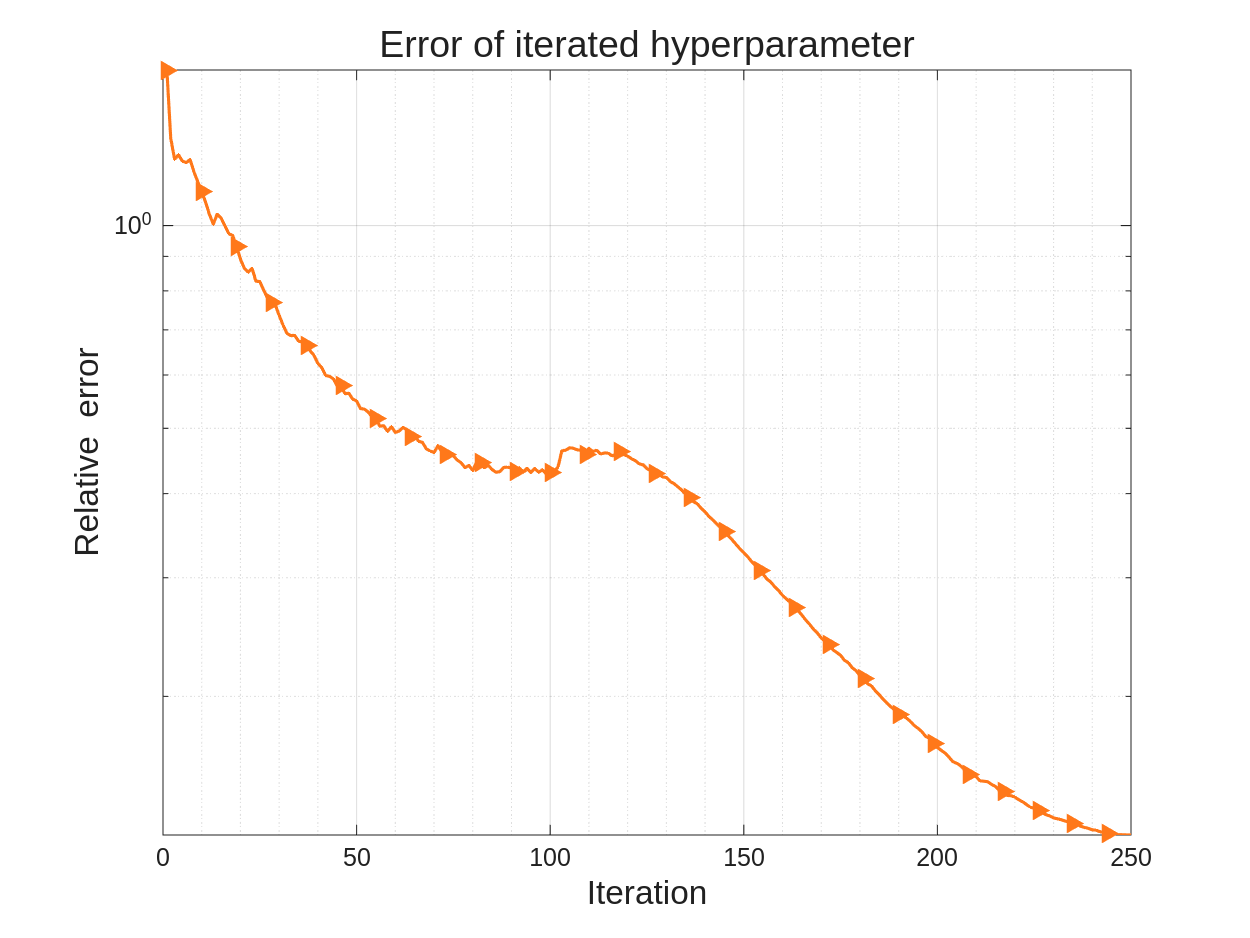}}\vspace{-2mm}
	\caption{Eigenvalue decay of $(M,\,\Gamma)$ and convergence behavior of the Q-GKB method for the first example.}
	\label{fig6}
\end{figure}

In \Cref{fig6}, we present the eigenvalue decay of the matrix pair $(M,\,\Gamma)$ and the projected quantities $P_{\mathcal{G}_i}\Gamma^{-1}y$ for the second example. We find that all generalized eigenvalues $\widehat{\mu}_{i}$ have a multiplicity of one and decay gradually toward zero without significant spectral gaps. Furthermore, since all projected quantities $P_{\mathcal{G}_i}\Gamma^{-1}y$ are non-zero, the Q-GKB method is capable of capturing all eigenspaces associated with nonzero generalized eigenvalues as $k$ increases. With $l_{\text{blur}}=0.01$, this example is a moderately ill-posed problem, which has a relatively slower decaying rate of the eigenvalues of the forward operator compared to the first example. Consequently, the generalized eigenvalues of $(M,\,\Gamma)$ decay at a moderate rate, resulting in an effective data-space dimension of several hundreds. Accordingly, the convergence of the relative errors for the reconstructed solution stagnates after about 150 iterations, while the hyperparameter continues to evolve before eventually stabilizing after about 250 iterations.

\begin{figure}[!t]
	\centering
	\subfloat 
  {\includegraphics[width=1.0\textwidth]{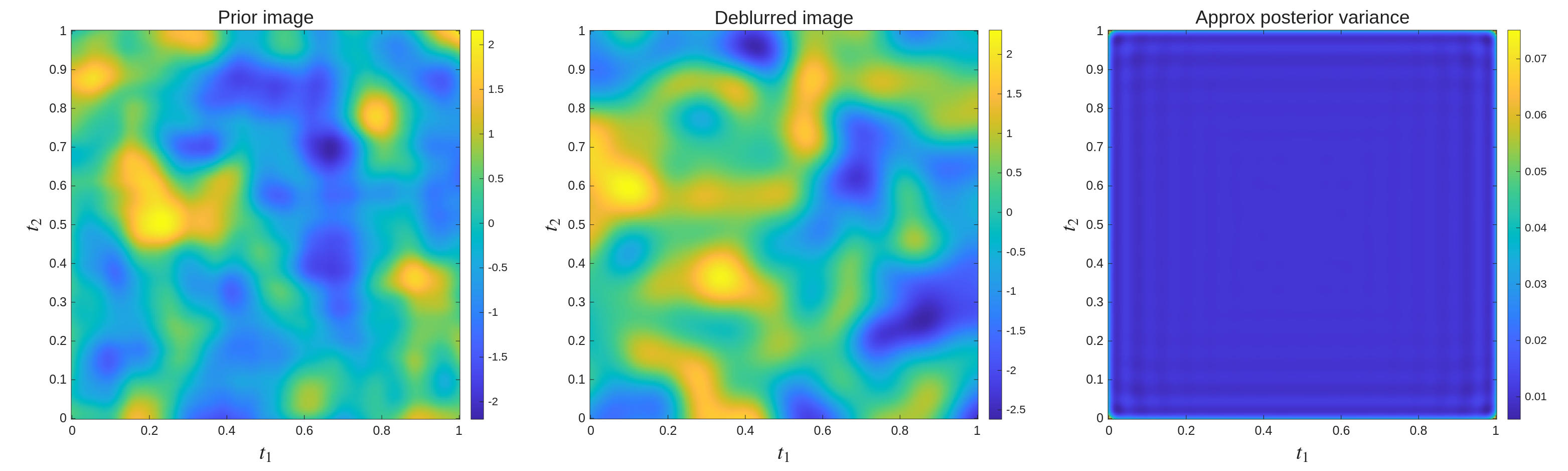}} \\
  \vspace{-5mm}
  \subfloat 
	{\includegraphics[width=0.72\textwidth]{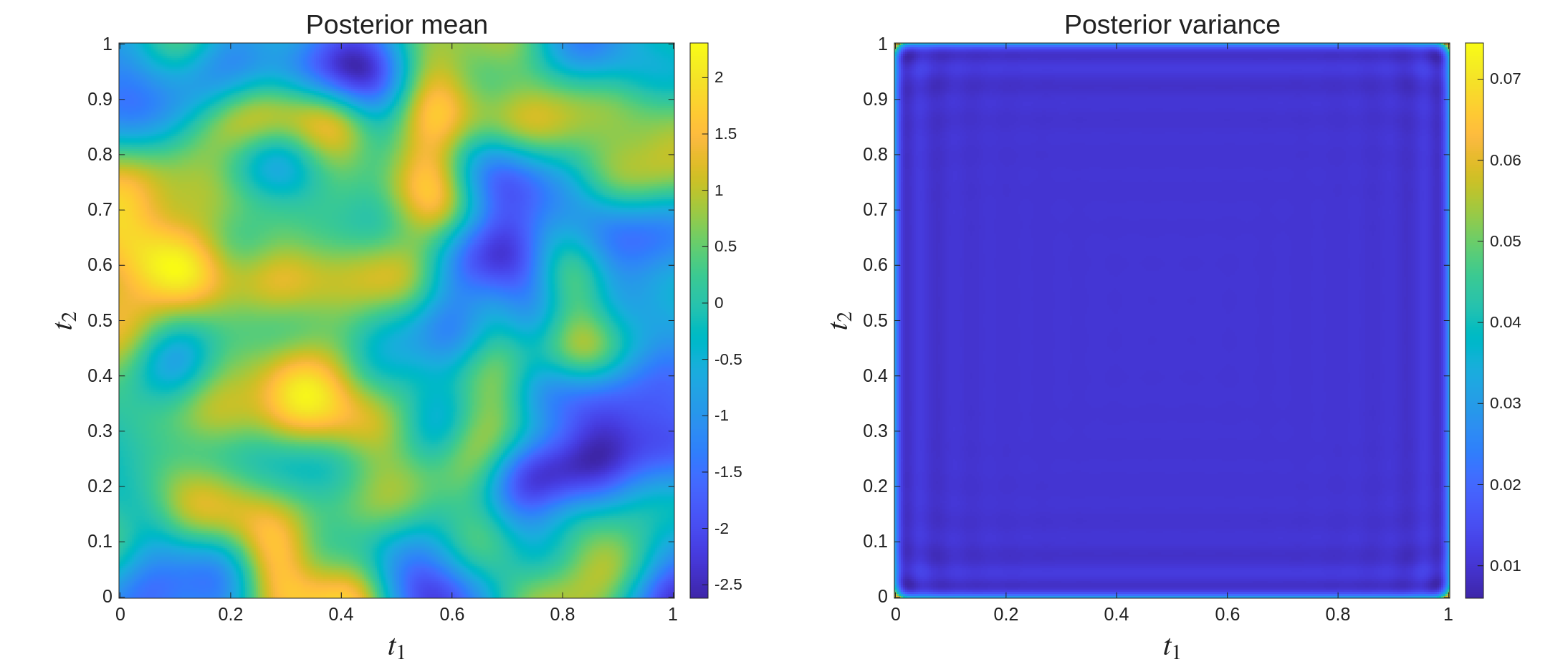}}\vspace{-2mm}
	\caption{Illustration of the prior image, deblurred image at $k=250$ and the image associated with exact posterior mean, and the corresponding variance of each pixel.}
	\label{fig7}
\end{figure}

We use \Cref{fig7} to illustrate the effectiveness of the QGKB-EB method for the image deblurring problem. For this large-scale case, computing the KL divergence to quantify approximation discrepancy is computationally prohibitive; instead, we examine the pixel-wise variance, represented by the diagonal elements of the posterior covariance matrix. The prior image shown is a random sample drawn from the prior distribution. The deblurred image and its corresponding variance represent the approximate posterior mean and covariance computed at $k=250$. For comparison, the exact posterior mean and variance are computed using $\lambda = \lambda_{250}$ according to \eqref{post_mean1} and \eqref{post_cov1}. Comparing the approximate and exact posteriors, we observe that the reconstructed image progressively converges to an accurate approximation of the posterior mean as the iteration proceeds. Furthermore, the approximate variance aligns closely with the exact posterior variance, thereby demonstrating the effectiveness of the Q-GKB-EB method for posterior approximation within the data-space formulation.

\begin{figure}[htbp]
	\centering
	\subfloat 
	{\label{fig:66a}\includegraphics[width=0.48\textwidth]{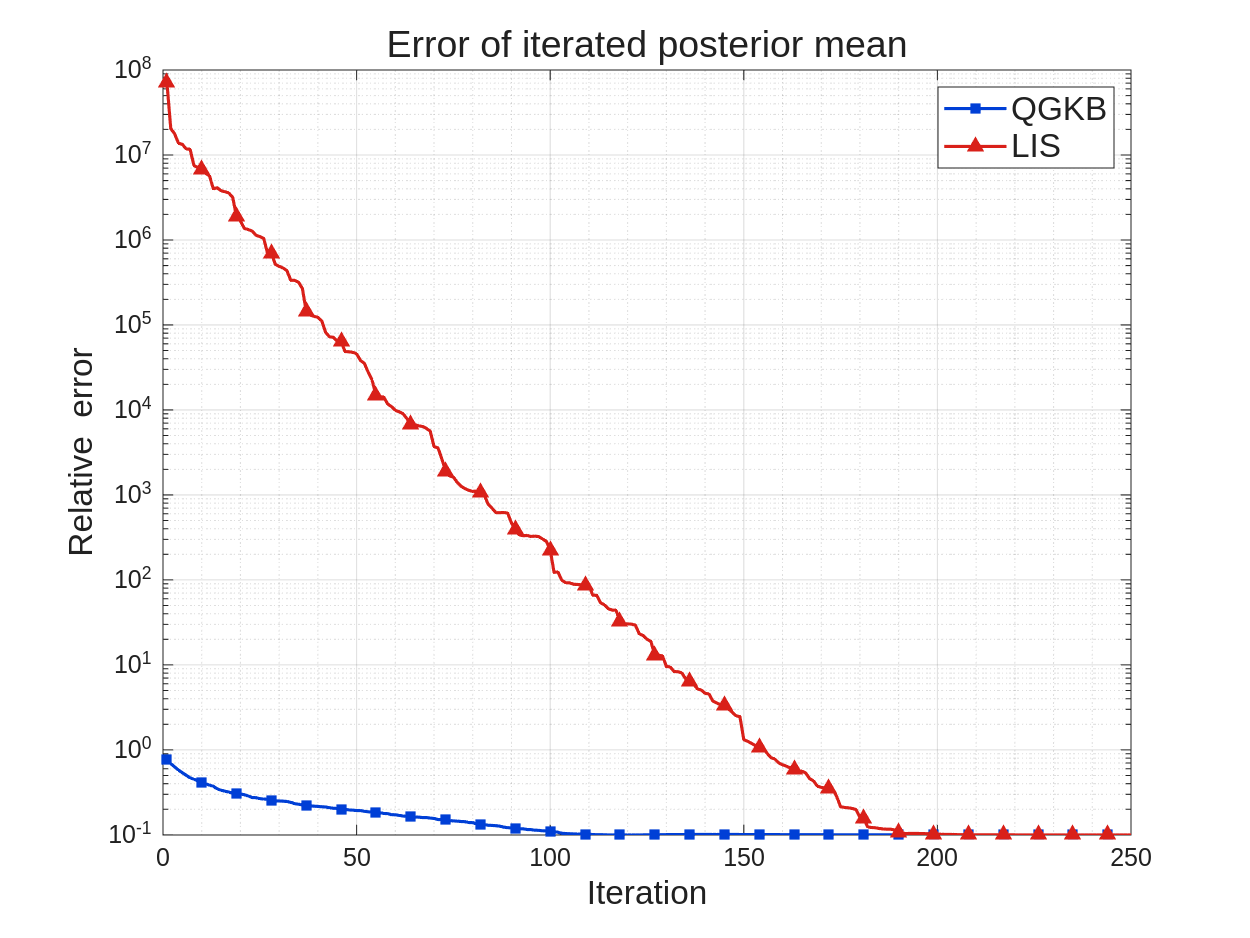}}\hspace{-5mm}
	\subfloat
	{\label{fig:66b}\includegraphics[width=0.48\textwidth]{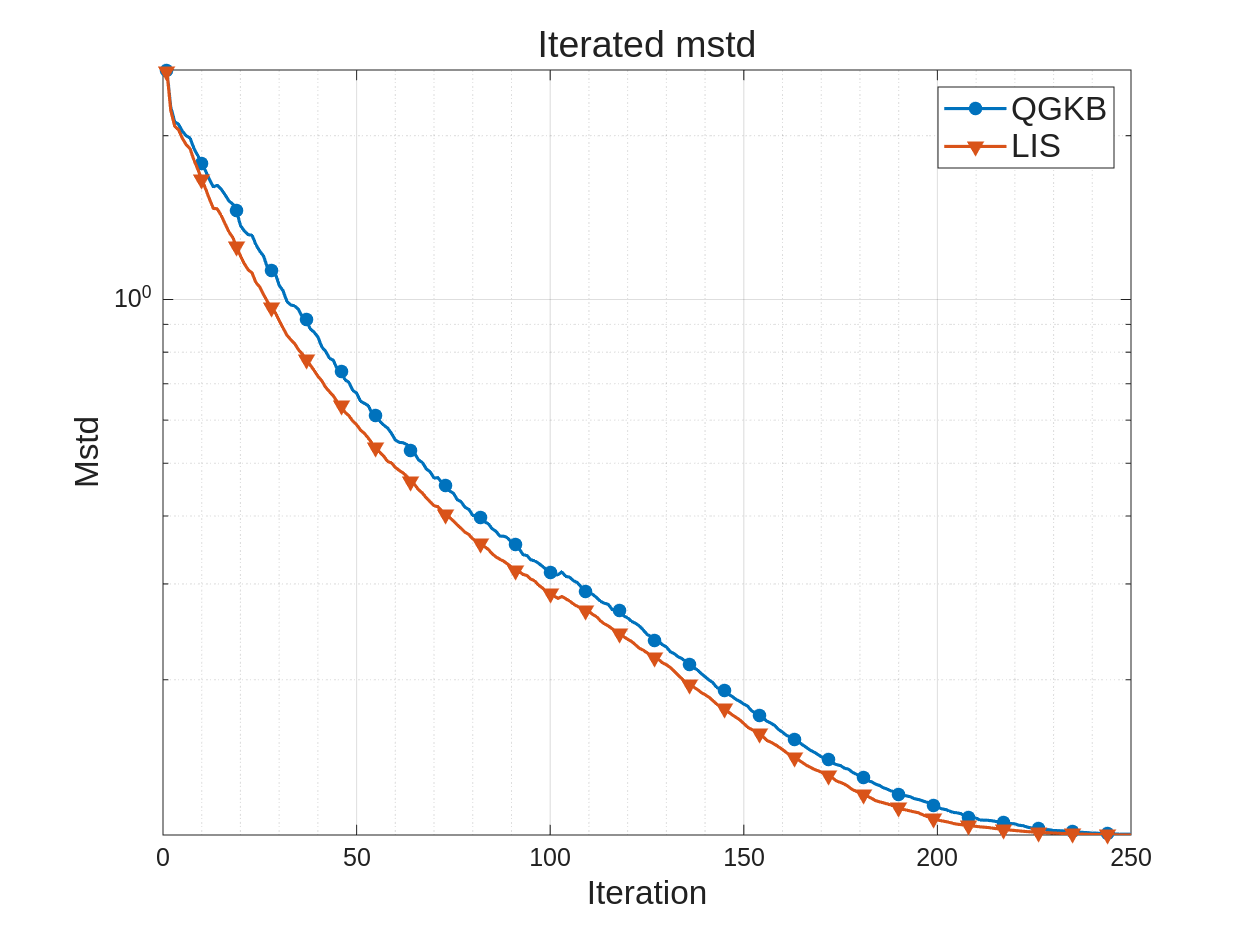}}\vspace{-2mm}
	\caption{The relative error of the approximate posterior mean and the corresponding mean standard variation computed by the Q-GKB method and LIS method.}
	\label{fig:comp}
\end{figure}

\begin{figure}[!t]
	\centering
	\subfloat 
	{\includegraphics[width=0.6\textwidth]{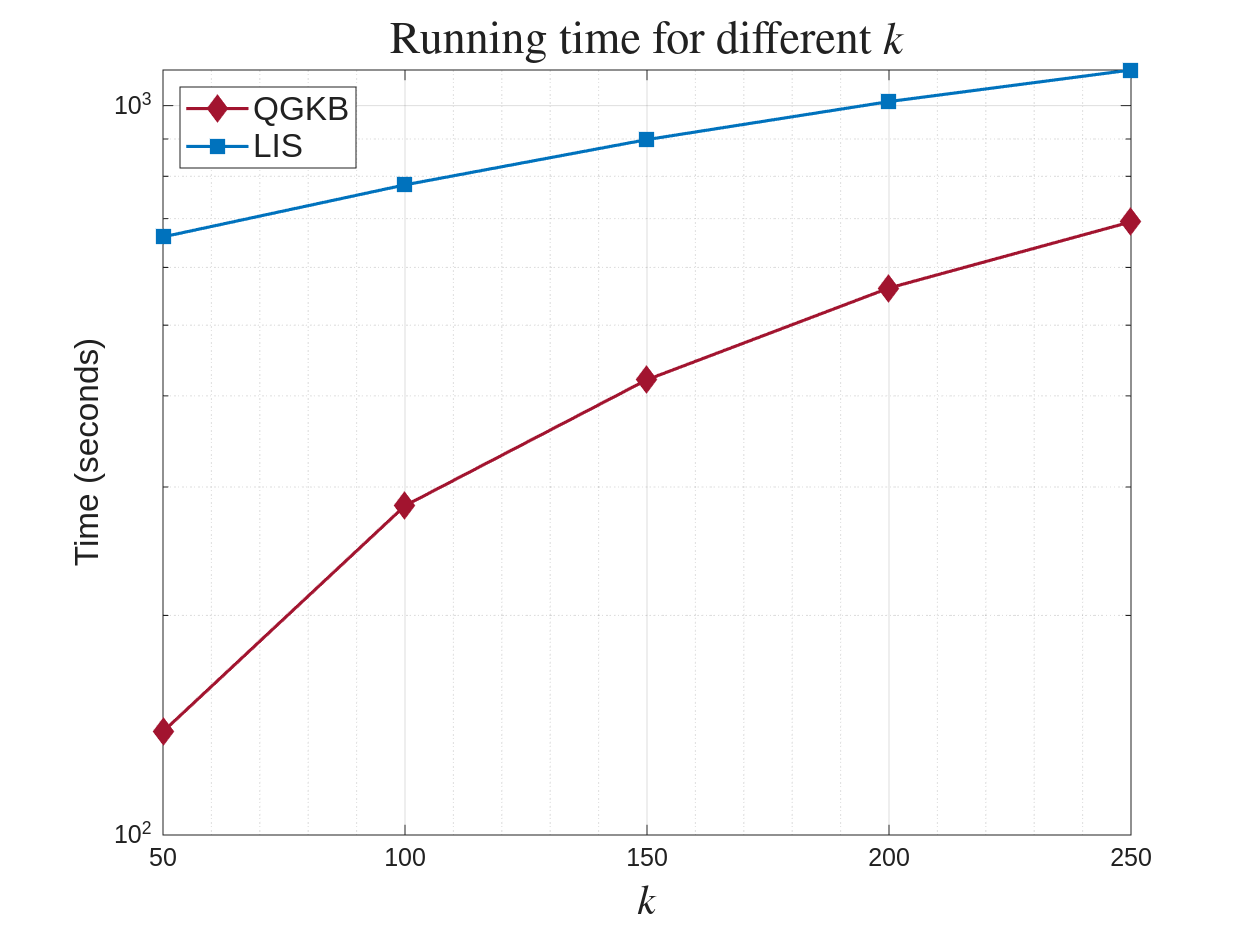}}\vspace{-2mm}
	\caption{Comparison of the running time for the Q-GKB method and LIS method as the iteration number or LIS dimension increases.}
	\label{fig:time}
\end{figure}

To further evaluate approximation accuracy and computational efficiency, we compare the proposed Q-GKB method with the direct LIS approach, which serves as a baseline. As shown in \cite{spantini2015optimal,spantini2017goal}, the LIS method provides an optimal low-rank approximation of the posterior covariance when the prior-to-posterior update is restricted to positive semidefinite matrices of rank at most $k$. For a fair comparison, the LIS method uses the same hyperparameter $\lambda_k$ estimated by the QGKB-EB approach at iteration $k$. To quantify the convergence behavior of the covariance approximation, we define the mean standard deviation (mstd) as
\begin{equation}\label{mstd}
  \mathrm{mstd} = \sqrt{\frac{1}{n}\sum_{i=1}^{n}\mathrm{var}_i} \ ,
\end{equation}
where $\mathrm{var}_i$ denotes the posterior variance at the $i$th pixel. To compute a $k$-dimensional LIS, we follow \cite{spantini2015optimal}. Specifically, we first compute the Cholesky factorization $\Sigma = SS^\top$, and then compute the leading $k$ eigenpairs of $S^\top H S$ using the MATLAB built-in function \texttt{eigs}, which yields the dominant LIS directions. From \Cref{fig:comp}, we have two important observations: first, in terms of the posterior mean, the Q-GKB method converges significantly faster than the LIS approach; second, for the covariance approximation measured by mstd, the LIS method exhibits slightly faster convergence. This difference can be explained by the intrinsic mechanisms of the two approaches: the Q-GKB iteration, initialized from the data vector $y$, prioritizes directions that are most relevant for reconstructing the posterior mean; in contrast, the LIS method targets the globally dominant eigendirections of the data-informed subspace, leading to a more efficient approximation of the posterior covariance. Nevertheless, as $k$ becomes sufficiently large to capture the dominant data-informed subspaces, both methods achieve nearly identical accuracy in approximating the posterior distribution. This observation is consistent with the theoretical characterization of data-informed Krylov subspaces and their relation to the LIS method.

Before comparing the running time of the two methods, it is important to emphasize that for a $65536 \times 65536$ dense prior covariance $\Sigma$, the LIS method is computationally prohibitive on a standard personal laptop due to the memory requirements for storing $\Sigma$, let alone computing its Cholesky factorization. Thus, for a fair comparison of the running time, we run both the LIS and Q-GKB method in a high performance computing (HPC) cluster. As shown in Figure~\ref{fig:time}, the total wall-clock time of Q-GKB remains significantly lower than that of LIS across all tested subspace dimensions $k$, even when using MATLAB's highly optimized \texttt{eigs} implementation. The dominant computational cost of LIS arises from the Cholesky factorization of $\Sigma$, which scales as $\mathcal{O}(n^3)$ and becomes prohibitively expensive at this resolution.

\subsection{X-ray computed tomography}
In the third example, we consider a two-dimensional X-ray computed tomography (CT) problem, where the goal is to reconstruct an unknown attenuation function $f: \Omega \to \mathbb{R}$ supported on a domain $\Omega \subset \mathbb{R}^2$ from a finite collection of line-integral measurements corrupted by noise. The mathematical model of X-ray CT is described by the Beer--Lambert law: when a 
monochromatic X-ray beam of initial intensity $I_0$ traverses a medium along a line $\ell$, the transmitted intensity satisfies
\begin{equation}
    I = I_0 \exp\!\left( -\int_{\ell} f(\boldsymbol{x})\, \mathrm{d}s \right),
\end{equation}
so that the measured log-attenuation along $\ell$ equals the line integral of $f$. The corresponding measurement is thus $-\log(I/I_0)$. The forward operator can be described by the Radon transform
\begin{equation}\label{eq:radon}
  \mathcal{R}[f](\phi, s) 
   = \int_{-\infty}^{\infty} f\!\left( s \boldsymbol{\theta} + t \boldsymbol{\theta}^\perp \right) \mathrm{d}t, 
   \qquad \phi \in [0, \pi),\quad s \in \mathbb{R}, 
\end{equation}
where $\boldsymbol{\theta} = (\cos\phi, \sin\phi)$ is the unit vector in the projection direction and $\boldsymbol{\theta}^\perp = (-\sin\phi, \cos\phi)$ is its orthogonal complement. The pair $(\phi,s)$ parametrizes all oriented lines in $\mathbb{R}^2$, where $\phi$ is the projection angle and $s$ is the signed distance from the origin to the line. We refer to \cite{Buzug2008,natterer2001mathematics} for further details.

The discrete linear system is constructed by discretizing \cref{eq:radon} on a uniform $N \times N$ pixel grid of $\Omega=[-1,1]^2$, yielding a vector $x \in \mathbb{R}^n$ with $n = N^2$, where the $j$th element $x_j$ represents the attenuation coefficient in the $j$th pixel. Measurements are collected at $n_\phi$ equally spaced projection angles $\phi_i \in [0, \pi)$ and $n_s$ the the number of rays per angle, giving $m = n_\phi n_s$ measurements assembled into a vector belonging to $\mathbb{R}^m$. The forward matrix $G \in \mathbb{R}^{m \times n}$ has entries $G_{ij} = \mathrm{length}\!\left( \mathrm{ray}_i \cap \mathrm{pixel}_j \right)$, i.e., the intersection length of the $i$th ray with the $j$th pixel, computed via the standard ray-driven projection model \cite{natterer2001mathematics}. We set $N=256$, and choose $\{\phi_i\}_{i=1}^{n_\phi} = \{\ang{0},\ang{2}, \ang{4} \dots, \ang{178}\}$ with $n_s = \mathrm{round}(\sqrt{2}\,N)$, where $\mathrm{round}(a)$ denotes the nearest integer to $a$. With this choice, we have $G\in\mathbb{R}^{32580 \times 65536}$. The ground-truth $x_{\mathrm{true}}$ is taken from \cite{Gazzola2019}, and we generate noisy data $y$ by adding a white Gaussian noise to $Gx_{\mathrm{true}}$ with noise level being $0.002$. The true image and the corresponding noisy observation are shown in~\Cref{fig9}.

\begin{figure}[!t]
	\centering
	\subfloat 
	{\includegraphics[width=0.8\textwidth]{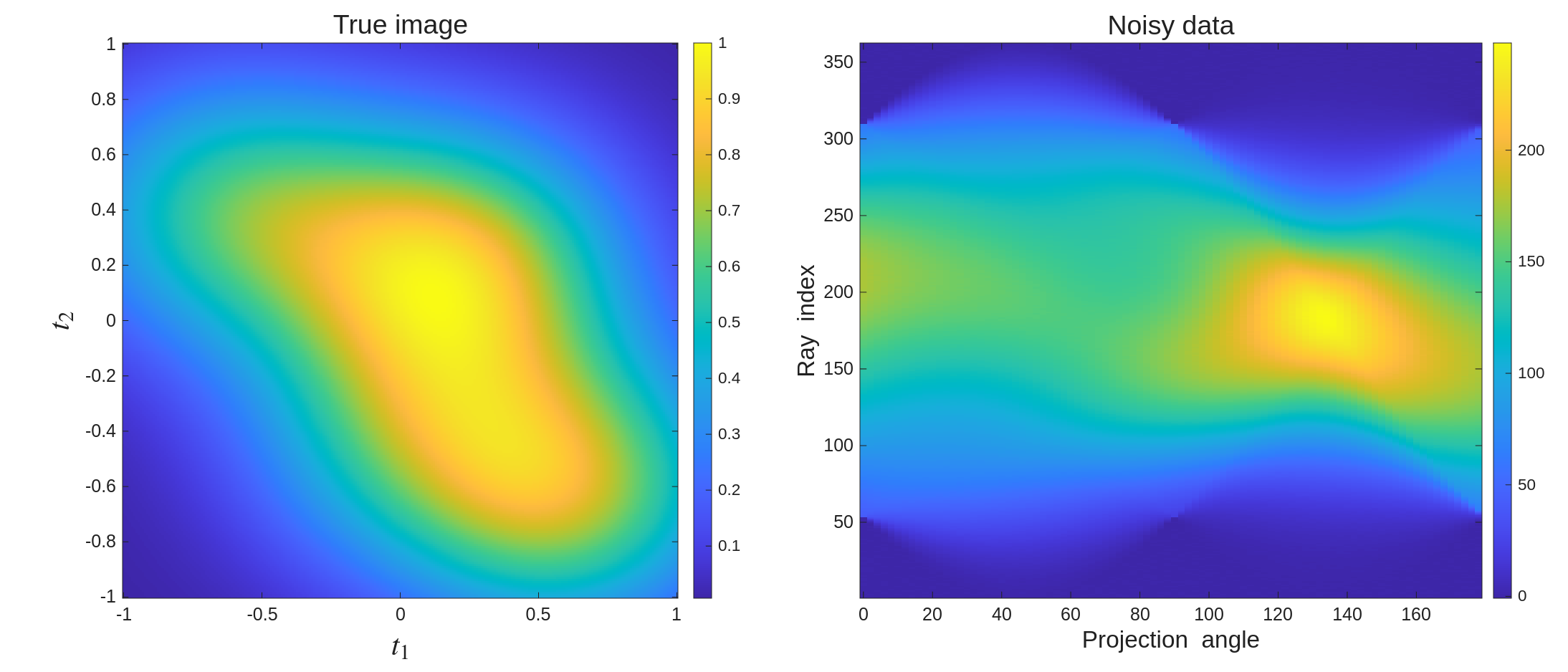}}\vspace{-2mm}
	\caption{Illustration of the true image and noisy observation data.}
	\label{fig9}
\end{figure}

To reconstructed the true image, we place a Gaussian prior on the unknown $f$ with covariance induced by the Mat\'ern kernel
\begin{equation}\label{eq:matern_kernel}
    K(\boldsymbol{x}, \boldsymbol{y})
    = \sigma^2\frac{2^{1-\nu}}{\Gamma(\nu)}
    \left(\frac{\sqrt{2\nu}\, \|\boldsymbol{x}-\boldsymbol{y}\|_{2}}{\rho}\right)^{\nu}
    B_{\nu}\!\left(\frac{\sqrt{2\nu}\, \|\boldsymbol{x}-\boldsymbol{y}\|_{2}}{\rho}\right),
\end{equation}
and we set $\nu=5/2$, $\rho=2$ and $\sigma=1$ to construct the matrix $\Sigma$. The value of $\sigma$ is iteratively updated by the QGKB-HB method via the relation $\sigma^2=1/\lambda$. To ensure computational tractability for the $256^2 \times 256^2$ covariance matrix $\Sigma$, we implement the matrix-vector multiplication $\Sigma v$ in a matrix-free way. Given the stationarity of the Matérn kernel, the covariance operator on a uniform grid can be computed as a circular convolution via the Fast Fourier Transform (FFT). To mitigate periodic wrap-around artifacts, we employ a circulant embedding strategy: the input $N \times N$ image is first augmented to size $2N \times 2N$ through zero-padding; then the covariance action is performed as point-wise multiplication in the spectral domain; the final result is recovered by extracting the $N \times N$ block corresponding to its original domain. This approach reduces the computational complexity of $\Sigma v$ from $O(N^4)$ to $O(N^2 \log N)$, enabling efficient reconstruction on high-resolution grids without explicitly forming the dense covariance matrix. For further implementation details, see, e.g., \cite{wood1994simulation, nowak2003efficient, lindgren2011explicit}

\begin{figure}[!t]
	\centering
	\subfloat 
	{\includegraphics[width=0.95\textwidth]{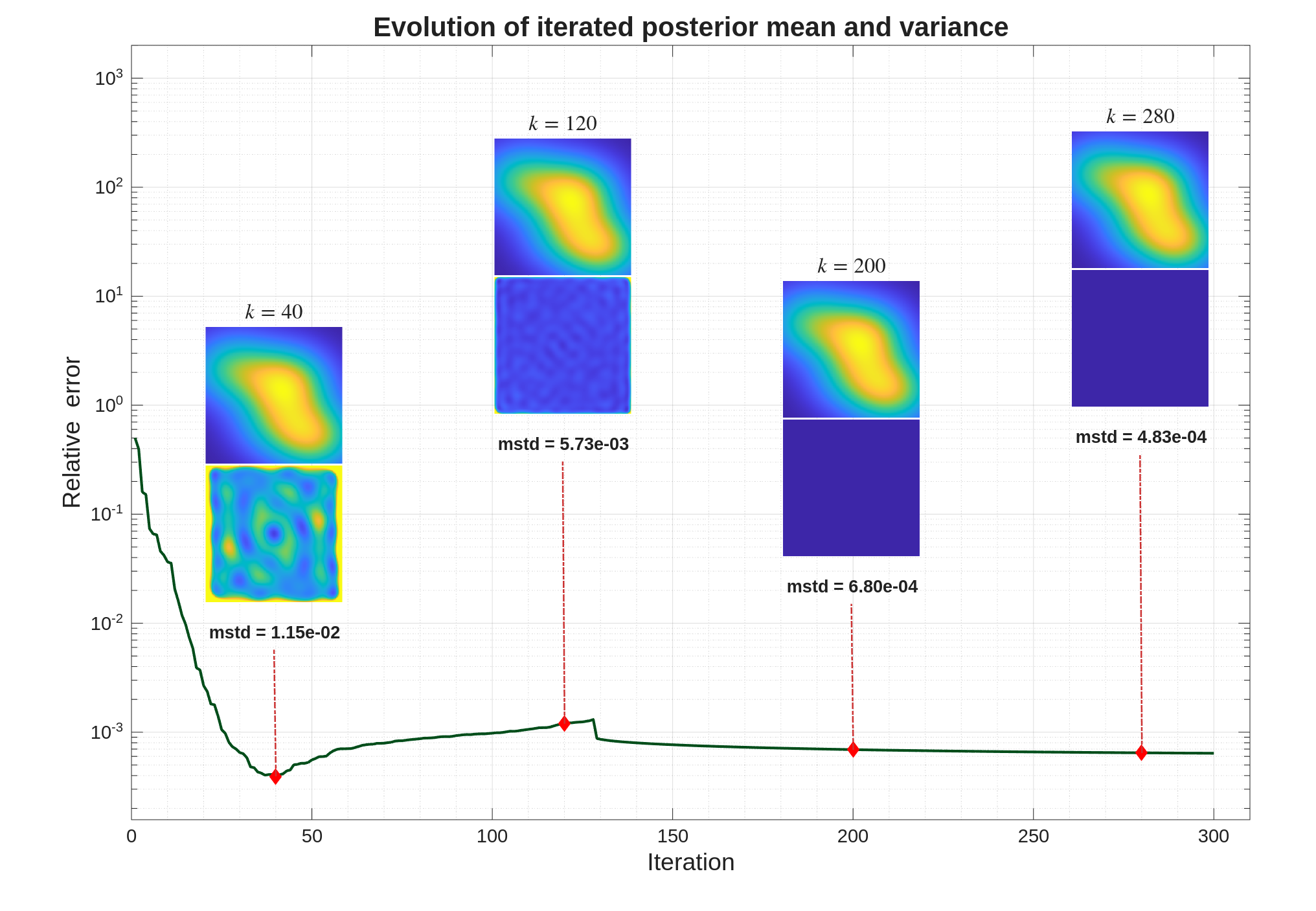}}\vspace{-3mm}
	\caption{Evolution of the approximate posterior mean and variance during the Q-GKB iteration. The top and bottom images show the approximate mean and the corresponding pixel-wise variance, respectively.}
	\label{fig11}
\end{figure}

\begin{figure}[!t]
	\centering
	\subfloat 
	{\includegraphics[width=1.0\textwidth]{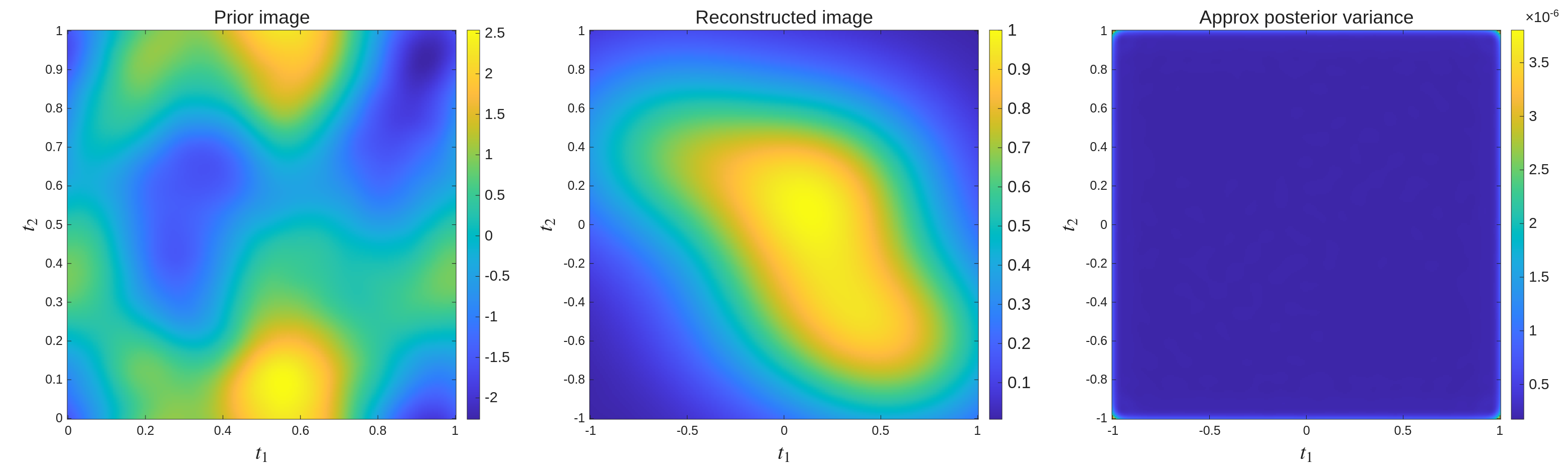}}\vspace{-2mm}
	\caption{Illustration of the prior image, reconstructed image, and corresponding pixel-wise variance at $k=300$.}
	\label{fig10}
\end{figure}

In this experiment, we assess the performance of the proposed Q-GKB posterior approximation in a genuinely large-scale setting. In particular, we emphasize its computational scalability, stability, and its ability to accurately capture both the posterior mean and uncertainty. In \Cref{fig11}, we illustrate the evolution of the approximate posterior mean and covariance as the Q-GKB iteration proceeds from $k=1$ to $k=300$. Representative reconstructions and variance fields (i.e., pixel-wise variance) are shown at $k=40$, $120$, $200$, and $280$, with a consistent color scale as in \Cref{fig10}. The corresponding mstd values are also computed following \cref{mstd}. From these results, we observe that the posterior mean converges rapidly: already at $k=40$, the reconstruction achieves high visual quality. In contrast, the posterior covariance requires a larger subspace dimension to stabilize. For small $k$, the covariance estimate remains relatively large, indicating that the dominant data-informed directions have not yet been fully captured. As $k$ increases, the covariance approximation gradually stabilizes, and the mstd decreases accordingly, indicating that the essential directions of the data-informed subspace are progressively captured. At $k=280$, both the posterior mean and covariance have essentially converged. This behavior is consistent with the theoretical characterization of data-informed Krylov subspaces and their role in posterior approximation.

The prior image, the reconstructed image, and the corresponding variance field are shown in \Cref{fig10}, where the variance level is at the order of $10^{-6}$, indicating very low uncertainty in the reconstruction. Importantly, for this large-scale CT problem, computing the exact posterior covariance, or even its diagonal, is computationally prohibitive under the FFT-based Mat\'ern prior. In contrast, the proposed Q-GKB framework enables efficient matrix-free computation of both the posterior mean and covariance approximation. These results demonstrate that the method remains numerically stable, scalable, and capable of capturing meaningful posterior uncertainty in high-dimensional settings.

\section{Conclusion}\label{sec6}
We have developed a data-informed framework for posterior approximation in large-scale Bayesian linear inverse problems. By shifting the perspective from the high-dimensional parameter space to the data space, we establish an isometric embedding from the data space to the parameter space, revealing an intrinsic low-dimensional geometric structure of the posterior update induced by the data. This structure naturally leads to a quotient-space formulation of the prior-to-posterior update, in which the informative directions are determined independently of the ambient parameter dimension. Building on this geometric insight, we propose a quotient-space Golub--Kahan bidiagonalization method to construct data-informed Krylov subspaces. By integrating empirical Bayesian inference into the iterative framework, the method enables simultaneous hyperparameter estimation and posterior approximation in a matrix-free manner. Numerical experiments on problems of increasing scale validate the proposed framework and support the theoretical findings.

The data-space geometric perspective offers a foundation for understanding low-dimensional structures in high-dimensional Bayesian inverse problems and provides a new path for posterior inference. It also opens several directions for future research. In particular, extending the quotient data space framework to inverse problems with infinite-dimensional parameter spaces and to nonlinear inverse problems, as well as integrating it with more general prior models and sampling-based inference methods, presents valuable directions.

\appendix 

\section{Proofs of technical lemmas}\label{apdx:A}
\begin{proofof}{\Cref{lem:lanczos_decay}}
Define $Q_k = \Sigma^{1/2}G^\top V_k$, $T_k = B_k^\top B_k$ and $P_k = Q_kQ_k^\top$. Then the columns of $Q_k$ are orthonormal and $P_k$ is the orthogonal projector onto $\mathcal{R}(Q_k)$, and we have
\begin{align}
  AQ_k &= \Sigma^{1/2}G^\top\Gamma^{-1}MV_k = \Sigma^{1/2}G^\top\Gamma^{-1}U_{k+1}B_k \nonumber \\
  &= \Sigma^{1/2}G^\top(V_kB_k^\top+\alpha_{k+1}v_{k+1}e_{k+1}^\top)B_{k} \nonumber \\
  &= Q_kT_k+\alpha_{k+1}\beta_{k+1}q_{k+1}e_k^\top,  \label{lanczos1}
\end{align}
where \(q_i=\Sigma^{1/2}G^\top v_i\). This implies that $Q_k^\top A Q_k=T_k$, and hence
\begin{equation}\label{A_k}
  \widehat A_k = \Sigma^{1/2}\widehat H_k\Sigma^{1/2}
  = \Sigma^{1/2}G^\top V_k B_k^\top B_k V_k^\top G\Sigma^{1/2}
  = Q_kT_kQ_k^\top = P_kAP_k.
\end{equation}
Since $\widehat A_k$ and $T_k$ share the same nonzero eigenvalues, the Cauchy interlacing theorem implies that $\mathrm{Tr}(A)\geq\mathrm{Tr}(T_k)=\mathrm{Tr}(\widehat{A}_k)$, leading to $\zeta_k>0$.
Using 
\[T_{k+1} = B_{k+1}^\top B_{k+1} =
\begin{pmatrix}
T_k & \alpha_{k+1}\beta_{k+1}e_k\\
\alpha_{k+1}\beta_{k+1}e_k^\top & \alpha_{k+1}^2+\beta_{k+2}^2
\end{pmatrix},\]
a direct computation gives
\begin{equation*}
  \widehat A_{k+1}-\widehat A_k = \alpha_{k+1}\beta_{k+1}(q_kq_{k+1}^\top+q_{k+1}q_k^\top) +(\alpha_{k+1}^2+\beta_{k+2}^2)q_{k+1}q_{k+1}^\top.
\end{equation*}
Therefore, we have
\begin{equation*}
  \mathrm{Tr}(\widehat A_{k+1}-\widehat A_k) = \alpha_{k+1}^2+\beta_{k+2}^2, \quad 
  \|\widehat A_{k+1}-\widehat A_k\|_F^2 = 2\alpha_{k+1}^2\beta_{k+1}^2+(\alpha_{k+1}^2+\beta_{k+2}^2)^2.
\end{equation*}
This leads to that 
\[\zeta_{k+1} = \mathrm{Tr}(A-\widehat A_{k+1}) = \zeta_k-\mathrm{Tr}(\widehat A_{k+1}-\widehat A_k) = \zeta_k-(\alpha_{k+1}^2+\beta_{k+2}^2),\]
where $\zeta_{0}=\mathrm{Tr}(A)=\mathrm{Tr}(H\Sigma)$.

For the recurrence of $\gamma_k$, notice that 
\begin{equation*}
  \gamma_k^2=\|A-\widehat A_k\|_F^2 = \|(I-P_k)A+P_kA(I-P_k)\|_{F}^2 = \|(I-P_k)A\|_{F}^2 + \|P_kA(I-P_k)\|_F^2,
\end{equation*}
since the two terms are orthogonal in the trace inner product. Using \cref{lanczos1}, we get 
\begin{equation*}
  P_kA(I-P_k) = Q_kQ_k^\top A(I-P_k) = \alpha_{k+1}\beta_{k+1}Q_ke_kq_{k+1}^\top = \alpha_{k+1}\beta_{k+1}q_kq_{k+1}^\top,
\end{equation*}
hence $\|P_kA(I-P_k)\|_F^2=\alpha_{k+1}^2\beta_{k+1}^2$. Using $I-P_k=(I-P_{k+1})+q_{k+1}q_{k+1}^\top$ and notice that the two terms are orthogonal in the trace inner product, we have 
\begin{equation*}
  \|(I-P_k)A\|_{F}^2 =  \|(I-P_{k+1})A\|_{F}^2 + \|q_{k+1}q_{k+1}^\top A\|_F^2
  = \|(I-P_{k+1})A\|_{F}^2 + \|Aq_{k+1}\|_2^2.
\end{equation*}
From the Q-GKB recurrence, we have 
\[\Gamma^{-1}Mv_{k+1} = \alpha_{k+1}\beta_{k+1}v_k + (\alpha_{k+1}^2+\beta_{k+2}^2)v_{k+1} + \alpha_{k+2}\beta_{k+2}v_{k+2}, \]
multiplying by $\Sigma^{1/2}G^\top$ gives
\[Aq_{k+1} = \alpha_{k+1}\beta_{k+1}q_k + (\alpha_{k+1}^2+\beta_{k+2}^2)q_{k+1} + \alpha_{k+2}\beta_{k+2}q_{k+2}. \]
This leads to $\|Aq_{k+1}\|_2^2=\alpha_{k+1}^2\beta_{k+1}^2+(\alpha_{k+1}^2+\beta_{k+2}^2)^2+\alpha_{k+2}^2\beta_{k+2}^2.$
Therefore, we obtain
\begin{align*}
  \gamma_{k}^2-\gamma_{k+1}^2 
  &= \|(I-P_k)A\|_{F}^2-\|(I-P_{k+1})A\|_{F}^2 + \alpha_{k+1}^2\beta_{k+1}^2 - \alpha_{k+2}^2\beta_{k+2}^2 \\
  &= \|Aq_{k+1}\|_2^2 + \alpha_{k+1}^2\beta_{k+1}^2 - \alpha_{k+2}^2\beta_{k+2}^2 \\
  &= (\alpha_{k+1}^2+\beta_{k+2}^2)^2 + 2\alpha_{k+1}^2\beta_{k+1}^2.
\end{align*}
A simple calculation leads to $\gamma_{0}^2=\mathrm{Tr}(A^2)=\mathrm{Tr}(H\Sigma H\Sigma)$. This completes the proof.
\end{proofof}

\begin{proofof}{\Cref{lem:inv_bnd}}
  Let $\delta=\|A_1-A_2\|_2$. Then $-\delta I \preceq A_1-A_2 \preceq \delta I$, or equivalently,   
  \[A_2-\delta I \preceq A_1 \preceq A_2+\delta I, \quad
  A_1-\delta I \preceq A_2 \preceq A_1+\delta I.\]
  Let $R_i=(I+A_i)^{-1}$ for $i=1,2$. Since $f(t)=(1+t)^{-1}$ is operator monotone decreasing on $[0,\infty)$, using $A_1 \preceq A_2+\delta I$ and $A_2 \preceq A_1+\delta I$ we obtain
  \[R_1 \succeq (I+A_2+\delta I)^{-1}, \quad
  R_2 \succeq (I+A_1+\delta I)^{-1}.\]
  Hence $R_1-R_2 \preceq R_1-(I+A_1+\delta I)^{-1}.$ Noticing that
  \[R_1-(I+A_1+\delta I)^{-1} = \delta (I+A_1)^{-1}(I+A_1+\delta I)^{-1},\]
  and using $0\prec (I+A_1)^{-1}\preceq I$ and $0\prec (I+A_1+\delta I)^{-1}\preceq \frac{1}{1+\delta}I$, we get $R_1-R_2 \preceq \frac{\delta}{1+\delta}I$. By symmetry, exchanging $A_1$ and $A_2$ yields $R_2-R_1 \preceq \frac{\delta}{1+\delta}I$. Therefore, we get 
  \begin{equation*}
    -\frac{\delta}{1+\delta}I \preceq R_1-R_2 \preceq \frac{\delta}{1+\delta}I \quad
    \Rightarrow \quad \|R_1-R_2\|_2\le \frac{\delta}{1+\delta} .
  \end{equation*}
  This completes the proof.
\end{proofof}

\section*{Acknowledgments}
The author thanks Prof. Fei Lu for many insightful discussions.

\bibliographystyle{plainnat} 
\bibliography{references}

\end{document}